\RequirePackage{fix-cm}
\documentclass[smallextended]{svjour3}       
\smartqed  
\usepackage{graphicx}
\usepackage{amsfonts}
\usepackage{amsmath}
\usepackage{amssymb}
\usepackage{array}
\usepackage[ruled]{algorithm}
\usepackage{algpseudocode}
\usepackage{empheq}
\usepackage{setspace}
\usepackage{xcolor}

\usepackage{multirow}
\usepackage{hyperref}

\newcommand{\A}{{\mathcal A}}
\newcommand{\F}{{\mathcal F}}
\newcommand{\G}{{\mathcal G}}
\newcommand{\J}{{\mathcal J}}
\newcommand{\X}{{\mathcal X}}
\newcommand{\Y}{{\mathcal Y}}
\newcommand{\V}{{\mathcal V}}
\newcommand{\M}{{\mathcal M}}
\newcommand{\I}{{\mathcal I}}
\newcommand{\U}{{\mathcal U}}
\newcommand{\Z}{{\mathcal Z}}
\newcommand{\dist}{{\mbox{dist}}}

\begin{document}

\title{Support matrix machine: exploring sample sparsity, low rank, and adaptive sieving in high-performance computing
}

\titlerunning{Support matrix machine: exploring sample sparsity...}        

\author{Can Wu \and Dong-Hui Li \and Defeng Sun
}

\authorrunning{C. Wu, D. Li, D. Sun} 

\institute{Can Wu \at
	          School of Mathematics and Statistics, Hainan University,
	          Haikou, 570228, China. {This work was conducted during her Postdoctoral Fellowship at the Department of Applied Mathematics, The Hong Kong Polytechnic University.}\\
              \email{wucan-opt@hainanu.edu.cn}           
           \and
           Dong-Hui Li \at
            School of Mathematical Sciences, South China Normal University, Guangzhou, 510631, China\\
            \email{lidonghui@m.scnu.edu.cn}
            \and
            Defeng Sun \at
            Department of Applied Mathematics, The Hong Kong Polytechnic University, Hung Hom, Hong Kong \\
            \email{defeng.sun@polyu.edu.hk}
}

\date{Received: date / Accepted: date}

\maketitle

\begin{abstract}
Support matrix machine (SMM) is a successful supervised classification model for matrix-type samples. Unlike support vector machines, it employs low-rank regularization on the regression matrix to effectively capture the intrinsic structure embedded in each input matrix.	
When solving a large-scale SMM, a major challenge arises from the potential increase in sample size, leading to substantial computational and storage burdens.
To address these issues, we design a semismooth Newton-CG (SNCG) based augmented Lagrangian method (ALM) for solving the SMM.
The ALM exhibits an asymptotic R-superlinear convergence if a strict complementarity condition is satisfied.
The SNCG method is employed to solve the ALM subproblems, achieving at least a superlinear convergence rate under the nonemptiness of an index set.
Furthermore, the sparsity of samples and the low-rank nature of solutions enable us to reduce the computational cost and storage demands for the Newton linear systems.
Additionally, we develop an adaptive sieving strategy that generates a solution path for the SMM by exploiting sample sparsity. The finite convergence of this strategy is also demonstrated.
Numerical experiments on both large-scale real and synthetic datasets validate the effectiveness of the proposed methods.
\end{abstract}

\section{Introduction}
\label{intro}
Numerous well-known classification methods, such as linear discriminant analysis, logistic regression, support vector machines (SVMs), and AdaBoost, were initially designed for vector or scalar inputs \cite{hastie2009elements}. On the other hand, matrix-structured data, like digital images with pixel grids \cite{yang2004two} and EEG signals with multi-channel voltage readings over time \cite{sanei2013eeg}, are also prevalent in practical applications.
Classifying matrix data often involves flattening it into a vector, but this causes several issues \cite{qian2019robust}: a) High-dimensional vectors increase dimensionality problems, b) Loss of matrix structure and correlations, and c) Structural differences require different regularization methods.
To circumvent these challenges, the support matrix machine (SMM) was introduced by Luo, Xie, Zhang, and Li \cite{luo2015support}.
Given a set of training samples $\{X_i,y_i\}^n_{i=1}$, where $X_i\in\mathbb{R}^{p\times q}$  represents the $i$th
feature matrix and $y_i\in\{-1,+1\}$ is its corresponding class label,
the optimization problem of the SMM is formulated as
\begin{equation}\label{Model:SMM}
	\mathop{\mbox{minimize}}\limits_{(W,b)\in\mathbb{R}^{p\times q}\times\mathbb{R}}\ \frac{1}{2}\|W\|^2_{\F}+\tau\|W\|_*
	+C\sum^n_{i=1}\max\left\{ 1-y_i\left[ \mbox{tr}(W^{\top}X_i)+b \right],0 \right\},
\end{equation}
where $W\in\mathbb{R}^{p\times q}$ is the matrix of regression coefficients, $b\in\mathbb{R}$ is an offset term,
$C$  and $\tau$ are positive regularization parameters. Here, $\|\cdot\|_{\F}$ and $\|\cdot\|_*$ denote the Frobenius norm
and nuclear norm of a matrix, respectively.
The objective function in (\ref{Model:SMM}) combines two key components:  a) the spectral elastic net penalty $(1/2)\|\cdot\|^2_\F+\tau\|\cdot\|_*$, enjoying grouping effect\footnote{At a solution pair
	$(\overline{W}, \overline{b})$ of the model (\ref{Model:SMM}), the columns of the regression matrix
	$\overline{W}$ exhibit a grouping effect when the associated features display strong correlation \cite{luo2015support}.}
and low-rank structures; b) the hinge
loss, ensuring sparsity and robustness.
This model has been incorporated as a key component in deep stacked networks \cite{hang2020deep,liang2022deep} and semi-supervised learning frameworks \cite{li2023intelligent}.
Additionally, the SMM model (\ref{Model:SMM}) and its variants are predominantly applied in fault diagnosis (e.g., \cite{pan2019fault,li2020non,li2022highly,li2022fusion,pan2022pinball,pan2023deep,geng2023fault,li2024dynamics,xu2024intelligent,pan2024semi,li2024intelligent}) and EEG classification (e.g., \cite{zheng2018robust,zheng2018multiclass,razzak2019multiclass,hang2020deep,chen2020novel,liang2022deep,hang2023deep,liang2024adaptive}).
For a comprehensive overview of SMM applications and extensions, see the recent survey \cite{kumari2024support}.

The two-block alternating direction methods of multipliers (ADMMs) are the most widely adopted methods for solving the convex SMM model (\ref{Model:SMM}) and its extensions, including least squares SMMs  \cite{li2020non,li2022highly,liang2024adaptive,hang2023deep}, multi-class SMMs \cite{zheng2018multiclass,pan2019fault}, weighted SMMs \cite{li2022fusion,li2024intelligent,pan2024semi}, transfer SMMs \cite{chen2020novel,pan2022pinball}, and pinball SMMs \cite{feng2022support,pan2022pinball,pan2023deep,geng2023fault}.
Compared to classical SVMs,
the primary challenge in solving a large-scale SMM model (\ref{Model:SMM}) and the above extensions arises from the additional nuclear norm term. Specifically, their dual problems are no longer a quadratic program (QP) due to the presence of the sphere constraint in the sense of the spectral norm, which precludes the use of the off-the-shelf QP solvers like LIBQP\footnote{The codes are available at \url{https://cmp.felk.cvut.cz/~xfrancv/pages/libqp.html}.}.
To overcome this drawback, Luo et al. \cite{luo2015support}  suggested to
utilize the two-block fast ADMM with restart (F-ADMM), proposed in \cite{goldstein2014fast}. Indeed, it is reasonable to choose two-block ADMM algorithms over three-block ADMM \cite{chen2017efficient} because the former consistently requires fewer iterations, resulting in lower singular value decomposition costs from the soft thresholding operator.
However, F-ADMM solves the dual subproblems by calling the LIBQP solver, which requires producing at least one $n \times n$ matrix. This can cause a memory burden as the sample size
$n$ increases significantly.
As a remedy, Duan et al. \cite{duan2017quantum}
introduced a quantum algorithm that employs the quantum matrix inversion algorithm \cite{harrow2009quantum} and  a quantum singular
value thresholding algorithm to update the subproblems's variables in the F-ADMM framework, respectively.
Nevertheless, no numerical experiments were reported in \cite{duan2017quantum}, leaving the practical effectiveness
of the proposed algorithm undetermined. Efficiently solving the SMM models under large-scale samples
remains a significant challenge.

The first purpose of this paper is to devise an efficient and robust algorithm for solving the SMM model (\ref{Model:SMM})
when $n$ is significantly large (e.g., several hundred thousand or up to a million) and $pq$ is large (e.g., ten of thousands).
By leveraging the sample sparsity and the low-rank property of the solution to model (\ref{Model:SMM}), we develop an efficient and robust semismooth Newton-CG based augmented Lagrangian method. Notably, these two properties depend on the values of the parameters $C$ and $\tau$.
Here, sample sparsity indicates that the solution of (\ref{Model:SMM}) relies solely on a subset of samples (feature matrices) known as active samples (support matrices), while disregarding the remaining portion referred to as non-active samples (non-support matrices). At each iteration, a conjugate gradient (CG) algorithm is employed to solve the Newton linear system. By making full use of the sample sparsity and the solution's low-rank property in (\ref{Model:SMM}), the primary computational cost
in each iteration of the CG method can be reduced to $O(pq\max\{|\mathcal{J}_1|,|\alpha|\})$, instead of $O(npq)$. Here, the index set $\mathcal{J}_1$ ultimately corresponds  the active support matrices. Its cardinality $|\mathcal{J}_1|$ is typically much smaller than the total sample size $n$. And the cardinality of the index set $\alpha$ eventually corresponds to the rank of the solution matrix $W$ in the model (\ref{Model:SMM}).
It is worth mentioning that the proposed computational framework can be readily extended to convex variants of the SMM model (\ref{Model:SMM}), such as multi-class SMMs \cite{zheng2018multiclass,pan2019fault}, weighted SMMs \cite{li2022fusion,li2024intelligent,pan2024semi}, transfer SMMs \cite{chen2020novel,pan2022pinball}, and pinball SMMs \cite{feng2022support,pan2022pinball,pan2023deep,geng2023fault}.


Our second goal in this paper is to develop a strategy that efficiently guesses and adjusts irrelevant samples before the starting of the aforementioned optimization process. This method is particularly useful for generating a solution path for models (\ref{Model:SMM}).
Recently, Ghaoui et al. \cite{el2012safe} introduced a safe feature elimination method for $\ell_1$-regularized convex problems.
It inspired researchers to extend these ideas to conventional SVMs and develop safe sample screening rules \cite{ogawa2013safe,wang2014scaling,zimmert2015safe,pan2019novel}.
Those rules aim to first bound the primal SVM solution within a valid region and then exclude non-active samples by relaxing the Karush-Kuhn-Tucker conditions, which reduces the problem size and saves computational costs and storage.
However, existing safe screening rules are typically problem-specific and cannot be directly applied to the SMM model (\ref{Model:SMM}).
The major challenges include two parts: a) the objective function in (\ref{Model:SMM}) is not strongly convex due to the intercept term $b$ \cite{zimmert2015safe}; b) the nuclear norm term complicates the removal of non-support matrices using first-order optimality conditions for the dual SVM model \cite{wang2014scaling}.	
Recently, Yuan, Lin, Sun, and Toh \cite{yuan2023adaptive,lin2020adaptive} introduced a highly efficient and
flexible adaptive sieving (AS) strategy for convex composite sparse machine learning models.
By exploiting the inherent sparsity of the solutions,
it significantly reduces computational load through solving a finite number of reduced subproblems on a smaller scale \cite{yuan2022dimension,li2023mars,wu2023convex}.
The effectiveness of this AS strategy crucially hinges on solution sparsity, which dictates the size of these reduced subproblems.
Unlike the sparsity emerges at solutions,
the SMM model (\ref{Model:SMM}) inherently exhibits sample sparsity.
As a result, the application of the AS strategy  to (\ref{Model:SMM}) is not straightforward.
We will extend the idea of the AS strategy in \cite{yuan2023adaptive,lin2020adaptive} to generate a solution path for the SMM model (\ref{Model:SMM}).

Our main contributions in this paper are outlined as follows:
\begin{itemize}
\item[1)]
To solve the SMM (\ref{Model:SMM}), we propose an augmented Lagrangian method (ALM), which guarantees asymptotic R-superlinear convergence of the KKT residuals under a mild strict complementarity condition, commonly satisfied in the classical soft margin SVM model. For solving each ALM subproblem, we use the semismooth Newton-CG method (SNCG), ensuring superlinear convergence when an index set is nonempty, with its cardinality ultimately corresponding to the number of active support matrices.
\item[2)]
The main computational bottleneck in solving semismooth Newton linear systems with the conjugate gradient method is the execution of the generalized Jacobian linear transformation. By leveraging the sample sparsity and solution's low-rank structure in model (\ref{Model:SMM}), we can reduce the cost of this transformation from $O(pq\max\{|\J_1|,|\alpha|\})$ to $O(npq)$, where the cardinalities of $\J_1$ and $\alpha$ are typically much smaller than the sample size $n$.
Numerical experiments demonstrate that ALM-SNCG achieves an average speedup of 422.7 times over F-ADMM on four real datasets, even for low-accuracy solutions of model (\ref{Model:SMM}).
\item[3)]
To efficiently generate a solution path for the SMM models with a huge sample size, we employ an AS strategy. This strategy iteratively identifies and removes inactive samples to obtain reduced subproblems. 
	Solving these subproblems inexactly ensures convergence to the desired solution path within a finite number of iterations (see Theorem \ref{Thm:AS_convergence}).
Numerical experiments reveal that, for generating high-accuracy solution paths of models (\ref{Model:SMM}), the AS strategy paired with ALM-SNCG is, on average, 2.53 times faster than the warm-started ALM-SNCG on synthetic dataset with one million samples.
\end{itemize}

The rest of the paper is organized as follows. Section \ref{Sec: structural_properties} discusses the sample sparsity and the solution's low-rank structure of the SMM model (\ref{Model:SMM}). Building on these properties, 
Section \ref{Sec:ALM-SNCG} introduces the computational framework of a semismooth Newton-CG based augmented Lagrangian method for solving this model. 
Section \ref{Sec:linear_system} explores how these properties reduce computational costs in solving the Newton linear system.
Section \ref{Section:AS} introduces an adaptive sieving strategy for computing a solution path of the SMM models across a fine grid of $C$'s. Finally, Section \ref{Sec:Experiments} presents extensive numerical experiments on synthetic and real data to demonstrate the effectiveness of the proposed methods.

\vspace{4mm}

{\bf\noindent Notation.} Let $\mathbb{R}^{m \times n}$ be the space of all $m \times n$ real matrices and $\mathbb{S}^n$ be the space of all $n \times n$ real symmetric matrices.
The notation  $0_{m\times n}$ stands for the zero matrix in $\mathbb{R}^{m\times n}$, and $I_n\in\mathbb{R}^{n\times n}$ is the identity matrix. We use $\mathcal{O}^n$ to denote the set of all $n\times n$ orthogonal matrices. For $X, Y\in\mathbb{R}^{m\times n}$, their inner product is defined by
$\langle X,Y \rangle=\mbox{tr}(X^{\top}Y)$,
where ``$\mbox{tr}$" represents the trace operation. The Frobenius norm of $X\in\mathbb{R}^{m\times n}$ is $\|X\|_{\F}=\sqrt{\langle X,X \rangle}$. The $n$-dimensional vector space $\mathbb{R}^{n\times 1}$ is abbreviated as $\mathbb{R}^n$. The vector $e_n\in\mathbb{R}^n$ is the vector whose elements are all ones.
We denote the index set $[m]:=\{1,2,\ldots,m\}$.
For any index subsets $\alpha\subseteq[m]$ and $\beta\subseteq[n]$, and any $X\in\mathbb{R}^{m\times n}$, let $X_{\alpha\beta}$ denote the $|\alpha|\times|\beta|$ submatrix of $X$ obtained by removing all the rows of $X$ not in $\alpha$ and all the columns of $X$ not in $\beta$.
For any $x\in\mathbb{R}^n$, we write $x_{\alpha}$ as the subvector of $x$ obtained by removing all the elements of $x$ not in $\alpha$, and $\mbox{Diag}(x)$ as the diagonal matrix with diagonal entries $x_i$, $i=1,\ldots,n$. The notation $\|\cdot\|_2$ represents the matrix spectral norm, defined as the largest singular value of the matrix.
For any $\tau>0$, we define $\mathbb{B}^{\tau}_2:=\{ X\in\mathbb{R}^{m\times n}\ |\ \|X\|_2\leq\tau \}$. The metric projection of $X$ onto $\mathbb{B}^{\tau}_2$ is denoted by $\Pi_{\mathbb{B}^{\tau}_2}(X)$, with $\partial\Pi_{\mathbb{B}^{\tau}_2}(X)$ representing its Clarke generalized Jacobian at $X$. Similarly, for any closed convex set $K\subset\mathbb{R}^n$, $\Pi_K(\omega)$ and $\partial\Pi_K(\omega)$ denote the metric projection of $\omega\in\mathbb{R}^n$ onto $K$ and its Clarke generalized Jacobian at $\omega$, respectively.

\section{The structural properties of the SMM model (\ref{Model:SMM})}\label{Sec: structural_properties}
This section highlights the sample sparsity and low-rank properties at the Karush-Kuhn-Tucker (KKT) points of the model (\ref{Model:SMM}), providing the foundation for developing an efficient algorithm.


For notational simplicity, we first reformulate the SMM model (\ref{Model:SMM}). Denote
\[
\mathcal{X}:=\mathbb{R}^{p\times q}\times\mathbb{R}\times\mathbb{R}^n\times\mathbb{R}^{p\times q}\quad \mbox{and}\quad \mathcal{Y}:=\mathbb{R}^n\times\mathbb{R}^{p\times q}.
\]
Define the linear operator $\A:\mathbb{R}^{p\times q}\rightarrow\mathbb{R}^n$ and its corresponding adjoint $\A^*:\mathbb{R}^n\rightarrow\mathbb{R}^{p\times q}$ as follows:
\begin{equation}\label{def:A_AT}
	\A W=\left( \langle y_1X_1,W \rangle,\ldots,\langle y_nX_n,W \rangle \right)^{\top}, \
	\A^*z=\sum^n_{k=1}z_ky_kX_k, \ \forall\ (z, W)\in\Y.
\end{equation}
Let $y=(y_1,y_2\ldots,y_n)^{\top}$, $v=e_n-\A W-by$, and $S=[0, C]^n$. The support function of $S$ is given by $\delta^*_S(v)=C\sum^n_{i=1}\max\{v_i,0\}$ for any $v\in\mathbb{R}^n$.
The SMM model (\ref{Model:SMM}) can be equivalently expressed as
\begin{equation}\label{Model:SMM_eq2}
	\begin{array}{cl}
		\mathop{\mbox{minimize}}\limits_{(W,b,v,U)\in\mathcal{X}} &
		\displaystyle\frac{1}{2}\|W\|^2_{\F}+\tau\|U\|_{*}+\delta^*_{S}(v)\\
		\mbox{subject to} & \A W+by+v=e_n,\\
		{} & W-U=0_{p\times q}. \tag{P}
	\end{array}
\end{equation}
Its Lagrangian dual problem can be written as
\begin{equation}\label{Model:dual_SMM}
	\begin{array}{cl}
		\mathop{\mbox{maximize}}\limits_{(\lambda,\Lambda)\in\mathcal{Y}} & -\Phi(\lambda,\Lambda):=\displaystyle\frac{1}{2}\|\A^*\lambda+\Lambda\|^2_{\F}+\langle \lambda,e_n \rangle+\delta_{S\times\mathbb{B}^\tau_2}(-\lambda,\Lambda)\\
		\mbox{subject to} & \qquad \qquad \quad\ \ y^{\top}\lambda=0.\\
	\end{array}\tag{$D$}
\end{equation}
It is easy to get the following KKT system of (\ref{Model:SMM_eq2}) and (\ref{Model:dual_SMM}):
\begin{equation}\label{eq:KKT_system_SMM}
	\left\{
	\begin{array}{l}
		W+\A^*\lambda+\Lambda=0,\ \lambda^{\top}y=0,\ 0\in\partial\delta^*_S(v)+\lambda,\\[2mm]
		0\in\partial\,\tau\|U\|_*-\Lambda,\ \A W+by+v=e_n,\ W-U=0_{p\times q}.
	\end{array}
	\right.
\end{equation}
Without loss of generality, suppose that there exist indices $i_0$, $j_0\in[n]$ such that $y_{i_0}=+1$ and $y_{j_0}=-1$ (otherwise, the classification task cannot be executed).
Clearly, the solution set of the KKT system (\ref{eq:KKT_system_SMM}) is always nonempty. Indeed,
the objective function in problem (\ref{Model:SMM}) is a real-valued,  lower level-bounded convex function.
It follows from \cite[Theorem 1.9]{rockafellar2009variational} that the solution set of (\ref{Model:SMM}) is nonempty, convex, and compact,
with a finite optimal value. Since the equivalent problem (\ref{Model:SMM_eq2}) of (\ref{Model:SMM}) contains linear constraints only, by \cite[Corollary 28.3.1]{rockafellar1970convex}, the KKT system (\ref{eq:KKT_system_SMM}) has at least one solution.
It is known from \cite[Theorems 28.3, 30.4, and Corollary 30.5.1]{rockafellar1970convex}
that  $(\overline{W},\overline{b},\overline{v},\overline{U}, \overline{\lambda}, \overline{\Lambda})$ is a solution to (\ref{eq:KKT_system_SMM}) if and only if $(\overline{W},\overline{b},\overline{v},\overline{U})$ is an optimal solution to (\ref{Model:SMM_eq2}),
and $(\overline{\lambda},\overline{\Lambda})$ is an optimal solution to (\ref{Model:dual_SMM}).

\subsection{Characterizing sample sparsity}\label{subsec:sample_sparsity}
For the SMM model (\ref{Model:SMM_eq2}), a sufficiently large sample size 
$n$ introduces significant computational and storage challenges. However, the inherent sample sparsity offers a promising avenue for developing efficient algorithms.

Suppose that $(\overline{W},\overline{b},\overline{v},\overline{U},\overline{\lambda},\overline{\Lambda})$ is a solution of the KKT system (\ref{eq:KKT_system_SMM}). It follows that 
$$
\overline{W}=-\A^*\overline{\lambda}-\overline{\Lambda}=\sum^n_{j=1}(-\overline{\lambda})_jy_jX_j-\overline{\Lambda}.
$$
This implies that the value of $\overline{W}$ depends solely on the samples associated with the non-zero elements of $-\overline{\lambda}$, a property referred to as sample sparsity.
In addition, by $-\overline{\lambda}\in\partial\delta^*_S(\overline{v})$ in (\ref{eq:KKT_system_SMM}), we deduce that for each $j\in[n]$,
\begin{equation}\label{eq:lam_in_partial_delta}
(-\overline{\lambda})_j\in
\left\{\eta\in\mathbb{R}\ \left|\ \begin{array}{ll}
	\eta=C & \mbox{if}\ \overline{v}_j>0\\
	\eta\in[0,C] & \mbox{if}\ \overline{v}_j=0\\
	\eta=0 & \mbox{otherwise}
\end{array}\right.
\right\}.
\end{equation}
It implies that 
$-\overline{\lambda}\in[0,C]^n$. Analogous to the concept of support vectors \cite{cortes1995support}, we introduce the definition of support matrices.

\begin{definition}[Support matrix]
An example $X_j$ that satisfies the inequality 
$0<(-\overline{\lambda})_j\leq C$ is said a support matrix at $(\overline{W},\overline{b})$; otherwise, it is a non-support matrix. The set of all support matrices at $(\overline{W},\overline{b})$ is denoted as
$$
SM:=\{X_j\ |\ 0<(-\overline{\lambda})_j\leq C, j\in[n]\}.
$$
\end{definition}

To illustrate the geometric interpretation of support matrices, we define the optimal hyperplane $H(\overline{W},\overline{b})$ and the decision boundaries $H_+(\overline{W},\overline{b})$ and $H_{-}(\overline{W},\overline{b})$ associated with $(\overline{W},\overline{b})$ as follows:
$$
\begin{array}{c}
	H(\overline{W},\overline{b}):=\{X\ |\ \langle \overline{W},X \rangle+\overline{b}=0\},\\
	H_+(\overline{W},\overline{b}):=\{X\ |\ \langle \overline{W},X \rangle+\overline{b}=1\}, \
	H_{-}(\overline{W},\overline{b}):=\{X\ |\ \langle \overline{W},X \rangle+\overline{b}=-1\}.
\end{array}
$$

Based on (\ref{eq:lam_in_partial_delta}) and $(\overline{v})_j=1-y_j(\langle \overline{W},X_j \rangle+\overline{b})$, we obtain that
\begin{equation}\label{eq:inclusion_SM}
SM\subseteq\{X_j\ |\ \overline{v}_j\geq0, j\in[n]\}=\{X_j\ |\ y_j(\langle \overline{W},X_j \rangle+\overline{b})\leq 1, j\in[n]\}.
\end{equation}
Thus, geometrically, support matrices include a subset of examples lying on the decision boundaries, within the region between them, or misclassified.

It can also be observed from (\ref{eq:lam_in_partial_delta}) that
$$
\{X_j\ |\ y_j(\langle \overline{W},X_j \rangle+\overline{b})<1, j\in[n]\}=\{X_j\ |\ \overline{v}_j>0, j\in[n]\}\subseteq\{X_j\ |\ (-\overline{\lambda})_j=C\}.
$$
This implies that, after excluding support matrices 
$X_j$ at $(\overline{W},\overline{b})$ where 
$(-\overline{\lambda})_j=C$, only those remaining on the decision boundaries are retained, i.e.,
\begin{align*}
	ASM:=\{X_j\ |\ (-\overline{\lambda})_j\in(0,C), j\in[n]\}
	\subseteq\{X_j\ |\ y_j(\langle \overline{W},X_j \rangle+\overline{b})=1, j\in[n]\}.
\end{align*}
Here, we define support matrices satisfying 
$0<(-\overline{\lambda})_j<C$ as active support matrices. Notably, the number of support matrices far exceeds that of active support matrices. As shown in Figure \ref{fig:SM_ASM}, the proportion of support matrices relative to the total sample size ranges from 9\% to 86\%, whereas the proportion of active support matrices is significantly lower, ranging from 0\% to 3\%.
This raises a natural question: {\bf Can we design an algorithm whose main computational cost depends solely on the samples corresponding to the active support matrices?}

\begin{figure}
	\centering \includegraphics[width=0.9\textwidth]{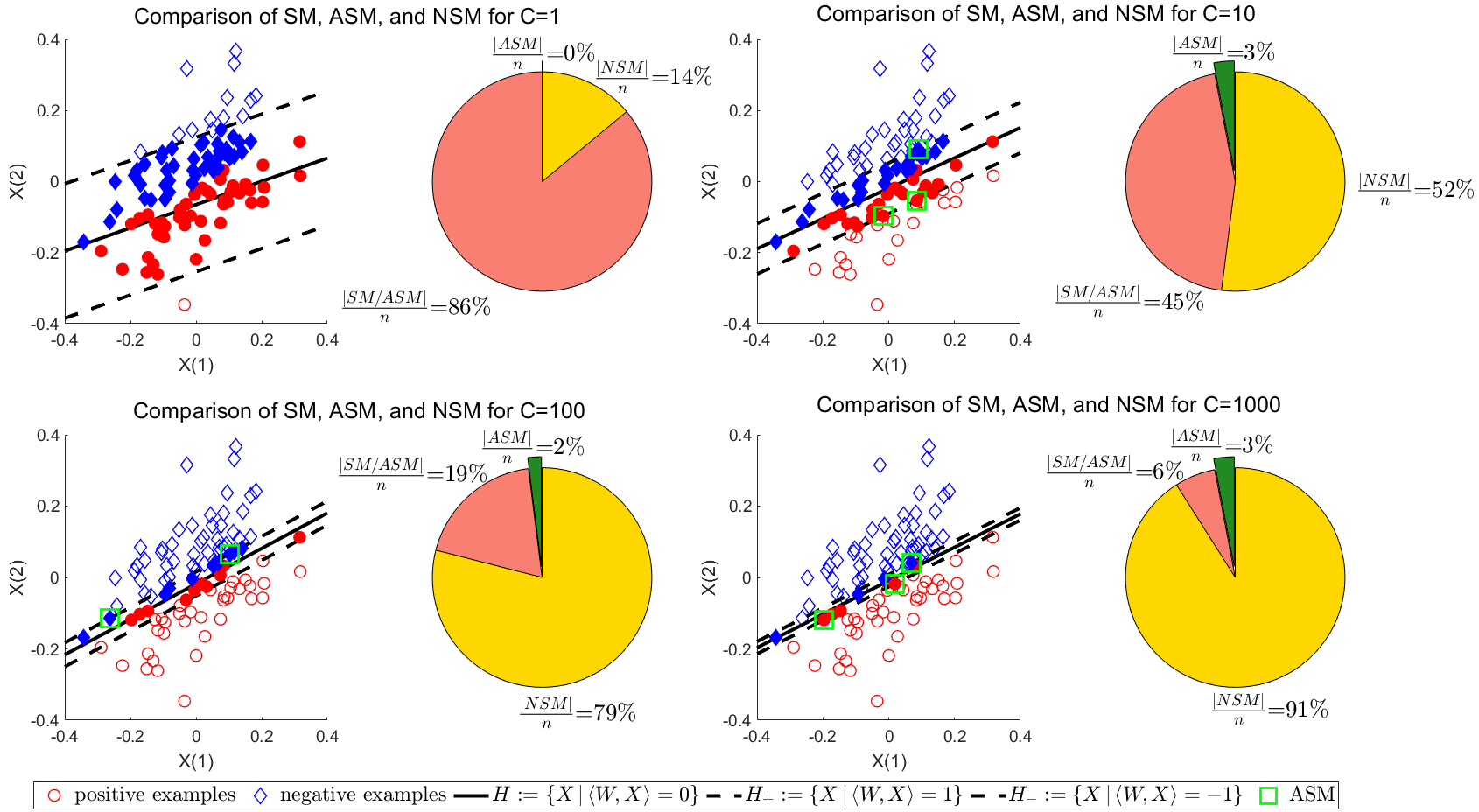}
	\caption{Comparison of SM, ASM, and NSM under ($n, p, q,\tau$)=(100, 2, 1, 0.1) (Here, NSM means the set of non-support matrix. Red circles represent positive examples, blue diamonds represent negative examples, and solid circles/diamonds indicate support matrices. Active support matrices are highlighted with a green square border. SM/ASM refers to the set of support matrices after excluding the active support matrices.
	We observe that (a) the cardinality of active support matrices is significantly smaller than that of support matrices; and (b) as the value of $C$ increases, both the number of support matrices and the optimal margin  $2/\|\overline{W}\|_{\F}$ decrease.)}
	\label{fig:SM_ASM}
\end{figure}

\subsection{Characterizing low-rank property of $\overline{W}$}\label{subsec:low-rank}

In addition to the sample sparsity of model (\ref{Model:SMM_eq2}), another key property is the low-rank structure of the solution $\overline{W}$. Indeed, since $(\overline{W},\overline{b},\overline{v},\overline{U},\overline{\lambda},\overline{\Lambda})$  is a solution to the KKT system (\ref{eq:KKT_system_SMM}),  one has that
$$
\overline{W}=\overline{U}=\mbox{Prox}_{\tau\|\cdot\|_*}(\overline{\Lambda}+\overline{U})=\mbox{Prox}_{\tau\|\cdot\|_*}(-\A^*\overline{\lambda}).
$$
Without loss of generality, we assume that $p\leq q$. And suppose $-\A^*\overline{\lambda}$ has the following singular value decomposition (SVD):
$$
-\A^*\overline{\lambda}=U[\mbox{Diag}(v)\ 0]V^{\top},
$$
where $U\in\mathcal{O}^p$, $V\in\mathcal{O}^q$, and $v:=(v_1,v_2,\ldots,v_p)\in\mathbb{R}^p$ with $v_1\geq v_2\geq\ldots\geq v_p$.
Furthermore, define 
$$
\bar{k}:=\max\{i\in[p]\ |\ v_i>\tau\}.
$$
By applying the proximal mapping $\mbox{Prox}_{\tau\|\cdot\|_*}(-\A^*\overline{\lambda})$ , as described in \cite{liu2012implementable}, we obtain
\begin{align*}
\overline{W}=&\mbox{Prox}_{\tau\|\cdot\|_*}(-\A^*\overline{\lambda})=U[\mbox{Diag}(\max\{v_1-\tau,0\},\ldots,\max\{v_p-\tau,0\})\ 0]V^{\top}\\
&=U[\mbox{ Diag}(v_1-\tau,\ldots,v_{\bar{k}}-\tau,0,\ldots,0)\ 0]V^{\top}.
\end{align*}
This implies that the rank of  $\overline{W}$ is $\bar{k}$. As illustrated in Figure \ref{fig:rank_W}, for sufficiently large $\tau$, $\bar{k}$ is non-increasing, potentially enhancing classification accuracy on the test set (denoted $\mbox{Accuracy}_{\mbox{test}}$ in (\ref{eq:accu_test})). Thus, the low-rank property of $\overline{W}$ may correspond to improved performance on the test set. Given this, {\bf how can an efficient algorithm be designed to leverage the low-rank structure of 
$\overline{W}$ when $\bar{k}$ is sufficiently small?}

\begin{figure}
	\centering \includegraphics[width=0.9\textwidth]{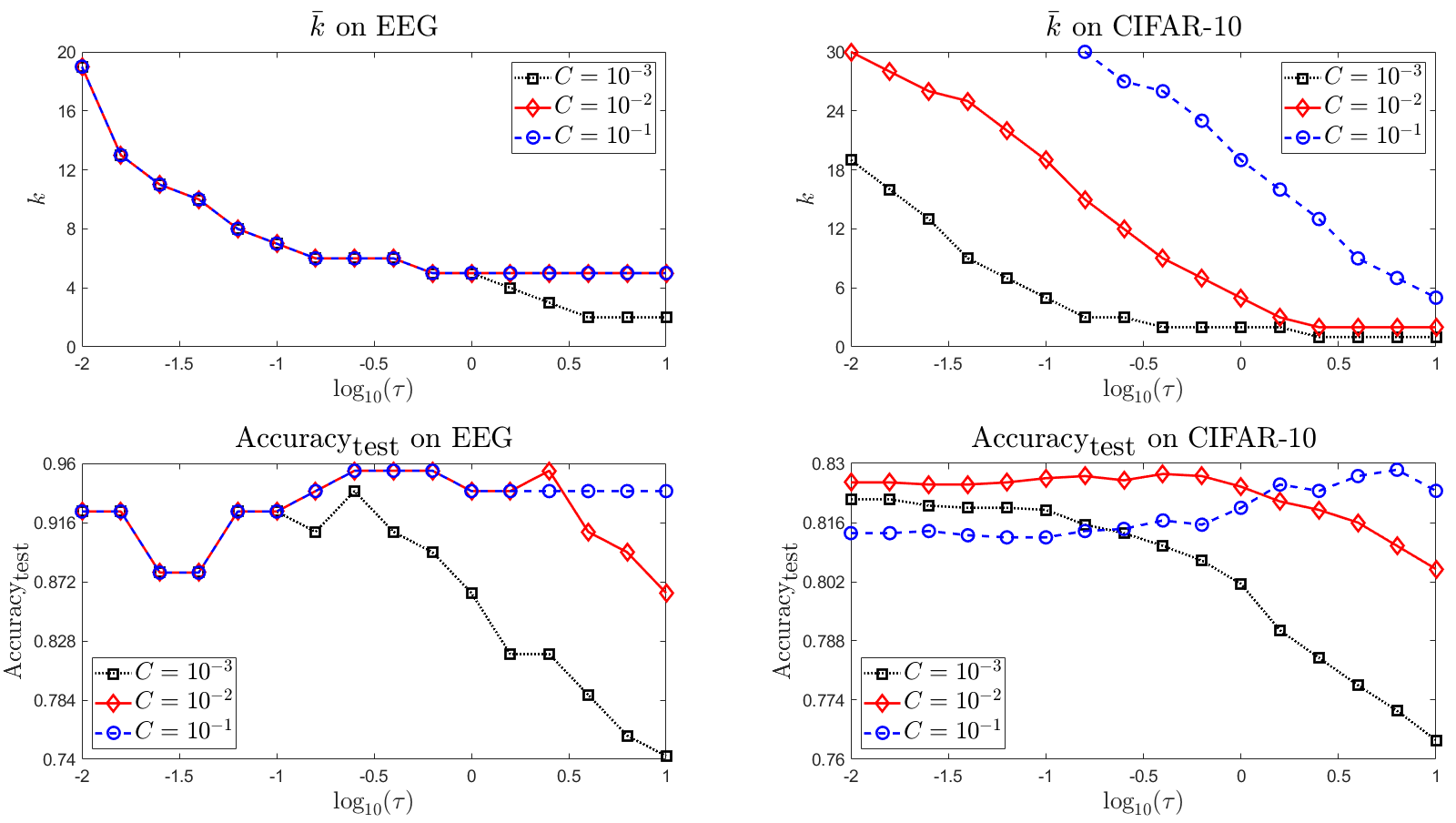}
	\caption{The changes of values of $\bar{k}$ and $\mbox{Accuracy}_{\mbox{test}}$ as the value of $\tau$ increases on EEG and CIFAR-10 real datasets}
	\label{fig:rank_W}
\end{figure}

\section{A semismooth Newton-CG based augmented Lagrangian method for the SMM model }\label{Sec:ALM-SNCG}
In this section, we present our algorithmic framework for solving the SMM model (\ref{Model:SMM}), which comprises an outer loop using the inexact augmented Lagrangian method (ALM) and inner loops employing the semismooth Newton-CG method for ALM subproblems.

\subsection{An inexact augmented Lagrangian method}\label{subsub:3.1}
We are going to develop an inexact augmented Lagrangian method for solving the problem (\ref{Model:SMM_eq2}).
Given a positive penalty parameter $\sigma$, the augmented Lagrangian function
$L_\sigma:\mathcal{X}\times\mathcal{Y}\rightarrow\mathbb{R}$ for (\ref{Model:SMM_eq2}) takes the form
\begin{equation}\label{eq:augmented Lagrangin funciton}
	\begin{aligned}
	L_{\sigma}(x;z)=&\frac{1}{2}\|W\|^2_\F+\tau\|U\|_{*}+\delta^*_{S}(v)+\langle \lambda,\A W+by+v-e_n \rangle+\langle \Lambda,W-U \rangle\\
	&+\frac{\sigma}{2}\|\A W+by+v-e_n\|^2+\frac{\sigma}{2}\|W-U\|^2_{\F},
	\end{aligned}
\end{equation}
where  $x=(W, b, v, U)\in \mathcal{X}$ and $z=(\lambda, \Lambda)\in \mathcal{Y}$.
Following the general framework in \cite{rockafellar1976augmented}, the steps of the ALM method is summarized in Algorithm \ref{alg_ALM_SMM} below.

\begin{algorithm}[htbp]
	\caption{An inexact augmented Lagrangian method}
	\vskip 0.05in
	\noindent
	{\bf Initialization:} Given are two positive parameters $\sigma_0$ and $\sigma_\infty$. Choose an initial point $z^0=(\lambda^0,\Lambda^0)\in\mathcal{Y}$. Set $k=0$. Perform the following steps in each iteration until a proper stopping criterion is satisfied:	
	\vskip 0.1in
	\begin{algorithmic}[0]
		\State
		{\bf Step 1}. Compute an approximate solution
		\begin{equation}\label{Model:ALM_SMM}
			(W^{k+1},b^{k+1},v^{k+1},U^{k+1})\approx\mbox{argmin}\left\{ f_k(x):=L_{\sigma_k}(x;z^k)\ |\ x\in\mathcal{X} \right\}.
		\end{equation}
		\State
		{\bf Step 2}. Update the multipliers
		$$
		(\lambda^{k+1},\Lambda^{k+1})=(\lambda^k+\sigma_k(\A W^{k+1}+b^{k+1}y+v^{k+1}-e_n),\Lambda^k+\sigma_k(W^{k+1}-U^{k+1})).
		$$
		\State
		{\bf Step 3}. Update $\sigma_{k+1}\uparrow\sigma_{\infty}\leq\infty$.
	\end{algorithmic}
	\label{alg_ALM_SMM}
\end{algorithm}

In Step 1 of Algorithm \ref{alg_ALM_SMM}, the subproblem (\ref{Model:ALM_SMM}) is solved inexactly.
Different inexact rules can be adopted. Here, we give two easily implementable inexact rules.
Observe that  $(\widehat{W},\widehat{b},\widehat{v},\widehat{U})$ is an optimal solution to the ALM subproblem (\ref{Model:ALM_SMM})
if and only if it satisfies
\begin{empheq}[left=\empheqlbrace]{align}
	&(\widehat{W},\widehat{b})\in\mbox{argmin}\left\{ \varphi_k(W,b)\ |\ (W,b)\in\mathbb{R}^{p\times q}\times\mathbb{R} \right\},\label{subprob:ALM_primal2}\\[2mm]
	&(\widehat{v}, \widehat{U})=\left( \sigma^{-1}_k\mbox{Prox}_{\delta^*_S}( \omega_k(\widehat{W},\widehat{b}) ), \sigma^{-1}_k\mbox{Prox}_{\tau\|\cdot\|_*}(X_k(\widehat{W})) \right),\label{subprob:ALM_primal_v_U}
\end{empheq}
where the function $\varphi_k: \mathbb{R}^{p\times q}\times\mathbb{R}\to \mathbb R$ is defined by
\begin{equation}\label{def:varphi}
	\begin{aligned}
	\varphi_k(W,b):=&(1/2)\|W\|^2_{\F}+(\sigma_k)^{-1}E_{\delta^*_{S}}\left( \omega_k(W,b) \right)-(2\sigma_k)^{-1}\|\lambda^k\|^2\\
	&+(\sigma_k)^{-1}E_{\tau\|\cdot\|_{*}}(X_k(W))-(2\sigma_k)^{-1}\|\Lambda^k\|^2_{\F}.
	\end{aligned}
\end{equation}
Here, $\omega_k:\mathbb{R}^{p\times q}\times\mathbb{R}\rightarrow\mathbb{R}^n$ and $X_k:\mathbb{R}^{p\times q}\rightarrow\mathbb{R}^{p\times q}$ are defined as
\begin{equation}\label{def:wk_X_k}
	\omega_k(W,b):=-\lambda^k-\sigma_k(\A W+b\,y-e_n)\quad \mbox{and}\quad X_k(W):=\Lambda^k+\sigma_k W.
\end{equation}
Additionally, $E_g(\cdot)$ and $\mbox{Prox}_g(\cdot)$ represent the Moreau envelop and proximal mapping of some proper closed convex function $g$, respectively.
Similar to those in \cite{cui2019r}, we adopt the following two easily implementable inexact rules  for (\ref{subprob:ALM_primal2}):
\begin{align*}
	(A)\quad & \|\nabla\varphi_k(W^{k+1},b^{k+1}) \|\leq\frac{\varepsilon^{\,2}_k/\sigma_k}{1+\|x^{k+1}\|+\|z^{k+1}\|}\min\left\{ 1/\chi_k,1 \right\},\\[2mm]
	(B)\quad & \|\nabla\varphi_k(W^{k+1},b^{k+1}) \|\leq\frac{(\eta^{\,2}_k/\sigma_k)\|z^{k+1}-z^k\|^2}{1+\|x^{k+1}\|+\|z^{k+1}\|}\min\left\{ 1/\chi_k,1 \right\}
\end{align*}
where  $\{\varepsilon_k\}$ and $\{\eta_k\}$ are two given positive summable sequences, $x^k=(W^k, b^k$, $v^k, U^k)$, $z^k=(\lambda^k, \Lambda^k)$, and $\chi_k=\|W^{k+1}\|_{\F}+\|z^{k+1}-z^k\|/\sigma_k+1/\sigma_k$.

Rockafellar's original work \cite{rockafellar1976augmented} demonstrated the asymptotic $Q$-superlinear convergence rate of the dual sequence produced by the ALM, assuming Lipschitz continuity of the dual solution mapping at the origin. However, this condition is challenging to satisfy for (\ref{Model:SMM_eq2}) due to the requirement of dual solution uniqueness. Therefore, we consider a weaker quadratic growth condition on the dual problem (\ref{Model:dual_SMM}) in this paper. The quadratic growth condition for (\ref{Model:dual_SMM}) at $\overline{z}:=(\overline{\lambda},\overline{\Lambda})\in\Omega_D$ is said to hold if there exist constants $\varepsilon>0$ and $\kappa>0$ such that
\begin{equation}\label{eq:qg condition_D}
	\Phi(z)\geq\Phi(\overline{z})+\kappa\,\mbox{dist}^2(z,\Omega_{D}), \quad \forall\ z:=(\lambda, \Lambda)\in F_D \cap \mathbb{B}_{\varepsilon}(\overline{z}).
\end{equation}
Here, the function $\Phi$ is defined in (\ref{Model:dual_SMM}),  $\Omega_D$ denotes the set of all optimal solutions to (\ref{Model:dual_SMM}), and
$
F_D:=\{(\lambda,\Lambda)\in\mathcal{Y}\ |\ y^{\top}\lambda=0, -\lambda\in S, \Lambda\in\mathbb{B}^{\tau}_2\}
$
represents the set of all feasible solutions to (\ref{Model:dual_SMM}). 

Denote $R^{kkt}:\mathcal{X}\times\mathcal{Y}\rightarrow\mathcal{X}\times\mathcal{Y}$ be the natural residual mapping as follows:
$$
R^{kkt}(x,z):=\left[
\begin{array}{c}
	W+\A^*\lambda+\Lambda\\
	y^{\top}\lambda\\
	v-\mbox{Prox}_{\delta^*_S}(v-\lambda)\\
	U-\mbox{Prox}_{\tau\|\cdot\|_*}(U+\Lambda)\\
	\A W+by+v-e_n\\
	W-U	
\end{array}
\right].
$$
The following theorem establishes the global convergence and asymptotic $R$-superlinear convergence of the KKT residuals in terms of $\|R^{kkt}(x^k,z^k)\|$ under criteria (A) and (B) for Algorithm \ref{alg_ALM_SMM}, which can be directly derived from \cite[Theorem 2]{cui2019r}.

\begin{theorem}\label{Thm:ALM}
	Let $\{(x^k,z^k)\}$ be an infinite sequence generated by Algorithm \ref{alg_ALM_SMM} under criterion ($A$)
	with $x^k:=(W^k, b^k, v^k, U^k)$ and $z^k:=(\lambda^k, \Lambda^k)$. Then the sequence $\{z^k\}$ converges to some $\overline{z}\in\Omega_D$.
	Moreover,  the sequence $\{x^k\}$ is  bounded with all of its accumulation points in the solution set of  (\ref{Model:SMM_eq2}).
	
	If in addition, criterion (B) is executed and the quadratic growth condition (\ref{eq:qg condition_D})
	holds at $\overline{z}$, then there exists $k'\geq0$ such that for all $k\geq k'$, $\beta\eta_k<1$ and
	\begin{subequations}
		\begin{align*}
			\mbox{dist}\,(z^{k+1},\Omega_D)\leq\theta_k\mbox{dist}\,(z^k,\Omega_{D}), \quad
			\|R^{kkt}(x^{k+1},z^{k+1})\|\leq\theta'_k\mbox{dist}\,(z^k,\Omega_D),
		\end{align*}
	\end{subequations}
	where 
	\begin{align*}
		\theta_k:=&\left[ \beta\eta_k+(\beta\eta_k+1)/\sqrt{1+\sigma^2_k\kappa^2} \right]/(1-\beta\eta_k)\rightarrow\theta_{\infty}:=1/\sqrt{1+\sigma^2_{\infty}\kappa^2},\\[2mm]
		\theta'_k:=&\left[ 1/\sigma_k+(\eta^2_k/\sigma_k)\|z^{k+1}-z^k\| \right]/(1-\beta\eta_k)\rightarrow\theta'_{\infty}:=1/\sigma_{\infty},\\[2mm]
		\beta:=&\sqrt{2\left[ 1+\tau\sqrt{p}+C\sqrt{n}+\gamma(\sqrt{n}+1)+2\gamma^2 \right]}\quad \mbox{with some}\quad \gamma\geq 1.
	\end{align*}
	Moreover, $\theta_{\infty}=\theta'_{\infty}=0$ if $\sigma_{\infty}=+\infty$.
\end{theorem}

The following proposition gives a sufficient condition for the quadratic growth condition (\ref{eq:qg condition_D})
with respect to the dual problem (\ref{Model:dual_SMM}). Here, we assume without loss of generality that $0<p\leq q$. Detailed proof is  given in Appendix \ref{Appendix:Prof_Prop_sufficient_condition}.

\begin{proposition}\label{Prop:sufficient_condition_quad_growth_cond}
	Let $\overline{z}:=(\overline{\lambda},\overline{\Lambda})\in\Omega_D$. Assume that there exists $(\widetilde{\lambda},\widetilde{\Lambda})\in\Omega_{D}$ such that
	\begin{equation}\label{eq:sufficient_condition_quatratic_growth}
		\mbox{rank}\,(-\overline{\Upsilon})+\mbox{rank}\,(\tau\overline{I}-\widetilde{\Lambda})=p\ \mbox{with}\ \overline{I}:=\overline{U}[I_p\ \,0_{p\times(q-p)}]\overline{V}^{\top},
	\end{equation}
	where $\overline{\Upsilon}:=\A^*\overline{\lambda}+\overline{\Lambda}$, and the SVD of $\widetilde{\Lambda}$ is
	\begin{equation}\label{eq:svd_Lambda_simple} \widetilde{\Lambda}=\overline{U}[\overline{\Sigma}(\widetilde{\Lambda})\ 0]\overline{V}^{\top}\ \mbox{with}\ \overline{U}\in\mathcal{O}^p\ \mbox{and}\ \overline{V}\in\mathcal{O}^q.
	\end{equation}
	Then the quadratic growth condition (\ref{eq:qg condition_D}) at $\overline{z}$ holds for the dual problem (\ref{Model:dual_SMM}).
\end{proposition}

Finally, condition (\ref{eq:sufficient_condition_quatratic_growth}) can be interpreted as a strict complementarity condition.
Indeed, if $(\overline{W}, \overline{b}, \overline{v}, \overline{U}, \overline{\lambda}, \overline{\Lambda})$ satisfies
the KKT system (\ref{eq:KKT_system_SMM}), then $\overline{W}=-(\A^*\overline{\lambda}+\overline{\Lambda})=-\overline{\Upsilon}$.
Consequently, the condition (\ref{eq:sufficient_condition_quatratic_growth}) means that there exists $(\widetilde{\lambda},\widetilde{\Lambda})\in\Omega_{D}$ such that
\begin{equation}\label{eq:rank(W)_rank(tauI-Lambda)_p}
	\mbox{rank}\,(\overline{W})+\mbox{rank}\,(\tau\overline{I}-\widetilde{\Lambda})=p.
\end{equation}
Moreover, by the use of  \cite[Proposition 10]{zhou2017unified} and  (\ref{eq:tauI_tildeLambda}), we can deduce
$\langle \overline{W},\tau\overline{I}-\widetilde{\Lambda} \rangle=0$. In this sense, equality
(\ref{eq:rank(W)_rank(tauI-Lambda)_p}) is regarded as a strict complementarity between the matrices $\overline{W}$
and $\tau\overline{I}-\widetilde{\Lambda}$. In particular,
(\ref{eq:rank(W)_rank(tauI-Lambda)_p}) retains true if $\mbox{rank}\,(\overline{W})=p$ and $\tau=0$, 
a condition inherently satisfied in the context of the soft margin support vector machine model \cite{cortes1995support}.

\subsection{A semismooth Newton-CG method to solve the subproblem (\ref{subprob:ALM_primal2})}
In Algorithm \ref{alg_ALM_SMM}, the primary computational cost arises from solving the convex subproblem (\ref{subprob:ALM_primal2}).
In this subsection, we propose a semismooth Newton-CG method to calculate an inexact solution to this subproblem.


Due to the convexity of   $\varphi_k(\cdot,\cdot)$,
solving subproblem (\ref{subprob:ALM_primal2}) is equivalent to solving the following nonlinear equation:
\begin{equation}\label{eq:gradphi}
	\nabla\varphi_k(W,b)=\left[ \begin{array}{c}
		W-\A^*\Pi_{S}\left( \omega_k(W,b) \right)+\Pi_{\mathbb{B}^{\tau}_2}(X_k(W))\\[2mm]
		-y^{\top}\Pi_{S}\left( \omega_k(W,b) \right)
	\end{array} \right]=0,
\end{equation}
where $\omega_k(W,b)$ and $X_k(W)$ are defined in (\ref{def:wk_X_k}). Notice that $\nabla\varphi_k(\cdot,\cdot)$ is not smooth but strongly semismooth (see, e.g., \cite[Proposition 7.4.4]{FP2007} and \cite[Theorem 2.3]{kaifeng2011algorithms}).
 It is then desirable to solve using the semismooth Newton method.
The semismooth Newton method has been extensively studied. Under some regularity conditions,
the method can be superlinearly/quadratically convergent, see, e.g., \cite{qi1993nonsmooth,qi2006quadratically,ZST2010}.

In what follows,
we construct an alternative set for the Clarke generalized Jacobian $\partial^2\varphi_k(W,b)$ of $\nabla\varphi_k$ at $(W,b)$, which is more computationally tractable. Define for $(W, b)\in\mathbb{R}^{p\times q}\times\mathbb{R}$,
$$
\widehat{\partial}^{\,2}\varphi_k(W,b):=\left[
\begin{array}{ll}
	\I +\sigma_k\partial\Pi_{\mathbb{B}^\tau_2}(X_k(W)) & {}\\
	{} & 0
\end{array}
\right]+\sigma_k\left[
\begin{array}{l}
	\A^*\\
	y^{\top}
\end{array}
\right]\partial\Pi_S(\omega_k(W,b))[\A\quad y].
$$
It follows from \cite[Proposition 2.3.3 and Theorem 2.6.6]{clarke1990optimization} that
$$
{\partial}^{\,2}\varphi_k(W,b)(d_W,d_y)\subseteq\widehat{\partial}^{\,2}\varphi_k(W,b)(d_W,d_y), \quad \forall\ (d_W,d_y)\in\mathbb{R}^{p\times q}\times\mathbb{R}.
$$
The element $\V\in \widehat{\partial}^{\,2}\varphi_k(W,b)$ takes the following form
\begin{equation}\label{eq:Hession_phi}
	\V=\left[
	\begin{array}{ll}
		\I+\sigma_k\G & {}\\
		{} & 0
	\end{array}
	\right]+\sigma_k\left[
	\begin{array}{l}
		\A^*\\
		y^{\top}
	\end{array}
	\right]M[\A \quad y],
\end{equation}
where $\G\in\partial\Pi_{\mathbb{B}^{\tau}_2}(X_k(W))$,  $M\in\partial\Pi_S(\omega_k(W, b))$, and $\I$ is the identity operator from $\mathbb{R}^{p\times q}$ to $\mathbb{R}^{p\times q}$.

The steps of the  semismooth Newton-CG method for solving (\ref{subprob:ALM_primal2}) are stated in Algorithm \ref{alg_SSNCG}.
\begin{algorithm}[htbp]
	\caption{A semismooth Newton-CG method for subproblem (\ref{subprob:ALM_primal2}) }
	\vskip 0.05in
	\noindent
	{\bf Initialization:} Choose positive scalars $\mu\in(0,1/2)$, $\bar{\eta}, \tau_1, \tau_2, \delta\in(0,1)$, $\varrho\in(0,1]$ and an initial point $(W^{k,0}, b^{k,0} )\in\mathbb{R}^{p\times q}\times\mathbb{R}$. Set $i=0$. Execute the following steps until the stopping criteria ($A$) and/or ($B$) at $(W^{k,i+1},b^{k,i+1})$ are satisfied.	
	\vskip 0.1in
	\begin{algorithmic}[0]
		\State {\bf Step 1 (finding the semismooth Newton direction)}. Choose $\G\in\partial\Pi_{\mathbb{B}^\tau_2}(X_k(W^{k,i}))$ and $M\in\partial\Pi_S(\omega_k(W^{k,i},b^{k,i}))$. Let $\V_i:=\V$ be given as in (\ref{eq:Hession_phi})
		with $W=W^{k,i}$, $b=b^{k,i}$, and $\rho_i=\tau_1\min\{\tau_2, \|\nabla\varphi_k(W^{k,i},b^{k,i})\| \}$. Apply the conjugate gradient (CG) algorithm to find an approximate solution $(d^{\,i}_W, d^{\,i}_b)\in\mathbb{R}^{p\times q}\times\mathbb{R}$ of the following linear equation
		\begin{equation}\label{eq:Newton_system}
			\V_i(d_W, d_b)+\rho_i(0,d_b)=-\nabla\varphi_k(W^{k,i},b^{k,i})
		\end{equation}
		such that
		$$
		\|\V_i(d^{\,i}_W,d^{\,i}_b)+\rho_i(0,d^{\,i}_b)+\nabla\varphi_k(W^{k,i},b^{k,i})\|\leq\min(\bar{\eta}, \|\nabla\varphi_k(W^{k,i},b^{k,i})\|^{1+\varrho}).
		$$
		\State {\bf Step 2 (line search)}. Set $\alpha_i=\delta^{m_i}$, where $m_i$ is the first nonnegative integer $m$ for which
		$$
		\varphi_k(W^{k,i}+\delta^m d^{\,i}_W, b^{k,i}+\delta^m d^{\,i}_b)\leq\varphi_k(W^{k,i},b^{k,i}) +\mu\,\delta^m\left\langle \nabla\varphi_k(W^{k,i},b^{k,i}),(d^{\,i}_W,d^{\,i}_b) \right\rangle.	
		$$
		\State {\bf Step 3}. Set $W^{k,i+1}=W^{k,i}+\alpha_i d^{\,i}_W$, $b^{k,i+1}=b^{k,i}+\alpha_i d^{\,i}_b$, and $i\leftarrow i+1$.
	\end{algorithmic}
	\label{alg_SSNCG}
\end{algorithm}


In our semismooth Newton method, the positive definiteness of $\V_i$ in the Newton linear system (\ref{eq:Newton_system}) is crucial for ensuring the superlinear convergence of the method, as discussed in \cite{ZST2010,li2018qsdpnal}.
In what follows, we show that the positive definiteness of $\V_i$ is equivalent to
the constraint nondegenerate condition from \cite{shapiro2003sensitivity}, which simplifies to the nonemptiness of a specific index set.
Here, the constraint nondegenerate condition associated with an optimal solution $(\widehat{\lambda},\widehat{\Lambda})$ of the dual problem for (\ref{Model:ALM_SMM})
can be expressed as (see \cite{shapiro2003sensitivity} )
\begin{equation}\label{eq:constraint_nondegenerate}
	y^{\top}\mbox{lin}\left(T_S(-\widehat{\lambda})\right)=\mathbb{R}.
\end{equation}

\begin{proposition}\label{Thm:constraint_nondegenerate}
	Let $(\widehat{W},\widehat{b})$ be an optimal solution for subproblem (\ref{subprob:ALM_primal2}). Denote $\widehat{\lambda}:=-\Pi_S( \omega_k(\widehat{W},\widehat{b}))$ and $\J_1(-\widehat{\lambda}):=\{j\in[n]\ |\ 0<(-\widehat{\lambda})_{j}<C\}$,
	where $\omega_k(\widehat{W},\widehat{b})$ is defined in (\ref{def:wk_X_k}). Then the following conditions are equivalent:\\
	(i) The constraint nondegenerate condition (\ref{eq:constraint_nondegenerate}) holds at $\widehat{\lambda}$;\\
	(ii) The index set $\J_1(-\widehat{\lambda})$ is nonempty;\\
	(iii) Every element in $\widehat{\partial}^{\,2}\varphi_k(\widehat{W},\widehat{b})$ is self-adjoint and positive definite.
\end{proposition}
\begin{proof}
	See Appendix \ref{Appendix:Prof_constraint_nondegenrate}. \qed
\end{proof}


We conclude this section by giving the convergence theorem of Algorithm \ref{alg_SSNCG}.
It can be proved in a way similar to those in \cite[Proposition 3.1]{qi1993nonsmooth} and \cite[Theorem 3.5]{ZST2010}.
\begin{theorem}\label{Thm:SNCG_convergence}
	The sequence $\{(W^{k,i},b^{k,i})\}_{i\geq0}$  generated by Algorithm \ref{alg_SSNCG} is bounded.
	Moreover, every accumulation point is an optimal solution to  (\ref{subprob:ALM_primal2}).
	If at some accumulation point $(\overline{W}^{k+1},\overline{b}^{k+1})$, the constraint nondegeneracy condition (\ref{eq:constraint_nondegenerate}) holds with
	$\overline{\lambda}^{k+1}:=-\Pi_{S}(\omega_k(\overline{W}^{k+1},\overline{b}^{k+1}))$. Then the whole sequence $\{(W^{k,i},b^{k,i})\}$ converges to $(\overline{W}^{k+1},\overline{b}^{k+1})$ and
	$$
	\|(W^{k,i+1},b^{k,i+1})-(\overline{W}^{k+1},\overline{b}^{k+1})\|=O\left(\|(W^{k,i},b^{k,i})-(\overline{W}^{k+1},\overline{b}^{k+1})\|^{1+\varrho}\right).
	$$
\end{theorem}

\section{Implementation of the semismooth Newton-CG method}\label{Sec:linear_system}

As previous mentioned, our focus is on situations where the sample size $n$ greatly exceeds the maximum of $p$ and $q$, particularly when $n\gg pq$. As a result, the main computational burden in each iteration of Algorithm \ref{alg_SSNCG} is solving the following Newton linear system:
\begin{equation}\label{eq:Newton_simple1}
	\V(d_W,d_b)+\rho(0,d_b)=(\mbox{Rhs1},\mbox{rhs2}), \quad (d_W,d_b)\in\mathbb{R}^{p\times q}\times\mathbb{R},
\end{equation}
where  $\V$ is defined in (\ref{eq:Hession_phi}) with $\mathcal{G}\in \partial\Pi_{\mathbb{B}^{\tau}_2}(X_k(\widetilde{W}))$, $M\in \partial\Pi_S(\omega_k(\widetilde{W}, \widetilde{b}))$, $\sigma:=\sigma_k$, $\widetilde{W}:=W^{k,i}$, and $\widetilde{b}:=b^{k,i}$. The right-hand sides are
$$
\mbox{Rhs1}:=-[ \widetilde{W}-\A^*\Pi_S(\omega_k(\widetilde{W},\widetilde{b}))+\Pi_{\mathbb{B}^{\tau}_2}(X_k(\widetilde{W}))]\ \mbox{and}\ \mbox{rhs2}:=y^{\top}\Pi_S(\omega_k(\widetilde{W},\widetilde{b})).
$$
It can be further rewritten as
\begin{equation}\label{eq:Newton_simple2}
	\left\{
	\begin{array}{l}
		\widetilde{\V}d_W=\mbox{Rhs1}-(\sigma\cdot\mbox{rhs2})(\sigma y^{\top}My+\rho)^{-1}\A^*My,\\[2mm]
		d_b=(\sigma y^{\top}My+\rho)^{-1}(\mbox{rhs2}-\sigma y^{\top}M\A\, d_W),
	\end{array}
	\right.
\end{equation}
where 
$$\widetilde{\V}:=\I+\sigma\G+\mathcal{H}
$$ with
\begin{equation}\label{eq:def_H}
\mathcal{H}:=\sigma\A^*M\A-\sigma^2(\sigma y^{\top}My+\rho)^{-1}\A^*Myy^{\top}M\A.
\end{equation}

Clearly, the major cost for finding an inexact solution of (\ref{eq:Newton_simple2}) lies in the computation of
\[
\widetilde{\V}d_W=d_W+\sigma\G d_W+\mathcal{H} d_W.
\]
In the next two parts, we discuss how to reduce the computational cost for computing $\G d_W$ and $\mathcal{H} d_W$ by selecting special linear operators $\G$ and $\mathcal{H}$, respectively.

\subsection{The computation of $\G d_W$}

According to \cite[Theorem 2.5]{kaifeng2011algorithms} or \cite[Proposition 2]{jiang2013solving}, if $\|X_k(\widetilde{W})\|_2\leq \tau$, then we can select $\G$ such that $\mathcal{G} d_W = d_W$.
In the case $\|X_k(\widetilde{W})\|_2>\tau$, the computation of $\G d_W$ depends on the SVD of $X_k(\widetilde{W})$. 
Without loss of generality, we assume  $p\leq q$.  Let $X_k(\widetilde{W})$ have the following SVD:
$$
X_k(\widetilde{W})=U\left[ \mbox{Diag}(\nu)\ 0 \right]V^{\top}=U\mbox{Diag}(\nu)V^{\top}_1,
$$
where $U\in\mathcal{O}^{p}$, $V=[V_1\ V_2]\in\mathcal{O}^q$ with $V_1\in\mathbb{R}^{q\times p}$ and $V_2\in\mathbb{R}^{q\times(q-p)}$, and $\nu=(\nu_1,\ldots,\nu_p)^{\top}\in\mathbb{R}^p$ is the vector of singular values
ordered as $\nu_1\geq\ldots\geq\nu_p$.

Denote
\begin{align}
	\alpha:=&[k_1(\nu)],\quad \beta:=\beta_1\cup\beta_2,\quad \gamma:=\{p+1, p+2, \ldots,q\},\label{eq:alp_beta_gamma}\\
	\beta_1:=&\{k_1(\nu)+1,\ldots,k_2(\nu)\},\quad \beta_2:=\{k_2(\nu)+1, k_2(\nu)+2,\ldots,p\},\nonumber
\end{align}
where $k_1(\nu):=\max\{ i\in[p]\ |\ \nu_i>\tau \}$ and $k_2(\nu):=\max\{ i\in[p]\ |\ \nu_i\geq\tau \}$.
Define three matrices $\Xi^1\in\mathbb{R}^{p\times p}$, $\Xi^2\in\mathbb{R}^{p\times p}$, and $\Xi^3\in\mathbb{R}^{p\times (q-p)}$:
\begin{align}
	\Xi^1 =\left[
	\begin{array}{ccc}
		E_{\alpha\alpha} & E_{\alpha\beta_1} & \Xi^1_{\alpha\beta_2}\\
		E_{\beta_1\alpha} & 0 & 0\\
		(\Xi^1_{\alpha\beta_2})^{\top} & 0 & 0
	\end{array}
	\right],  \
	\Xi^2=\left[
	\begin{array}{ccc}
		\Xi^2_{\alpha\alpha} & \Xi^2_{\alpha\beta_1} & \Xi^2_{\alpha\beta_2}\\
		(\Xi^2_{\beta_1\alpha})^{\top} & 0 & 0\\
		(\Xi^2_{\alpha\beta_2})^{\top} & 0 & 0
	\end{array}
	\right], \ \Xi^3=\left[
	\begin{array}{c}
		\Xi^3_{\alpha\gamma} \\
		0\\
		0
	\end{array}
	\right], \label{eq:B_Xi^3}
\end{align}
where $E\in\mathbb{R}^{p\times p}$ is the matrix whose elements are all ones, and
\begin{align*}
	&(\Xi^1_{\alpha\beta_2})_{ij}=\frac{\nu_i-\tau}{\nu_i-\nu_j},\ \mbox{for}\ i\in\alpha, j\in\beta_2;\quad (\Xi^2_{\alpha\alpha})_{ij}=1-\frac{2\tau}{\nu_i+\nu_j}, \ \mbox{for}\ i, j\in\alpha;\\[2mm]
	&(\Xi^2_{\alpha\beta})_{ij}=\frac{\nu_i-\tau}{\nu_i+\nu_j}, \ \mbox{for}\ i\in\alpha, j\in\beta; \quad
	(\Xi^3_{\alpha\gamma})_{ij}=1-\frac{\tau}{\nu_i},\ \mbox{for}\ i\in\alpha, j\in\gamma.
\end{align*}
As a result, If $\|X_k(\widetilde{W})\|_2>\tau$, $\G d_W$ can be expressed as (see \cite[Theorem 2.5]{kaifeng2011algorithms} or \cite[Proposition 2]{jiang2013solving})
\begin{equation}\label{eq:GW}
	\G d_W=d_{W}-\smash{\underbrace{U\left[ \Xi^1\circ S(\widetilde{H}_1)+\Xi^2\circ T(\widetilde{H}_1) \right]V^{\top}_1}_{\G_1d_W}}-\smash{\underbrace{U(\Xi^3\circ\widetilde{H}_2)V^{\top}_2}_{\G_2d_W}}, \vspace{4mm}
\end{equation}
where $\widetilde{H}_1:=U^{\top}d_WV_1$ and  $\widetilde{H}_2:=U^{\top}d_WV_2$. Here, $S:\mathbb{R}^{p\times p}\rightarrow \mathbb{S}^p$ and $T:\mathbb{R}^{p\times p}\rightarrow\mathbb{R}^{p\times p}$ are linear matrix operators defined by
\[
S(X):=(X+X^{\top})/2\quad \mbox{and}\quad T(X):=(X-X^{\top})/2,\quad \forall X\in\mathbb{R}^{p\times p}.
\]
%

The primary computational cost of $\G d_W$  is approximately $4p^2q+4pq^2$ flops. However, by leveraging the sparsity of the matrices $\Xi^1$, $\Xi^2$, and $\Xi^3$, the computation of $\G_1d_W$ and $\G_2d_W$ in (\ref{eq:GW}) can be performed more efficiently.

\vspace{2mm}

\noindent\underline{\bf The term $\G_1d_W$}. Based on the definitions of  $\Xi^1$ and $\Xi^2$ in (\ref{eq:B_Xi^3}), $\G_1d_W$ in (\ref{eq:GW}) can be computed by
\begin{align*}
	\G_1d_W=&U\left[ \Xi^1\circ S(\widetilde{H}_1)+\Xi^2\circ T(\widetilde{H}_1) \right]V^{\top}_1\\
	=&U_\alpha\left[(S(\widetilde{H}_1))_{\alpha\beta_1} + \Xi^2_{\alpha\beta_1}\circ (T(\widetilde{H}_1))_{\alpha\beta_1}\right] V^{\top}_{1\beta_1} +U_\alpha\left[\Xi^1_{\alpha\beta_2}\circ(S(\widetilde{H}_1))_{\alpha\beta_2}\right.\\
	&\left.+\Xi^2_{\alpha\beta_2}\circ (T(\widetilde{H}_1))_{\alpha\beta_2}\right] V^{\top}_{1\beta_2}
	+\left\{U_{\alpha}\left[(S(\widetilde{H}_1))_{\alpha\alpha}+\Xi^1_{\alpha\alpha}\circ(T(\widetilde{H}_1))_{\alpha\alpha}  \right]\right.\\
	&+\left.U_{\beta_2}\left[(\Xi^1_{\alpha\beta_2})^{\top}\circ(S(\widetilde{H}_1))_{\beta_2\alpha} + (\Xi^2_{\alpha\beta_2})^{\top}\circ(T(\widetilde{H}_1))_{\beta_2\alpha}\right] \right.\\
	&+\left. U_{\beta_1}\left[(S(\widetilde{H}_1))_{\beta_1\alpha} + (\Xi^2_{\alpha\beta_1})^{\top}\circ(T(\widetilde{H}_1))_{\beta_1\alpha}\right]
	\right\} V^{\top}_{1\alpha}.
\end{align*}
It shows that the computational cost for computing $\G_1d_W$
can be reduced from $4p^3+4p^2q$ flops to $14|\alpha|pq+4|\alpha|p^2$ flops.
From the definition of $\alpha$ in (\ref{eq:alp_beta_gamma}), if the tuple 
$(W^{k+1}, b^{k+1}, v^{k+1}, U^{k+1}, \lambda^{k+1}, \Lambda^{k+1})$ generated by Algorithm \ref{alg_ALM_SMM} is a KKT solution to (\ref{eq:KKT_system_SMM}), then 
$|\alpha|=\mbox{rank}(W^{k+1})$ follows immediately.
As mentoned in Subsection \ref{subsec:low-rank}, for sufficient large $\tau$, the rank of $W^{k+1}$ is generally much less than $p$, i.e.,  $|\alpha|\ll p$. Therefore, the low-rank structure of 
$W^{k+1}$ can be effectively leveraged within our computational framework.

%
%
%

\vspace{2mm}

\noindent\underline{\bf The term $\G_2d_W$}. According to the definition of $\Xi^3$ in (\ref{eq:B_Xi^3}), we have
$$
\begin{aligned}
	\G_2d_W=&U(\Xi^3\circ\widetilde{H}_2)V^{\top}_2=U_{\alpha}[\Xi^3_{\alpha\gamma}\circ(U^{\top}_{\alpha}d_W V_2)]V^{\top}_2	=U_{\alpha}\mbox{Diag}(d)U^{\top}_{\alpha}d_WV_2V_2^{\top}\\[2mm]
	=&U_{\alpha}\mbox{Diag}(d)U^{\top}_{\alpha}d_W(I_q-V_1V^{\top}_1)=U_{\alpha}\mbox{Diag}(d)[U^{\top}_{\alpha}d_W-(U^{\top}_{\alpha}d_WV_1)V^{\top}_1]
\end{aligned}
$$
with $d:=(d_1,\ldots,d_{|\alpha|})^{\top}\in\mathbb{R}^{|\alpha|}$ and $d_{\ell}:=1-\tau/\nu_{\ell}$ for $\ell=1,\ldots, |\alpha|$.
The last equality shows that by the use of the sparsity in $\Xi^3$, the cost for computing
$U(\Xi^3\circ\widetilde{H}_2)V^{\top}_2$ can also be notably decreased from $4p^2(q-p)+4pq(q-p)$ flops to $4|\alpha|pq$ flops when $|\alpha|\ll p$.

The preceding arguments show that exploiting the sparsity in $\mathcal{G}$, induced by the low rank of $W^{k+1}$, reduces the computational cost of $\mathcal{G} d_W$ from $4p^2q + 4pq^2$ flops to $18|\alpha|pq + 4|\alpha|p^2$ flops, with $|\alpha|$ typically much smaller than $p$.

\subsection{The computation of $\mathcal{H} d_W$}

By the definition of $\mathcal{H}$ in (\ref{eq:def_H}), we have
\begin{equation}\label{temp-1}
	\mathcal{H} d_W=\sigma\A^*M\A d_W-\sigma^2(\sigma y^{\top}My+\rho)^{-1}\A^*Myy^{\top}M\A d_W.
\end{equation}
Notice that in the above expression, $M$ can be an arbitrary element in  $\partial\Pi_{S}(\omega(\widetilde{W},\widetilde{b}))$.
To save computational cost, we select a simple  $M\in \partial\Pi_{S}(\omega(\widetilde{W},\widetilde{b}))$ as follows:
$$
M=\mbox{Diag}(v)\quad \mbox{with}\quad v_j=\left\{
\begin{array}{ll}
	1, &\quad \mbox{if}\ 0<(\omega(\widetilde{W},\widetilde{b}))_j<C, \\
	0, &\quad \mbox{otherwise},
\end{array}
\right. \quad j=1,2,\ldots,n.
$$
Let
\begin{equation}\label{eq:def_J_1}
	\J_1=\J_1(\omega(\widetilde{W},\widetilde{b})):=\{j\in[n]\ |\ 0<(\omega(\widetilde{W},\widetilde{b}))_j<C\}.
\end{equation}
It is easy to get $y^{\top}My=|\J_1|$, and
\begin{align*}
	\A^*My=\A^*_{\scriptscriptstyle\J_1}y_{\scriptscriptstyle\J_1},\quad
	y^{\top}M\A\,d_W=y^{\top}_{\scriptscriptstyle\J_1}\A_{\scriptscriptstyle\J_1} d_W,\quad
	\A^*M\A\, d_W=\A^*_{\scriptscriptstyle\J_1}\A_{\scriptscriptstyle\J_1}d_W,
\end{align*}
where the linear mapping $\A_{\scriptscriptstyle\I}$ and its adjoint mapping $\A^*_{\scriptscriptstyle\I}$ with  $\I\subseteq[n]$ are defined by
\begin{equation}\label{eq:def_AI}
\A_{\scriptscriptstyle\I}Y:=(\A Y)_{\scriptscriptstyle\I}\quad \mbox{and}\quad \A^*_{\scriptscriptstyle\I}z:=\sum_{j\in\I}z_jy_jX_j, \quad \forall\ Y\in\mathbb{R}^{p\times q}, \quad\forall z\in\mathbb{R}^n.
\end{equation}
Therefore, it follows from (\ref{temp-1})  that
\[
\mathcal{H} d_W=\sigma\A^*_{\scriptscriptstyle\J_1}\A_{\scriptscriptstyle\J_1}d_W-\sigma^2(\sigma|\J_1|
+\rho)^{-1}\A^*_{\scriptscriptstyle\J_1}y_{\scriptscriptstyle\J_1}y^{\top}_{\scriptscriptstyle\J_1}\A_{\scriptscriptstyle\J_1}d_W.
\]
It means that the cost for computing  $\mathcal{H} d_W$ can be decreased from $O(npq)$ to $O(|\J_1|pq)$.
Observe that if $(W^{k+1}, b^{k+1}, v^{k+1}, U^{k+1}, \lambda^{k+1}, \Lambda^{k+1})$ is generated by Algorithm \ref{alg_ALM_SMM} as a solution of the KKT system (\ref{eq:KKT_system_SMM}), then the index set $\J_1$ corresponds to the active support matrices at $(W^{k+1}, b^{k+1})$. In fact, based on Step 2 in Algorithm \ref{alg_ALM_SMM} and (\ref{subprob:ALM_primal_v_U}), we deduce from the Moreau identity that
$$
\begin{aligned}
	\lambda^{k+1}=-\omega_k(W^{k+1},b^{k+1})+\mbox{Prox}_{\delta^*_S}(\omega_k(W^{k+1},b^{k+1}))
	=-\Pi_{S}(\omega_k(W^{k+1},b^{k+1})),
\end{aligned}
$$
where $\omega_k(\cdot,\cdot)$ is defined in (\ref{def:wk_X_k}). It further implies from (\ref{eq:def_J_1}) that
\begin{equation}\label{eq:relation_J1_ASM}
		\J_1=\{j\in[n]\ |\ (\omega_k(W^{k+1},b^{k+1}))_j\in(0,C)\}
		=\{j\in[n]\ |\ (-\lambda^{k+1})_j\in(0,C)\}.
\end{equation}
The analysis in Subsection \ref{subsec:sample_sparsity} reveals that the cardinality of the index set $\J_1$ is typically much smaller than $n$, thereby providing a positive answer to the question posed earlier. This results in a significant decrease in the computational cost for calculating $\mathcal{H} d_W$.

\subsection{Solving equation (\ref{eq:Newton_simple2})}

The discussion above has shown that  the equation (\ref{eq:Newton_simple2}) can be reduced as
\begin{equation}\label{eq:Newton_simple3}
	\left\{
	\begin{array}{l}
		\widetilde{\V}d_W=\mbox{Rhs1}-(\sigma\cdot\mbox{rhs2})(\sigma|\J_1|+\rho)^{-1}\A^*_{\scriptscriptstyle\J_1}y_{\scriptscriptstyle\J_1},\\[2mm]
		d_b=(\sigma|\J_1|+\rho)^{-1}\left( \mbox{rhs2}-\sigma y^{\top}_{\scriptscriptstyle\J_1}\A_{\scriptscriptstyle\J_1}d_W \right)
	\end{array}
	\right.
\end{equation}
with
\begin{equation}\label{eq:V_J}
	\widetilde{\V}=\I+\sigma\G+\sigma\A^*_{\scriptscriptstyle\J_1}\A_{\scriptscriptstyle\J_1}
	-\sigma^2(\sigma|\J_1|+\rho)^{-1}\A^*_{\scriptscriptstyle\J_1}y_{\scriptscriptstyle\J_1}y^{\top}_{\scriptscriptstyle\J_1}\A_{\scriptscriptstyle\J_1}.
\end{equation}
Accordingly, the inexact rule in Step 1 of Algorithm \ref{alg_SSNCG} is written as
$$
\left\|\widetilde{\V}d_W-\left[ \mbox{Rhs1}-(\sigma\cdot\mbox{rhs2})(\sigma |\J_1|+\rho)^{-1}\A^*_{\scriptscriptstyle\J_1}y_{\scriptscriptstyle\J_1} \right]\right\|\leq\min(\bar{\eta}, \|\nabla\varphi_k(\widetilde{W},\widetilde{b})\|^{1+\varrho}).
$$
The positive definiteness of the above linear operator $\widetilde{\V}$ is proved in the next proposition.
\begin{proposition}
	The linear operator $\widetilde{\V}:\mathbb{R}^{p\times q}\times\mathbb{R}\rightarrow \mathbb{R}^{p\times q}\times\mathbb{R}$ defined in (\ref{eq:V_J}) is a self-adjoint and positive definite operator.
\end{proposition}
\begin{proof}
	From \cite[Proposition 1]{meng2005semismoothness}, we know that each element $\G$ in $\partial\Pi_{\mathbb{B}^\tau_2}(\widetilde{W})$ is self-adjoint
	and positive semidefinite. So $\widetilde{\V}$ is self-adjoint.
	If the index set $\J_1$ is empty, then $\widetilde{\V}=\I+\sigma\G$ is positive definite. Otherwise, for any $d_W\in\mathbb{R}^{p\times q}\setminus\{0\}$, it follows from (\ref{eq:V_J}) that
	\begin{align*}
		&\left\langle d_W,\widetilde{\V}d_W\right\rangle
		=\left\langle d_W,\left[\I+\sigma\G+\sigma\A^*_{\scriptscriptstyle\J_1}\left( I_{|\scriptscriptstyle\J_1|}-\sigma(\sigma|\J_1|+\rho)^{-1}y_{\scriptscriptstyle\J_1}y^{\top}_{\scriptscriptstyle\J_1} \right)\A_{\scriptscriptstyle\J_1}\right]d_W\right\rangle\\
		=&\left\langle d_W,(\I+\sigma\G)d_W \right\rangle+\sigma\left\langle\A_{\J_1}d_W, \left(I_{|\J_1|}-\sigma(\sigma|\J_1|+\rho)^{-1}y_{\J_1}y^{\top}_{\J_1}\right)(\A_{\J_1}d_W)\right\rangle\\
		\geq&\left\langle d_W,(\I+\sigma\G)d_W \right\rangle+\sigma\left\langle\A	_{\J_1}d_W, \left(I_{|\J_1|}-y_{\J_1}y^{\top}_{\J_1}/|\J_1|\right)(\A_{\J_1}d_W)\right\rangle\\
		\geq&\langle d_W, (\I+\sigma\G)d_W \rangle>0,
	\end{align*}
	where the last second inequality follows from the fact
	that $I_{|\scriptscriptstyle\J_1|}-y_{\scriptscriptstyle\J_1}y^{\top}_{\scriptscriptstyle\J_1}/|\J_1|$ is an orthogonal projection matrix in $\mathbb{R}^{\scriptscriptstyle|\J_1|\times|\J_1|}$.
	The desired conclusion is achieved. \qed
\end{proof}


Table \ref{tab_cost_VH} summarizes the main computational costs of $\widetilde{\V}H$ obtained by a traditional manner in (\ref{eq:Newton_simple2}) and by exploiting the solution's low rank and sample sparsity in (\ref{eq:Newton_simple3}).
This indicates that when calculating the term $\widetilde{\V}H$, the main expenses are only $O(\max\{|\J_1|,|\alpha|\}pq)$, which is always significantly lower than $O(npq)$ when $n\gg q\geq p$.

\begin{table}[H]
	\caption{The main computational costs of $\widetilde{\V}H$ for a given $H$ under two different manners with $n\geq q\geq p$}
	\centering
	\label{tab_cost_VH}
	\footnotesize
		\vspace{1mm}
		\begin{spacing}{1.50}
			\begin{tabular}{ l c|l c}
				\hline
				\multicolumn{2}{ l| }{Traditional manner in (\ref{eq:Newton_simple2})} &
				\multicolumn{2}{ c }{Exploiting the sparsity of $M$ and $\G$ in (\ref{eq:Newton_simple3})} \\ \hline
				main terms & cost & main terms & cost \\ \hline
				$y^{\top}My$ & $O(n)$ & $y^{\top}_{\scriptscriptstyle\J_1}y_{\scriptscriptstyle\J_1}$ & $O(|\J_1|)$\\ \hline
				$\A^{*}My$ & $O(npq)$ & ${\A}^{*}_{\scriptscriptstyle\J_1}y_{\scriptscriptstyle\J_1}$ & $O(|\J_1|pq)$\\ \hline
				$y^{\top}M\A H$ & $O(npq)$ & $y^{\top}_{\scriptscriptstyle\J_1}\A_{\scriptscriptstyle\J_1}H$ & $O(|\J_1|pq)$\\ \hline
				$\A^{*}M\A H$ & $O(npq)$ & $\A^{*}_{\scriptscriptstyle\J_1}\A_{\scriptscriptstyle\J_1}H$ & $O(|\J_1|pq)$\\ \hline
				$\G H$ & $O(pq^2)$ & $\G H$ & $O(|\alpha|pq)$\\ \hline
				total cost & $O(npq)$ & total cost & $O(\max\{|\J_1|,|\alpha|\}pq)$ \\ \hline
			\end{tabular}
	\end{spacing}
\end{table}

\section{An adaptive strategy}\label{Section:AS}
The previous section concentrated on solving the SMM model (\ref{Model:SMM_eq2}) with fixed parameters $C$ and $\tau$.
In many cases, one needs to compute a solution path $\{W(C_i), b(C_i)\}^{N}_{i=1}$ across a specified sequence of grid points $0<C_1<C_2<\ldots<C_N$.
A typical example   is to optimize the hyperparameter $C$ through cross-validation while maintaining $\tau$ constant in model (\ref{Model:SMM_eq2}).
This section presents an adaptive sieving (AS) strategy that iteratively reduces the sample size, thereby facilitating the efficient computation of the solution path for $C$
 in models (\ref{Model:SMM_eq2}), particularly for extremely large $n$.

We develop the AS strategy based on the following two key observations.
First, the solution $(\overline{W}, \overline{b})$ of  (\ref{Model:SMM}) only depends  on the
samples corresponding to its support matrices (abbreviated as active samples).
Identifying these active samples can reduce the scale of the problem and thereby reduce its computational cost.
Second, as shown in  Figure \ref{fig:SM_ASM}, when  $C_{i-1}<C_{i}$, the support matrices  at $C=C_{i-1}$ is always a superset of those at $C=C_{i}$.
Thus, to solve model (\ref{Model:SMM}) at $C=C_i$, it is reasonable to  initially  restrict  it to the active sample set at $(W(C_{i-1}),b(C_{i-1}))$.

The core idea behind the AS strategy can be viewed as a specialized warm-start approach. It aggressively guesses
and adjusts the active samples at $(W(C_i), b(C_i))$ for a new grid point $C_i$, using the solution
$(W(C_{i-1}), b(C_{i-1}))$ obtained from the previous grid point $C_{i-1}$. Specifically, we initiate
the $i$th problem within a restricted sample space that includes all active samples at $(W(C_{i-1}), b(C_{i-1}))$.
After solving the restricted SMM model, we update active samples at the computed solution.
The process is  repeated until no additional active samples need to be included.

The similar strategy has been employed to efficiently compute solution paths for convex composite sparse
machine learning problems by leveraging the inherent sparsity of solutions
\cite{yuan2023adaptive,yuan2022dimension,li2023mars,wu2023convex}. However, since the SMM model
 (\ref{Model:SMM_eq2}) does not guarantee solution's sparsity, the method proposed in \cite{yuan2023adaptive,lin2020adaptive}
 cannot be directly applied here. 
 We will develop an adaptive sieving strategy that efficiently generates a solution path for the SMM model (\ref{Model:SMM_eq2}) by effectively exploiting the intrinsic sparsity of the samples.
For a specified index set $\I\subseteq[n]$, we define
\[
S^{\,i}_{\I}:=\{ x\in\mathbb{R}^{|\I|}\ |\ 0\leq x_j\leq C_i, j=1,\ldots,|\I| \}.
\]
The subscript $\I$ is omitted when $\I=[n]$.
Detailed framework of the AS strategy is presented in Algorithm \ref{alg_AS}.

\begin{algorithm}[htbp]
	\caption{An AS strategy } 
	\vskip 0.05in
	\noindent
	Choose an initial point $(\overline{W}(C_0),\overline{b}(C_0))\in\mathbb{R}^{p\times q}\times\mathbb{R}$
	, a sequence of grid points $0<C_0<C_1<\ldots<C_N$, a tolerance $\varepsilon\geq0$, a scalar $\widehat{\varepsilon}\geq0$, and a positive integer $d_{\max}$. Compute the initial index set $\I^*(C_0):=\{j\in[n]\ |\ y_j(\langle \overline{W}(C_0),X_j \rangle + \overline{b}(C_0))\leq 1+\widehat{\varepsilon}\}$. Execute the following steps for $i=1,2,\ldots,N$ with the initialization $\I^0(C_i)=\I^*(C_{i-1})$ and $k=0$.
	\vskip 0.1in
	\begin{algorithmic}[0]
		\State {\bf Step 1}. Find a KKT tuple $(W^k(C_i), b^k(C_i), v^{k}_{\scriptscriptstyle\I}(C_i), U^k(C_i), \lambda^k_{\scriptscriptstyle\I}(C_i), \Lambda^k(C_i))$ $\in\mathbb{R}^{p\times q}\times\mathbb{R}\times\mathbb{R}^{|\I^k(C_i)|}\times\mathbb{R}^{p\times q}\times\mathbb{R}^{|\I^k(C_i)|}\times\mathbb{R}^{p\times q}$ of the following problem
		\begin{equation}\label{model:AS_subpro}
			\begin{array}{cc}
				\mathop{\mbox{minimize}}\limits_{\tiny W, b, v_{\scriptscriptstyle\I}, U} & {\displaystyle\frac{1}{2}}\|W\|^2_{\F}+\tau\|U\|_*+\delta^*_{S^{\,i}_{\scriptscriptstyle\I^ k}}(v_{\scriptscriptstyle\I})-\langle \delta_W,W \rangle-b\,\delta_b-\langle \delta_{v_{\scriptscriptstyle\I}},v_{\scriptscriptstyle\I} \rangle-\langle \delta_U,U \rangle\\
				\mbox{subject to} & \A_{\I^k(C_i)} W +by_{\scriptscriptstyle\I^k(C_i)}+v_{\scriptscriptstyle\I}-(e_n)_{\I^k(C_i)}=\delta_{\lambda_{\scriptscriptstyle\I}},\ W-U=\delta_{\Lambda},
			\end{array}
		\end{equation}
		where $\A_{\I^k(C_i)}$ is defined in (\ref{eq:def_AI}) and
		$(\delta_W, \delta_b,\delta_{v_{\scriptscriptstyle\I}},\delta_U,\delta_{\lambda_{\scriptscriptstyle\I}},\delta_{\Lambda})\in\mathbb{R}^{p\times q}\times\mathbb{R}\times\mathbb{R}^{|\I^k(C_i)|}\times\mathbb{R}^{p\times q}\times\mathbb{R}^{|\I^k(C_i)|}\times\mathbb{R}^{p\times q}$ is an error satisfying
		\begin{equation}\label{eq:error_AS}
			\max(\|\delta_W\|, |\delta_b|, \|\delta_{v_{\scriptscriptstyle\I}}\|, \|\delta_U\|, \|\delta_{\lambda_{\scriptscriptstyle\I}}\|, \|\delta_{\Lambda}\|)\leq\varepsilon.
		\end{equation}
		\State {\bf Step 2}. Compute  index set
		\begin{equation}\label{def:indexJ}
			\J^k(C_i):=\{j\in\overline{\I}^k(C_i)\ |\ v^k_j(C_i)\geq0\},
		\end{equation}
		where $\overline{\I}^k(C_i):=[n]\setminus\I^k(C_i)$.
Extend $v^k_{\scriptscriptstyle\I}(C_i)$ and $\lambda^k_{\I}(C_i)$ to $n$-dimensional vectors $v^k(C_i)$ and $\lambda^k(C_i)$ by setting
$v^k_j(C_i)=1-y_j(\langle W^k(C_i),X_j \rangle+b^k(C_i))$ for $j\in\overline{\I}^k(C_i)$, and $\lambda^k_{j}(C_i) = -C$ if $j\in\J^k(C_i)$, otherwise $\lambda^k_{j}(C_i)=0$.
		\vskip 0.1in
		\State {\bf Step 3}. If $\J^k(C_i)=\emptyset$, stop and set
$(\overline{W}(C_i), \overline{b}(C_i), \overline{v}(C_i), \overline{U}(C_i), \overline{\lambda}(C_i)$, $\overline{\Lambda}(C_i))=(W^k(C_i), b^k(C_i), v^k(C_i), U^k(C_i), \lambda^k(C_i), \Lambda^k(C_i))$ with $\I^*(C_i)=\{j\in[n]\ |\ y_j(\langle \overline{W}(C_i),X_j \rangle + \overline{b}(C_i))\leq 1+\widehat{\varepsilon}\}$. Let $i\leftarrow i+1$ (unless $i=N$ already so that the algorithm should be stopped). Otherwise compute the index set
		$$
		\widehat{\J}^{k+1}(C_i):=\left\{j\in\J^{k}(C_i)\ \bigg|
		\begin{array}{l}
			(v^k(C_i))_j \ \mbox{is among the first}\ d\ \mbox{largest} \\
			\mbox{values}\ \mbox{in}\ \{v^k_t(C_i)\}_{t\in\J^{k+1}(C_i)}
		\end{array}
		\right\},
		$$
		where $d$ is a given positive integer satisfying $d\leq\min\{|\J^{k}(C_i)|, d_{\max}\}$. Let $\I^{k+1}(C_i):=\I^k(C_i)\cup\widehat{\J}^{k+1}(C_i)$.
Set $k\leftarrow k+1$ and go to Step 1.
	\end{algorithmic}
	\label{alg_AS}
\end{algorithm}

\begin{remark}	
	\begin{itemize}
	\item[(a)] In Step 2 of Algorithm \ref{alg_AS}, the index set $\J^k(C_i)$ can be expressed as
	$$
	\J^k(C_i)=\{j\in\overline{\I}^k(C_i)\ |\ y_j(\langle W^k(C_i),X_j\rangle+b^k(C_i))\leq 1\}.
	$$
It aims to identify all indices of the support matrices at $(W^k(C_i), b^k(C_i))$ beyond $\I^k(C_i)$, as described by (\ref{eq:inclusion_SM}).
	\item[(b)] In Step 3 of Algorithm \ref{alg_AS}, the nonnegative scalar $\widehat{\varepsilon}$ provides
flexibility in adjusting the size of the index set $\I^*(C_i)$. Specifically,
if $\widehat{\varepsilon}=0$, $\I^*(C_i)$ includes all indices of active samples at $(\overline{W}(C_i),\overline{b}(C_i))$.
In general, by selecting an appropriate positive scalar $\widehat{\varepsilon}>0$, we aim to ensure that
$\I^*(C_{i})$ includes as many active indices as possible at the solution for the next grid point
$C_{i+1}$.
	\item[(c)] The reduced subproblem (\ref{model:AS_subpro}) shares the same mathematical structure as (\ref{Model:SMM_eq2})
 with a reduced subset of training samples $\{X_j, y_j\}_{j\in\I^k(C_i)}$. This similarity allows us to efficiently apply Algorithm \ref{alg_ALM_SMM} to solve (\ref{model:AS_subpro}).
	\end{itemize}
\end{remark}

Lastly, we establish the finite convergence of Algorithm \ref{alg_AS} in the following theorem. Its proof is given in Appendix \ref{Appendix:Theorem_AS}.
\begin{theorem}\label{Thm:AS_convergence}
	For each $C_i$, Algorithm \ref{alg_AS} terminates within  a finite number of iterations. That is,
 there exists an integer $\bar{k}\in[1,n]$ such that $\J^{\,\bar{k}}(C_i)=\emptyset$. Furthermore, any output solution
 $\left(\overline{W}(C_i), \overline{b}(C_i), \overline{v}(C_i), \overline{U}(C_i), \overline{\lambda}(C_i), \overline{\Lambda}(C_i)\right)$
 is a KKT tuple for the following problem
	\begin{equation}\label{eq:AS_whole_model}
		\begin{array}{cc}
			\mathop{\mbox{minimize}}\limits_{(W,b,v,U)\in\X} & {\displaystyle\frac{1}{2}}\|W\|^2_{\F}+\tau\|U\|_*+\delta^*_{S^{i}}(v)-\langle \delta_W,W \rangle-b\,\delta_b-\langle \delta_{v},v \rangle-\langle \delta_U,U \rangle\\
			\mbox{subject to} & \A W +by+v-e_n=\delta_{\lambda},\ W-U=\delta_{\Lambda},
		\end{array}
	\end{equation}
	where the error $(\delta_W, \delta_b, \delta_v, \delta_U, \delta_{\lambda}, \delta_{\Lambda})$ satisfies $\max(\|\delta_{W}\|, |\delta_b|, \|\delta_v\|,$\ $\|\delta_U\|, \|\delta_{\lambda}\|$, $\|\delta_{\Lambda}\|)\leq\varepsilon$.
\end{theorem}

\section{Numerical experiments}\label{Sec:Experiments}
We have executed extensive numerical experiments to demonstrate the efficiency of our proposed algorithm and strategy on both large-scale synthetic and real datasets.
All the experiments were implemented in MATLAB on a windows workstation (8-core, Intel(R) Core(TM) i7-10700 @ 2.90GHz, 64G RAM).

\subsection{Implementation details}
For the SMM model (\ref{Model:SMM_eq2}), we measure the qualities of the computed solutions via the following relative KKT residual of the iterates:
\begin{equation}\label{eq:relkkt}
	\eta_{kkt}=\max\{ \eta_{W}, \eta_b, \eta_v, \eta_U, \eta_{\lambda}, \eta_{\Lambda} \},
\end{equation}
where
\begin{align*}
	& \eta_{W}:=\frac{\|W+\A^*\lambda+\Lambda\|_{\F}}{1+\|W\|_{\F}+\|\A^*\lambda\|_{\F}+\|\Lambda\|_{\F}}, \ \eta_b:=\frac{|\lambda^{\top}y|}{1+\sqrt{n}}, \ \eta_{v}:=\frac{\|\lambda+\Pi_S(v-\lambda)\|}{1+\|\lambda\|+\|v\|},\\
	& \eta_{U}=\frac{\|\Lambda-\Pi_{\mathbb{B}^{\tau}_2}(U+\Lambda)\|_{\F}}{1+\|\Lambda\|_{\F}+\|U\|_{\F}}, \ \eta_{\lambda}=\frac{\|\ \A W+by+v-e_n\|}{1+\sqrt{n}}, \ \eta_{\Lambda}=\frac{\|W-U\|_{\F}}{1+\|W\|_{\F}+\|U\|_{\F}}.
\end{align*}
By default, the termination criteria for Algorithm \ref{alg_ALM_SMM} (denoted as ALM-SNCG) involves either satisfying
$\eta_{kkt}\leq\varepsilon$ or reaching the maximum number 500 of iterations.
In the case where both  algorithms stopped with  different criteria, we used the value of the objective functions obtained
from the ALM-SNCG satisfying $\eta_{kkt}\leq 10^{-8}$ (denoted as $\mbox{obj}_{\scriptsize\mbox{opt}}$) as
benchmarks to evaluate their quality (denoted as $\mbox{obj}_{\scriptsize\mbox{computed}}$):
\begin{equation}\label{eq:Relobj}
	\mbox{Relobj}:=\frac{|\mbox{obj}_{\scriptsize\mbox{computed}}-\mbox{obj}_{\scriptsize\mbox{opt}}|}{1+|\mbox{obj}_{\scriptsize\mbox{opt}}|}.
\end{equation}

Both synthetic and real datasets were utilized in the subsequent experiments.
The sampling data $\{(X_i,y_i)\}^n_{i=1}\subseteq\mathbb{R}^{p\times q}\times\{-1,1\}$ was generated following the process  in \cite{luo2015support}.
For any $k\in[p]$ and $\ell\in[q]$, denote $a_{k\ell}:=\left[ (X_1)_{k\ell},\ldots,(X_n)_{k\ell} \right]^{\top}\in\mathbb{R}^n$. Each vector $a_{k\ell}$ is constructed as:
$$
a_{k\ell}=b_{\lceil r\ell/q\rceil}+\epsilon_{k\ell}\ \mbox{with}\ \epsilon_{k\ell}\sim N(0,\delta^2I_n),\qquad k=1,\ldots,p, \quad \ell=1,\ldots,q.
$$
Here, $b_1, b_2, \ldots, b_r$ are $r$
orthonormal vectors in $\mathbb{R}^n$, and $0<r\leq\min\{p,q\}$.
Given a $p\times q$ matrix $W$ with rank $r$,
each sample's class label is determined by $y_i=\mbox{sign}(\langle W,X_i \rangle)$ for $i\in[n]$, where $\mbox{sign}(x)$ equals $1$ if $x\geq0$ and $-1$ otherwise.
In our experiments, we set $r=20$ and $\delta=2\times 10^{-4}$.
For each data configuration,  the sample sizes $n$ vary from $1.25\times 10^4$ to $1.25\times 10^6$, with $80\%$ used as training data and the remaining $20\%$ used for testing.

We also utilized four distinct real datasets. They are
\begin{itemize}
\item  the electroencephalogram alcoholism dataset
(EEG)\footnote{\url{http://kdd.ics.uci.edu/databases/eeg/eeg.html}} for classifying subjects based on EEG signals as alcoholics or non-alcoholics;
\item the INRIA person dataset (INRIA)\footnote{\url{ftp://ftp.inrialpes.fr/pub/lear/douze/data/INRIAPerson.tar}}
for detecting the presence of individuals in an image;
\item the CIFAR-10 dataset (CIFAR-10)\footnote{\url{https://www.cs.toronto.edu/~kriz/cifar.html}}
for classifying images as depicting either a dog or a truck;
\item the MNIST handwritten digit dataset (MNIST)\footnote{\url{https://yann.lecun.com/exdb/mnist}}
for classifying handwritten digit images as the number 0 or not.
\end{itemize}
Additional information for these datasets is provided in Tabel \ref{Table:real_data},
 where $n_{\scriptsize\mbox{train}}$ and $n_{\scriptsize\mbox{test}}$ denote the number of training and test samples, respectively.

\begin{table}[H]
	\caption{Summary of four real datasets}
	\centering
	\label{Table:real_data}
	\small
		\vspace{1mm}
		\begin{spacing}{1.50}
			\begin{tabular}{ c c c c | c c c c}
				\hline
				datasets & $p\times q$ & $n_{\scriptsize\mbox{train}}$ & $n_{\scriptsize\mbox{test}}$ & datasets & $p\times q$ & $n_{\scriptsize\mbox{train}}$ & $n_{\scriptsize\mbox{test}}$\\ \hline
				EEG & $256\times 64$ & 300 & 66 & CIFAR-10 & $32\times 32$ & 10000 & 2000\\ \hline
				INRIA & $70\times 134$ & 3634 & 1579 & MNIST & $28\times 28$ & 60000 & 10000\\ \hline
			\end{tabular}
	\end{spacing}
\end{table}

In subsequent experiments, we set the parameter $\tau$ in the SMM model (\ref{Model:SMM}) to 10 and 100 for random data, and to 1 and 10 for real data.

\subsection{Solving the SMM model at a fixed parameter value $C$}\label{Section:exp_SMM_fixed_C}

In this subsection, we evaluate four algorithms - ALM-SNCG, the inexact semi-proximal ADMM (isPADMM), the symmetric Gauss-Seidel based isPADMM (sGS-isPADMM),
and Fast ADMM with restart (F-ADMM) - for solving the SMM model (\ref{Model:SMM}) with fixed parameters $\tau$
and $C$.  Additional details on isPADMM and sGS-isPADMM can be found in Appendices \ref{Appendex:isPADMM} and \ref{Appendex:sGS-isPADMM}, respectively.
The F-ADMM is based on the works of Luo et al. \cite{luo2015support}\footnote{Its code is available for download at:
 \url{http://bcmi.sjtu.edu.cn/\textasciitilde luoluo/code/smm.zip}} and Goldstein et al. \cite{goldstein2014fast}.
Each method is subject to a two-hour computational time limit.
For consistent comparison, we adopt the termination criterion $\mbox{Relobj}\leq\varepsilon$, where $\mbox{Relobj}$ is defined in (\ref{eq:Relobj}).
We assess their performance at two accuracy levels: $\varepsilon=10^{-4}$ and $\varepsilon=10^{-6}$.

Firstly, we evaluate the performance of two-block ADMM algorithms (F-ADMM and isPADMM) and three-block ADMM (sGS-isPADMM) for solving the SMM model (\ref{Model:SMM}) on the EEG training dataset. The ``Iteration time" column in Table \ref{tab:Real_data_fixed_C_sGS} represents the average time per iteration for each algorithm.
Numerical results in Table \ref{tab:Real_data_fixed_C_sGS} show that, on average, sGS-isPADMM is 2.6 times faster per iteration than F-ADMM and 52.2 times faster than isPADMM. However, sGS-isPADMM is 7.2 times slower than isPADMM in terms of total time consumed due to having 184.8 times more iteration steps on average, leading to higher SVD costs from the soft thresholding operator. Henceforth, we will focus on comparing the performance of two-block ADMM algorithms and ALM-SNCG on synthetic and real datasets.

\begin{table}[http]
	\caption{Results of sGS-isPADMM (sGS), F-ADMM (F), and isPADMM (isP) with {$\mbox{Relobj}\leq \varepsilon$} on EEG training dataset}
	\centering
	\label{tab:Real_data_fixed_C_sGS}
	\scriptsize
	\resizebox{\linewidth}{!}{
		\begin{tabular}{|c|c|c|c|c|c|c|c|c|c|c|c|c|c|c|}
			\hline
			\multicolumn{1}{ |c| }{\multirow{2}{*}{$\varepsilon$}} & \multicolumn{1}{ c| }{\multirow{2}{*}{$\tau$}} &  \multicolumn{1}{ c| }{\multirow{2}{*}{$C$}} &  \multicolumn{3}{ c| }{\multirow{1}{*}{Relobj}} & \multicolumn{3}{ c| }{\multirow{1}{*}{Iteration numbers}} & \multicolumn{3}{ c| }{\multirow{1}{*}{Time (seconds)}} & \multicolumn{3}{ c| }{\multirow{1}{*}{Iteration time (seconds)}} \\ \cline{4-15}
			& & & sGS & F & isP & sGS & F & isP & sGS & F & isP & sGS & F & isP \\ \hline
			\multicolumn{1}{ |c| }{\multirow{8}{*}{1e-4}} &
			\multicolumn{1}{ c| }{\multirow{4}{*}{1}} &
			1e-4 & 1e-4 & 6e-5 & 9e-5 & 3876 & 477 & 19 & 26.1 & 5.0 & 2.0 & 7e-3 & 1e-2 & 1e-1 \\ \cline{3-15}
			& & 1e-3 & 1e-4 & 1e-4 & 5e-5 & 4251 & 285 & 16 & 27.5 & 4.4 & 3.8 & 6e-3 & 2e-2 & 2e-1\\ \cline{3-15}
			& & 1e-2 & 1e-4 & 5e-4 & 1e-4 & 4329 & 30000 & 22 & 28.5 & 616.3 & 11.3 & 7e-3 & 2e-2 & 5e-1\\ \cline{3-15}
			& & 1e-1 & 1e-4 & 5e-3 & 8e-5 & 4317 & 30000 & 23 & 27.7 & 616.2 & 16.1 & 6e-3 & 2e-2 & 7e-1\\ \cline{2-15}
			& \multicolumn{1}{ c| }{\multirow{4}{*}{10}} &
			1e-4 & 6e-5 & 6e-5 & 6e-5 & 3324 & 117 & 20 & 20.6 & 1.2 & 6.0 & 6e-3 & 1e-2 & 3e-1\\ \cline{3-15}
			& & 1e-3 & 1e-4 & 1e-4 & 1e-4 & 8818 & 8141 & 111 & 55.0 & 100.2 & 9.7 & 6e-3 & 1e-2 & 9e-2\\ \cline{3-15}
			& & 1e-2 & 1e-4 & 1e-4 & 9e-5 & 7986 & 2841 & 93 & 49.9 & 58.5 & 22.1 & 6e-3 & 2e-2 & 2e-1\\ \cline{3-15}
			& & 1e-1 & 1e-4 & 1e-1 & 1e-4 & 5428 & 30000 & 63 & 33.9 & 621.1 & 31.1 & 6e-3 & 2e-2 & 5e-1\\ \hline
			
			\multicolumn{1}{ |c| }{\multirow{8}{*}{1e-6}} &
			\multicolumn{1}{ c| }{\multirow{4}{*}{1}} &
			1e-4 &  1e-6 & 1e-6 & 9e-7 & 23861 & 25759 & 45 & 151.5 & 279.1 & 3.3 & 6e-3 & 1e-2 & 7e-2\\ \cline{3-15}
			& & 1e-3 & 1e-6 & 1e-6 & 8e-7 & 7147 & 501 & 25 & 44.2 & 9.0 & 5.5 & 6e-3 & 2e-2 & 2e-1\\ \cline{3-15}
			& & 1e-2 & 1e-6 & 5e-4 & 7e-7 & 7614 & 30000 & 48 & 48.5 & 621.7 & 21.7 & 6e-3 & 2e-2 & 5e-1\\ \cline{3-15}
			& & 1e-1 & 1e-6 & 5e-3 & 9e-7 & 7544 & 30000 & 44 & 47.1 & 623.7 & 29.4 & 6e-3 & 2e-2 & 7e-1\\ \cline{2-15}
			& \multicolumn{1}{ c| }{\multirow{4}{*}{10}} &
			1e-4 & 8e-7 & 7e-7 & 6e-7 & 4268 & 121 & 28 & 26.4 & 1.3 & 11.6 & 6e-3 & 1e-2 & 4e-1\\ \cline{3-15}
			& & 1e-3 & 1e-6 & 1e-6 & 1e-6 & 26985 & 21117 & 151 & 168.4 & 259.1 & 12.5 & 6e-3 & 1e-2 & 8e-2\\ \cline{3-15}
			& & 1e-2 & 1e-6 & 1e-6 & 7e-7 & 11853 & 5363 & 111 & 74.2 & 109.6 & 25.9 & 6e-3 & 2e-2 & 2e-1\\ \cline{3-15}
			& & 1e-1 & 1e-6 & 6e-2 & 4e-7 & 9044 & 30000 & 89 & 56.5 & 617.2 & 41.1 & 6e-3 & 2e-2 & 5e-1\\ \hline
	\end{tabular}}
\end{table}

It is worth mentioning that ALM-SNCG is initialized with a lower-accuracy starting point from isPADMM, using up to four iterations for synthetic datasets and up to ten iterations for real datasets. Otherwise, the origin is used as the default starting point for the algorithms.
Tables \ref{tab: random_fixed_C} and \ref{tab:Real_data_fixed_C} present the numerical performance of the above three algorithms,  including the following information:
\begin{itemize}
	\item $|\J_1|$: the cardinality of the index set $\J_1$ defined in (\ref{eq:def_J_1});
	\item $|\alpha|$: the cardinality of the index set $\alpha$ defined in (\ref{eq:alp_beta_gamma});
	\item $\mbox{Accuracy}{\scriptsize\mbox{test}}$: the classification accuracy on test set for solutions from each algorithm, i.e.,
	\begin{equation}\label{eq:accu_test}
		\mbox{Accuracy}_{\scriptsize\mbox{test}}:=|\{i\ |\ \hat{y}_i=(y_{\scriptsize\mbox{test}})_i,\ i=1,\ldots,n_{\scriptsize\mbox{test}}\}|/n_{\scriptsize\mbox{test}}.
	\end{equation}
	Here, $\hat{y}$ represents the predicted class label vector, and $y_{\scriptsize\mbox{test}}$ denotes the true class label vector for the test set.
\end{itemize}

\subsubsection{Numerical results on synthetic data}\label{Section:exp_SMM_fixed_C_random}

Table \ref{tab: random_fixed_C} shows the computational results for ALM-SNCG, isPADMM, and F-ADMM on randomly
generated data with $C$ ranging from $0.1$ to $100$.
The results indicate that ALM-SNCG and isPADMM can solve all problems
within two hours. 
 In contrast, F-ADMM encountered memory trouble when the sample sizes exceed $10^5$. 
For $(n, p, q, \tau, C)=(10000, 1000, 500, 100, 100)$, F-ADMM exibited lower classification accuracy
($\mbox{Accuracy}_{\scriptsize\mbox{test}}$) compared to isPADMM and ALM-SNCG due to its inability to achive the desired solution accuracy.
Even for the smallest instance $(n, p, q)=(10000, 100, 100)$, F-ADMM is, on average, 101.2 times slower
than ALM-SNCG to reach $\mbox{Relobj}\leq 10^{-4}$ and 210.2 times slower to reach $\mbox{Relobj}\leq 10^{-6}$.
On the other hand, ALM-SNCG is on average 8.7 times faster than isPADMM for $\varepsilon=10^{-4}$ and 13.0 times
faster for $\varepsilon=10^{-6}$. It shows that  ALM-SNCG is  more efficient and stable, particularly for high-accuracy solutions.

\begin{table}[http]
	\caption{Results of F-ADMM (F), isPADMM (isP), and ALM-SNCG (A) with $\mbox{Relobj}\leq\varepsilon$ on random data}
	\centering
	\label{tab: random_fixed_C}
	\scriptsize
	\resizebox{\linewidth}{!}{
		\begin{tabular}{|c|c|c|c|c|c|c|c|c|c|c|c|c|c|c|}
			\hline
			Data & \multicolumn{1}{ c| }{\multirow{2}{*}{$\varepsilon$}} & \multicolumn{1}{ c| }{\multirow{2}{*}{$\tau$}} & \multicolumn{1}{ c| }{\multirow{2}{*}{$C$}} & \multicolumn{1}{ c| }{\multirow{2}{*}{$|\mathcal{J}_1|$}}  & \multicolumn{1}{ c| }{\multirow{2}{*}{$|\alpha|$}} & \multicolumn{3}{ c| }{\multirow{1}{*}{$\mbox{Accuracy}_{\scriptsize\mbox{test}}$}} & \multicolumn{3}{ c| }{\multirow{1}{*}{Relobj}}& \multicolumn{3}{ c| }{\multirow{1}{*}{Time (seconds)}} \\ \cline{7-15}
			$(n, p, q)$ & & & & & & F & isP & A & F & isP & A & F & isP & A  \\ \hline
			
			\multicolumn{1}{ |c| }{\multirow{16}{*}{(1e4, 1e2, 1e2)}} & \multicolumn{1}{ c| }{\multirow{8}{*}{1e-4}} &
			\multicolumn{1}{ c| }{\multirow{4}{*}{10}} &
			0.1 & 27 & 1 & 0.9876 & 0.9872 & 0.9872 & 8e-5 & 8e-5 & 3e-5 & 11.2 & 18.2 & 1.6\\ \cline{4-15}
			& & & 1 & 65 & 1 & 0.9932 & 0.9932 & 0.9936 & 8e-5 & 1e-4 & 8e-5 & 13.5 & 13.5 & 1.7\\ \cline{4-15}
			& & & 10 & 46 & 1 & 0.9928 & 0.9928 & 0.9928 & 4e-5 & 3e-5 & 1e-4 & 26.0 & 20.6 & 1.6\\ \cline{4-15}
			& & & 100 & 26 & 1 & 0.9936 & 0.9936 & 0.9936 & 9e-5 & 6e-5 & 4e-5 & 28.1 & 21.7 & 7.9\\ \cline{3-15}
			& & \multicolumn{1}{ c| }{\multirow{4}{*}{100}} &
			0.1 & 19 & 1 & 0.9844 & 0.9844 & 0.9848 & 8e-5 & 8e-5 & 1e-4 & 14.8 & 10.4 & 2.2\\ \cline{4-15}
			& & & 1 & 42 & 1 & 0.9880 & 0.9880 & 0.9880 & 3e-5 & 6e-5 & 8e-5 & 10.5 & 17.1 & 2.3\\ \cline{4-15}
			& & & 10 & 39 & 1 & 0.9932 & 0.9932 & 0.9932 & 1e-4 & 9e-5 & 7e-5 & 26.3 & 69.1 & 1.7\\ \cline{4-15}
			& & & 100 & 37 & 1 & 0.9932 & 0.9932 & 0.9932 & 1e-4 & 9e-5 & 8e-5 & 1526.2 & 37.5 & 2.0\\ \cline{2-15}
			
			& \multicolumn{1}{ c| }{\multirow{8}{*}{1e-6}} &
			\multicolumn{1}{ c| }{\multirow{4}{*}{10}} &
			0.1 & 21 & 1 & 0.9872 & 0.9872 & 0.9872 & 4e-7 & 8e-7 & 8e-7 & 21.7 & 29.4 & 2.4\\ \cline{4-15}
			& & & 1 & 20 & 1 & 0.9932 & 0.9932 & 0.9932 & 4e-7 & 8e-7 & 8e-7 & 21.2 & 27.1 & 5.5\\ \cline{4-15}
			& & & 10 & 23 & 1 & 0.9928 & 0.9928 & 0.9928 & 7e-7 & 3e-7 & 8e-7 & 33.1 & 37.0 & 4.3\\ \cline{4-15}
			& & & 100 & 22 & 1 & 0.9936 & 0.9936 & 0.9936 & 3e-5 & 7e-7 & 9e-7 & {\bf7200.6} & 42.4 & 14.0\\ \cline{3-15}
			& & \multicolumn{1}{ c| }{\multirow{4}{*}{100}} &
			0.1 & 11 & 1 & 0.9848 & 0.9848 & 0.9848 & 9e-7 & 7e-7 & 8e-7 & 29.5 & 18.9 & 4.9\\ \cline{4-15}
			& & & 1 & 18 & 1 & 0.9880 & 0.9880 & 0.9880 & 8e-7 & 8e-7 & 9e-7 & 15.4 & 25.6 & 5.3\\ \cline{4-15}
			& & & 10 & 23 & 1 & 0.9932 & 0.9932 & 0.9932 & 7e-7 & 4e-7 & 1e-6 & 38.7 & 46.4 & 6.0\\ \cline{4-15}
			& & & 100 &22 & 1 & 0.9932 & 0.9932 & 0.9932 & 5e-4 & 8e-7 & 7e-7 & {\bf7200.0} & 116.7 & 6.4\\ \hline
			
			\multicolumn{1}{ |c| }{\multirow{16}{*}{(1e4, 1e3, 5e2)}} & \multicolumn{1}{ c| }{\multirow{8}{*}{1e-4}} &
			\multicolumn{1}{ c| }{\multirow{4}{*}{10}} &
			0.1 & 22 & 1 & 0.9940 & 0.9940 & 0.9940 & 6e-5 & 9e-5 & 8e-5 & 421.3 & 2020.6 & 237.8\\ \cline{4-15}
			& & & 1 & 23 & 1 & 0.9924 & 0.9924 & 0.9924 & 3e-6 & 8e-5 & 9e-5 & 536.0 & 1082.0 & 295.1\\ \cline{4-15}
			& & & 10 & 23 & 1 & 0.9924 & 0.9924 & 0.9924 & 2e-3 & 8e-5 & 9e-5 & {\bf7200.1} & 1565.1 & 255.8\\ \cline{4-15}
			& & & 100 & 32 & 107 & 0.9888 & 0.9936 & 0.9936 & 2e+0 & 9e-5 & 8e-5 & {\bf7202.5} & 1249.0 & 864.5\\ \cline{3-15}
			& & \multicolumn{1}{ c| }{\multirow{4}{*}{100}} &
			0.1 & 20 & 1 & 0.9880 & 0.9880 & 0.9880 & 9e-5 & 8e-5 & 2e-5 & 556.5 & 2030.1 & 243.7\\ \cline{4-15}
			& & & 1 & 24 & 1 & 0.9940 & 0.9940 & 0.9940 & 4e-3 & 7e-5 & 7e-5 & {\bf7202.7} & 1229.2 & 242.3\\ \cline{4-15}
			& & & 10 & 23 & 1 & 0.9936 & 0.9936 & 0.9936 & 3e-2 & 4e-5 & 6e-5 & {\bf7202.5} & 2050.7 & 258.6\\ \cline{4-15}
			& & & 100 & 23 & 1 & 0.9836 & 0.9920 & 0.9920 & 2e+0 & 7e-5 & 8e-5 & {\bf7200.7} & 2118.4 & 510.0\\ \cline{2-15}
			
			& \multicolumn{1}{ c| }{\multirow{8}{*}{1e-6}} &
			\multicolumn{1}{ c| }{\multirow{4}{*}{10}} &
			0.1 & 20 & 1 & 0.9940 & 0.9940 & 0.9940 & 9e-7 & 1e-6 & 2e-7 & 525.8 & 2353.6 & 334.3\\ \cline{4-15}
			& & & 1 & 21 & 1 & 0.9924 & 0.9924 & 0.9924 & 7e-5 & 3e-7 & 4e-7 & {\bf7203.5} & 1580.2 & 517.2\\ \cline{4-15}
			& & & 10 & 21 & 1 & 0.9924 & 0.9924 & 0.9924 & 2e-3 & 4e-7 & 7e-7 & {\bf7201.9} & 3531.0 & 413.7\\ \cline{4-15}
			& & & 100 & 32 & 107 & 0.9916 & 0.9936 & 0.9936 & 2e+0 & 7e-7 & 6e-7 & {\bf7202.8} & 2177.1 & 1229.5\\ \cline{3-15}
			& & \multicolumn{1}{ c| }{\multirow{4}{*}{100}} &
			0.1 & 20 & 1 & 0.9880 & 0.9880 & 0.9880 & 9e-7 & 8e-7 & 7e-7 & 969.2 & 2685.1 & 265.5\\ \cline{4-15}
			& & & 1 & 21 & 1 & 0.9940 & 0.9940 & 0.9940 & 4e-3 & 7e-7 & 4e-7 & {\bf7201.5} & 3305.9 & 411.1\\ \cline{4-15}
			& & & 10 & 22 & 1 & 0.9936 & 0.9936 & 0.9936 & 3e-2 & 9e-7 & 9e-7 & {\bf7204.6} & 6706.1 & 371.6\\ \cline{4-15}
			& & & 100 & 21 & 1 & 0.9792 & 0.9920 & 0.9920 & 3e+0 & 6e-7 & 4e-7 & {\bf7200.6} & 3331.2 & 628.9\\ \hline
			
			\multicolumn{1}{ |c| }{\multirow{16}{*}{(1e5, 50, 1e2)}} & \multicolumn{1}{ c| }{\multirow{8}{*}{1e-4}} &
			\multicolumn{1}{ c| }{\multirow{4}{*}{10}} &
			0.1 & 197 & 1 & - & 0.9681 & 0.9683 & - & 7e-5 & 5e-5 & - & 27.6 & 3.6\\ \cline{4-15}
			& & & 1 & 636 & 1 & - & 0.9682 & 0.9682 & - & 5e-5 & 2e-5 & - & 33.1 & 6.8\\ \cline{4-15}
			& & & 10 & 502 & 1 & - & 0.9686 & 0.9687 & - & 4e-5 & 4e-5 & - & 38.5 & 15.2\\ \cline{4-15}
			& & & 100 & 430 & 45 & - & 0.9692 & 0.9693 & - & 7e-5 & 4e-5 & - & 93.4 & 88.8\\ \cline{3-15}
			& & \multicolumn{1}{ c| }{\multirow{4}{*}{100}} &
			0.1 & 189 & 1 & - & 0.9682 & 0.9681 & - & 7e-5 & 5e-5 & - & 56.9 & 3.6\\ \cline{4-15}
			& & & 1 & 568 & 1 & - & 0.9685 & 0.9685 & - & 9e-5 & 7e-5 & - & 64.6 & 6.7\\ \cline{4-15}
			& & & 10 & 590 & 1 & - & 0.9685 & 0.9685 & - & 6e-5 & 7e-5 & - & 76.9 & 4.8\\ \cline{4-15}
			& & & 100 & 388 & 1 & - & 0.9690 & 0.9690 & - & 7e-5 & 2e-5 & - & 111.5 & 38.1\\ \cline{2-15}
			
			& \multicolumn{1}{ c| }{\multirow{8}{*}{1e-6}} &
			\multicolumn{1}{ c| }{\multirow{4}{*}{10}} &
			0.1 & 64 & 1 & - & 0.9683 & 0.9683 & - & 1e-6 & 9e-7 & - & 365.7 & 7.0\\ \cline{4-15}
			& & & 1 & 97 & 1 & - & 0.9682 & 0.9682 & - & 1e-6 & 9e-7 & - & 278.2 & 13.9\\ \cline{4-15}
			& & & 10 & 93 & 1 & - & 0.9687 & 0.9686 & - & 9e-7 & 8e-7 & - & 96.6 & 22.4\\ \cline{4-15}
			& & & 100 & 117 & 45 & - & 0.9693 & 0.9693 & - & 1e-6 & 9e-7 & - & 282.0 & 104.6\\ \cline{3-15}
			& & \multicolumn{1}{ c| }{\multirow{4}{*}{100}} &
			0.1 & 36 & 1 & - & 0.9682 & 0.9682 & - & 7e-7 & 9e-7 & - & 324.6 & 10.3\\ \cline{4-15}
			& & & 1 & 48 & 1 & - & 0.9684 & 0.9685 & - & 9e-7 & 9e-7 & - & 654.4 & 19.8\\ \cline{4-15}
			& & & 10 & 66 & 1 & - & 0.9685 & 0.9685 & - & 6e-7 & 8e-7 & - & 171.1 & 15.2\\ \cline{4-15}
			& & & 100 & 87 & 1 & - & 0.9690 & 0.9690 & - & 8e-7 & 1e-6 & - & 417.5 & 45.2\\ \hline
			
			\multicolumn{1}{ |c| }{\multirow{16}{*}{(1e6, 50, 1e2)}} & \multicolumn{1}{ c| }{\multirow{8}{*}{1e-4}} &
			\multicolumn{1}{ c| }{\multirow{4}{*}{10}} &
			0.1 & 1935 & 1 & - & 0.9018 & 0.9017 & - & 5e-5 & 5e-5 & - & 194.2 & 29.0\\ \cline{4-15}
			& & & 1 & 11357 & 1 & - & 0.9017 & 0.9017 & - & 4e-5 & 3e-5 & - & 167.5 & 60.0\\ \cline{4-15}
			& & & 10 & 9267 & 30 & - & 0.9074 & 0.9078 & - & 7e-5 & 4e-6 & - & 694.4 & 271.5\\ \cline{4-15}
			& & & 100 & 13027 & 50 & - & 0.9445 & 0.9448 & - & 7e-5 & 5e-5 & - & 2346.5 & 67.0\\ \cline{3-15}
			& & \multicolumn{1}{ c| }{\multirow{4}{*}{100}} &
			0.1 & 1348 & 1 & - & 0.9017 & 0.9017 & - & 8e-5 & 4e-5 & - & 339.2 & 34.7\\ \cline{4-15}
			& & & 1 & 7545 & 1 & - & 0.9018 & 0.9018 & - & 8e-5 & 6e-5 & - & 370.5 & 49.4\\ \cline{4-15}
			& & & 10 & 8750 & 1 & - & 0.9018 & 0.9018 & - & 7e-5 & 7e-6 & - & 327.2 & 154.8\\ \cline{4-15}
			& & & 100 & 12252 & 29 & - & 0.9408 & 0.9410 & - & 9e-5 & 3e-5 & - & 1702.6 & 1136.2\\ \cline{2-15}
			
			& \multicolumn{1}{ c| }{\multirow{8}{*}{1e-6}} & \multicolumn{1}{ c| }{\multirow{4}{*}{10}} &
			0.1 & 862 & 1 & - & 0.9018 & 0.9017 & - & 7e-7 & 7e-7 & - & 372.7 & 42.7\\ \cline{4-15}
			& & & 1 & 3072 & 2 & - & 0.9017 & 0.9017 & - & 1e-6 & 9e-7 & - & 1222.0 & 89.9\\ \cline{4-15}
			& & & 10 & 4294 & 30 & - & 0.9077 & 0.9077 & - & 1e-6 & 9e-7 & - & 2003.1 & 290.6\\ \cline{4-15}
			& & & 100 & 2409 & 50 & - & 0.9450 & 0.9450 & - & 7e-7 & 9e-7 & - & 5413.7 & 224.1\\ \cline{3-15}
			& & \multicolumn{1}{ c| }{\multirow{4}{*}{100}} &
			0.1 & 312 & 1 & - & 0.9017 & 0.9017 & - & 9e-7 & 8e-7 & - & 1703.4 & 72.9\\ \cline{4-15}
			& & & 1 & 835 & 1 & - & 0.9018 & 0.9018 & - & 8e-7 & 1e-6 & - & 4299.5 & 127.2\\ \cline{4-15}
			& & & 10 & 2319 & 1 & - & 0.9018 & 0.9018 & - & 6e-7 & 9e-7 & - & 2320.4 & 187.4\\ \cline{4-15}
			& & & 100 & 1833 & 29 & - & 0.9413 & 0.9413 & - & 8e-7 & 9e-7 & - & 5371.3 & 1289.7\\ \hline
	\end{tabular}}
\end{table}


\subsubsection{Numerical results on real data}\label{SMM_fixed_C_real_data}

Table \ref{tab:Real_data_fixed_C} lists the results for ALM-SNCG, F-ADMM, and isPADMM on the model (\ref{Model:SMM})
using four real datasets described in Table \ref{Table:real_data}, with the parameter $C$ ranging from $10^{-4}$ to $100$. 
The results in the table show that both ALM-SNCG and isPADMM achieved high accuracy ($\mbox{Relobj}\leq 10^{-6}$) in all instances,
while F-ADMM only stopped successfully for $56\%$ problems with accuracy $\mbox{Relobj}\leq 10^{-4}$, and for  $47\%$
problems with accuracy  $\mbox{Relobj}\leq 10^{-6}$.
Table \ref{tab:Real_data_fixed_C} also reveals that ALM-SNCG performed on average 2.6 times (5.9 times) faster than
isPADMM and 422.7 times (477.2 times) faster than F-ADMM when the  stop criteria was set to  $\mbox{Relobj}\leq 10^{-4}$ ($\mbox{Relobj}\leq 10^{-6}$).

	\begin{table}[http]
	\caption{Results of F-ADMM (F), isPADMM (isP), and ALM-SNCG (A) with {$\mbox{Relobj}\leq \varepsilon$} on real data}
	\centering
	\label{tab:Real_data_fixed_C}
	\scriptsize
	\resizebox{\linewidth}{!}{
		\begin{tabular}{|c|c|c|c|c|c|c|c|c|c|c|c|c|c|c|}
			\hline
			Prob & \multicolumn{1}{ c| }{\multirow{2}{*}{$\varepsilon$}} & \multicolumn{1}{ c| }{\multirow{2}{*}{$\tau$}} &  \multicolumn{1}{ c| }{\multirow{2}{*}{$C$}} & \multicolumn{1}{ c| }{\multirow{2}{*}{$|\mathcal{J}_1|$}}  & \multicolumn{1}{ c| }{\multirow{2}{*}{$|\alpha|$}} & \multicolumn{3}{ c| }{\multirow{1}{*}{$\mbox{Accuracy}_{\scriptsize\mbox{test}}$}} & \multicolumn{3}{ c| }{\multirow{1}{*}{Relobj}} & \multicolumn{3}{ c| }{\multirow{1}{*}{Time (seconds)}} \\ \cline{7-15}
			$(n, p, q)$ & & & & & & F & isP & A & F & isP & A & F & isP & A \\ \hline
			\multicolumn{1}{ |c| }{\multirow{16}{*}{$\begin{array}{c}
						\mbox{EEG}\\(300,256,64)\end{array}$}} & \multicolumn{1}{ c| }{\multirow{8}{*}{1e-4}} &
			\multicolumn{1}{ c| }{\multirow{4}{*}{1}} &
			1e-4 & 6 & 2 & 0.7424 & 0.7424 & 0.7424 & 6e-5 & 9e-5 & 6e-5 & 5.0 & 2.0 & 2.1 \\ \cline{4-15}
			& & & 1e-3 & 60 & 5 & 0.8636 & 0.8636 & 0.8636 & 1e-4 & 5e-5 & 6e-5 & 4.4 & 3.8 & 4.5\\ \cline{4-15}
			& & & 1e-2 & 123 & 5 & 0.9394 & 0.9545 & 0.9394 & 5e-4 & 1e-4 & 7e-5 & 616.3 & 11.3 & 6.9\\ \cline{4-15}
			& & & 1e-1 & 123 & 5 & 0.9545 & 0.9545 & 0.9394 & 5e-3 & 8e-5 & 9e-5 & 616.2 & 16.1 & 12.1\\ \cline{3-15}
			& & \multicolumn{1}{ c| }{\multirow{4}{*}{10}} &
			1e-4 & 95 & 0 & 0.6667 & 0.6667 & 0.6667 & 6e-5 & 6e-5 & 4e-6 & 1.2 & 6.0 & 2.7\\ \cline{4-15}
			& & & 1e-3 & 6 & 2 & 0.7424 & 0.7424 & 0.7424 & 1e-4 & 1e-4 & 9e-5 & 100.2 & 9.7 & 5.7\\ \cline{4-15}
			& & & 1e-2 & 60 & 5 & 0.8636 & 0.8636 & 0.8636 & 1e-4 & 9e-5 & 6e-5 & 58.5 & 22.1 & 9.2\\ \cline{4-15}
			& & & 1e-1 & 124 & 5 & 0.9394 & 0.9394 & 0.9394 & 1e-1 & 1e-4 & 6e-5 & 621.1 & 31.1 & 10.8\\ \cline{2-15}
			
			& \multicolumn{1}{ c| }{\multirow{8}{*}{1e-6}} &
			\multicolumn{1}{ c| }{\multirow{4}{*}{1}} &
			1e-4 & 6 & 2 & 0.7424 & 0.7424 & 0.7424 & 1e-6 & 9e-7 & 7e-7 & 279.1 & 3.3 & 2.2\\ \cline{4-15}
			& & & 1e-3 & 60 & 5 & 0.8636 & 0.8636 & 0.8636 & 1e-6 & 8e-7 & 9e-7 & 9.0 & 5.5 & 4.7\\ \cline{4-15}
			& & & 1e-2 &123 & 5 & 0.9394 & 0.9394 & 0.9394 & 5e-4 & 7e-7 & 6e-7 & 621.7 & 21.7 & 8.2\\ \cline{4-15}
			& & & 1e-1 & 123 & 5 & 0.9394 & 0.9394 & 0.9394 & 5e-3 & 9e-7 & 5e-7 & 623.7 & 29.4 & 12.2\\ \cline{3-15}
			& & \multicolumn{1}{ c| }{\multirow{4}{*}{10}} &
			1e-4 & 91 & 0 & 0.6667 & 0.6667 & 0.6667 & 7e-7 & 6e-7 & 5e-8 & 1.3 & 11.6 & 2.6\\ \cline{4-15}
			& & & 1e-3 & 6 & 2 & 0.7424 & 0.7424 & 0.7424 & 1e-6 & 1e-6 & 7e-7 & 259.1 & 12.5 & 6.1\\ \cline{4-15}
			& & & 1e-2 & 60 & 5 & 0.8636 & 0.8636 & 0.8636 & 1e-6 & 7e-7 & 5e-7 & 109.6 & 25.9 & 10.2\\ \cline{4-15}
			& & & 1e-1 & 123 & 5 & 0.9394 & 0.9394 & 0.9394 & 6e-2 & 4e-7 & 5e-7 & 617.2 & 41.1 & 15.8\\ \hline

			\multicolumn{1}{ |c| }{\multirow{16}{*}{$\begin{array}{c}
						\mbox{INRIA}\\(3634,70,134)\end{array}$}} & \multicolumn{1}{ c| }{\multirow{8}{*}{1e-4}} &
			\multicolumn{1}{ c| }{\multirow{4}{*}{1}} &
			1e-3 & 12 & 2 & 0.8892 & 0.8892 & 0.8892 & 1e-4 & 6e-5 & 1e-4 & 5.4 & 6.7 & 5.6\\ \cline{4-15}
			& & & 1e-2 & 100 & 6 & 0.8898 & 0.8898 & 0.8898 & 8e-5 & 1e-4 & 1e-4 & 6.0 & 10.6 & 6.5\\ \cline{4-15}
			& & & 1e-1 & 807 & 35 & 0.8638 & 0.8645 & 0.8645 & 5e-3 & 1e-4 & 7e-5 & {\bf 7200.1} & 103.5 & 37.3\\ \cline{4-15}
			& & & 1 & 1080 & 44 & 0.8385 & 0.8379 & 0.8379 & 8e-1 & 8e-5 & 7e-5 & {\bf 7200.1} & 173.8 & 44.2\\ \cline{3-15}
			& & \multicolumn{1}{ c| }{\multirow{4}{*}{10}} &
			1e-3 & 1716 & 0 & 0.7131 & 0.7131 & 0.7131 & 9e-5 & 8e-5 & 7e-5 & 392.7 & 328.7 & 23.9\\ \cline{4-15}
			& & & 1e-2 & 17 & 2 & 0.8904 & 0.8904 & 0.8904 & 1e-4 & 9e-5 & 9e-5 & 21.8 & 13.5 & 8.2\\ \cline{4-15}
			& & & 1e-1 & 123 & 6 & 0.8879 & 0.8879 & 0.8879 & 1e-4 & 9e-5 & 9e-5 & 43.8 & 30.0 & 17.1\\ \cline{4-15}
			& & & 1 & 938 & 27 & 0.8493 & 0.8436 & 0.8461 & 3e-1 & 1e-4 & 7e-5 & {\bf 7200.2} & 148.5 & 123.4\\ \cline{2-15}
			
			& \multicolumn{1}{ c| }{\multirow{8}{*}{1e-6}} &
			\multicolumn{1}{ c| }{\multirow{4}{*}{1}} &
			1e-3 & 11 & 2 & 0.8892 & 0.8892 & 0.8892 & 9e-7 & 8e-7 & 8e-7 & 6.7 & 16.3 & 6.2\\ \cline{4-15}
			& & & 1e-2 & 98 & 6 & 0.8898 & 0.8898 & 0.8898 & 9e-7 & 9e-7 & 7e-7 & 12.6 & 26.7 & 8.2\\ \cline{4-15}
			& & & 1e-1 & 807 & 35 & 0.8645 & 0.8645 & 0.8645 & 5e-3 & 1e-6 & 9e-7 & {\bf 7200.0} & 416.1 & 88.2	\\ \cline{4-15}
			& & & 1 & 1080 & 44 & 0.8391 & 0.8379 & 0.8379 & 8e-1 & 1e-6 & 1e-6 & {\bf 7200.2} & 502.1 & 59.1\\ \cline{3-15}
			& & \multicolumn{1}{ c| }{\multirow{4}{*}{10}} &
			1e-3 &2028 & 0 & 0.7131 & 0.7131 & 0.7131 & 3e-5 & 7e-7 & 8e-7 & {\bf 7200.1} & 1508.3 & 26.0				\\ \cline{4-15}
			& & & 1e-2 & 17 & 2 & 0.8911 & 0.8911 & 0.8911 & 1e-6 & 8e-7 & 9e-7 & 49.1 & 19.8 & 9.1\\ \cline{4-15}
			& & & 1e-1 & 123 & 6 & 0.8879 & 0.8879 & 0.8879 & 1e-6 & 9e-7 & 9e-7 & 142.8 & 48.6 & 19.9\\ \cline{4-15}
			& & & 1 & 937 & 27 & 0.8436 & 0.8461 & 0.8461 & 4e-1 & 1e-6 & 8e-7 & {\bf 7200.0} & 362.8 & 174.4\\ \hline
			
			\multicolumn{1}{ |c| }{\multirow{16}{*}{$\begin{array}{c}
						\mbox{CIFAR-10}\\(10000,32,32)\end{array}$}} & \multicolumn{1}{ c| }{\multirow{8}{*}{1e-4}} &
			\multicolumn{1}{ c| }{\multirow{4}{*}{1}} &
			1e-3 & 15 & 2 & 0.8015 & 0.8015 & 0.8020 & 9e-5 & 9e-5 & 1e-4 & 14.2 & 2.0 & 1.3\\ \cline{4-15}
			& & & 1e-2 & 62 & 5 & 0.8245 & 0.8245 & 0.8245 & 9e-5 & 8e-5 & 1e-4 & 8.3 & 1.7 & 1.1\\ \cline{4-15}
			& & & 1e-1 & 319 & 19 & 0.8205 & 0.8210 & 0.8190 & 1e-4 & 9e-5 & 7e-5 & 26.5 & 4.6 & 1.5\\ \cline{4-15}
			& & & 1 & 751 & 30 & 0.7790 & 0.7990 & 0.8020 & 4e-1 & 9e-5 & 7e-5 & {\bf 7200.6} & 12.5 & 2.1\\ \cline{3-15}
			& & \multicolumn{1}{ c| }{\multirow{4}{*}{10}} &
			1e-3 & 5 & 1 & 0.7645 & 0.7645 & 0.7645 & 1e-4 & 6e-5 & 8e-5 & 25.9 & 2.2 & 1.5\\ \cline{4-15}
			& & & 1e-2 & 15 & 2 & 0.8050 & 0.8050 & 0.8050 & 1e-4 & 7e-5 & 1e-4 & 26.2 & 2.1 & 1.6	\\ \cline{4-15}
			& & & 1e-1 & 70 & 5 & 0.8235 & 0.8235 & 0.8235 & 1e-4 & 9e-5 & 8e-5 & 43.9 & 2.4 & 1.9 \\ \cline{4-15}
			& & & 1 & 457 & 20 & 0.7805 & 0.8170 & 0.8180 & 2e-1 & 8e-5 & 9e-5 & {\bf 7200.3} & 4.2 & 1.9\\ \cline{2-15}
			
			& \multicolumn{1}{ c| }{\multirow{8}{*}{1e-6}} &
			\multicolumn{1}{ c| }{\multirow{4}{*}{1}} &
			1e-3 & 11 & 2 & 0.8015 & 0.8020 & 0.8015 & 9e-7 & 6e-7 & 7e-7 & 19.4 & 6.4 & 1.5\\ \cline{4-15}
			& & & 1e-2 &55 & 5 & 0.8245 & 0.8245 & 0.8245 & 9e-7 & 5e-7 & 7e-7 & 19.7 & 6.8 & 1.6\\ \cline{4-15}
			& & & 1e-1 & 304 & 19 & 0.8200 & 0.8200 & 0.8195 & 4e-5 & 9e-7 & 8e-7 & {\bf 7200.3} & 17.6 & 1.6\\ \cline{4-15}
			& & & 1 & 723 & 30 & 0.7940 & 0.8015 & 0.8015 & 2e-1 & 9e-7 & 9e-7 & {\bf 7200.4} & 42.0 & 3.6\\ \cline{3-15}
			& & \multicolumn{1}{ c| }{\multirow{4}{*}{10}} &
			1e-3 & 6 & 1 & 0.7645 & 0.7645 & 0.7645 & 1e-6 & 7e-7 & 1e-6 & 79.3 & 4.3 & 2.0\\ \cline{4-15}
			& & & 1e-2 & 12 & 2 & 0.8050 & 0.8050 & 0.8050 & 1e-6 & 8e-7 & 7e-7 & 92.2 & 7.5 & 2.1\\ \cline{4-15}
			& & & 1e-1 & 62 & 5 & 0.8235 & 0.8235 & 0.8235 & 1e-6 & 1e-6 & 8e-7 & 146.9 & 7.3 & 2.2\\ \cline{4-15}
			& & & 1 & 438 & 20 & 0.8020 & 0.8175 & 0.8170 & 3e-1 & 9e-7 & 8e-7 & {\bf 7200.2} & 13.5 & 2.9\\ \hline
			
			\multicolumn{1}{ |c| }{\multirow{16}{*}{$\begin{array}{c}
						\mbox{MNIST}\\(60000,28,28)\end{array}$}} & \multicolumn{1}{ c| }{\multirow{8}{*}{1e-4}} &
			\multicolumn{1}{ c| }{\multirow{4}{*}{1}} &
			1e-1 & 212 & 14 & 0.9928 & 0.9928 & 0.9927 & 9e-5 & 9e-5 & 9e-5 & 225.0 & 12.8 & 4.0\\ \cline{4-15}
			& & & 1 & 422 & 22 & 0.9930 & 0.9930 & 0.9927 & 2e-3 & 9e-5 & 1e-4 & {\bf 7201.4} & 16.9 & 5.7\\ \cline{4-15}
			& & & 1e+1 & 498 & 25 & 0.9764 & 0.9922 & 0.9922 & 6e+0 & 8e-5 & 6e-5 & {\bf 7201.0} & 36.4 & 9.6\\ \cline{4-15}
			& & & 1e+2 & 549 & 26 & 0.9540 & 0.9917 & 0.9915 & 2e+1 & 9e-5 & 9e-5 & {\bf 7201.6} & 42.2 & 16.7\\ \cline{3-15}
			& & \multicolumn{1}{ c| }{\multirow{4}{*}{10}} &
			1e-1 & 91 & 5 & 0.9924 & 0.9924 & 0.9924 & 9e-5 & 7e-5 & 9e-5 & 178.2 & 7.7 & 4.4\\ \cline{4-15}
			& & & 1 & 274 & 14 & 0.9926 & 0.9926 & 0.9926 & 1e-3 & 9e-5 & 1e-4 & {\bf 7200.3} & 12.4 & 4.1\\ \cline{4-15}
			& & & 1e+1 & 473 & 21 & 0.9769 & 0.9926 & 0.9926 & 5e+0 & 9e-5 & 6e-5 & {\bf 7201.0} & 17.7 & 8.3\\ \cline{4-15}
			& & & 1e+2 & 541 & 25 & 0.9710 & 0.9916 & 0.9916 & 1e+1 & 1e-4 & 8e-5 & {\bf 7202.4} & 34.5 & 17.1\\ \cline{2-15}
			
			& \multicolumn{1}{ c| }{\multirow{8}{*}{1e-6}} &
			\multicolumn{1}{ c| }{\multirow{4}{*}{1}} &
			1e-1 & 213 & 14 & 0.9928 & 0.9928 & 0.9928 & 4e-6 & 9e-7 & 1e-6 & {\bf 7200.7} & 42.0 & 4.3\\ \cline{4-15}
			& & & 1 & 414 & 22 & 0.9929 & 0.9929 & 0.9928 & 2e-3 & 1e-6 & 1e-6 & {\bf 7202.9} & 50.3 & 7.7\\ \cline{4-15}
			& & & 1e+1 & 496 & 25 & 0.9764 & 0.9923 & 0.9923 & 6e+0 & 9e-7 & 7e-7 & {\bf 7201.0} & 94.7 & 16.8\\ \cline{4-15}
			& & & 1e+2 & 542 & 26 & 0.9540 & 0.9915 & 0.9915 & 2e+1 & 9e-7 & 8e-7 & {\bf 7202.2} & 101.8 & 48.8\\ \cline{3-15}
			& & \multicolumn{1}{ c| }{\multirow{4}{*}{10}} &
			1e-1 & 91 & 5 & 0.9924 & 0.9924 & 0.9924 & 9e-7 & 9e-7 & 8e-7 & 483.5 & 17.1 & 5.2\\ \cline{4-15}
			& & & 1 & 269 & 14 & 0.9926 & 0.9926 & 0.9926 & 2e-3 & 8e-7 & 8e-7 & {\bf 7202.8} & 28.3 & 5.1\\ \cline{4-15}
			& & & 1e+1 & 470 & 21 & 0.9855 & 0.9926 & 0.9926 & 4e+0 & 1e-6 & 9e-7 & {\bf 7200.2} & 64.5 & 13.3\\ \cline{4-15}
			& & & 1e+2 & 532 & 25 & 0.9657 & 0.9916 & 0.9916 & 1e+1 & 8e-7 & 9e-7 & {\bf 7201.1} & 129.6 & 51.7\\ \hline
	\end{tabular}}
\end{table}
		

\subsection{Computing a solution path of the SMM model for $\{C_i\}^N_{i=1}$}
In this subsection, we assess the numerical performance of the AS strategy on both synthetic and real datasets.
While do numerical experiments, we used the ALM-SNCG as inner solver. It means that
for given parameters $C_i$ and $\tau$, we solved problem (\ref{Model:SMM}) by ALM-SNCG.
The ALM-SNCG with AS strategy is abbreviated as AS+ALM.
We compared the performance of AS+ALM with that of the warm-started ALM-SNCG (abbreviated as Warm+ALM), evaluated at $\eta_{kkt}\leq 10^{-4}$ and $\eta_{kkt}\leq 10^{-6}$, respectively.
In all experiments, we set $d_{\max}=500$ in Algorithm \ref{alg_AS}. Additionally, $\widehat{\varepsilon}$ is set to 0.05 for $n<500,000$ and 0.1 otherwise for random data, and to 0.4 for real data.

Tables \ref{tab_AS_solution_path_random} and  \ref{tab_AS_solution_path MNIST} show the performance of both methods  AS+ALM and Warm+ALM.
Each column in the tables takes the following meaning:
\begin{itemize}
\item $\mbox{Avg}\underline{~}{\mbox{nSM}}$: the average number of the support matrices for the reduced subproblems (\ref{model:AS_subpro});
\item $\mbox{Avg}\underline{~}\mbox{sam}$/$\mbox{Max}\underline{~}\mbox{sam}$: the average/maximum sample size of (\ref{model:AS_subpro});
\item  Worst relkkt: the maximum relative KKT residuals as in (\ref{eq:relkkt});
\item  $\mbox{Avg}\underline{~}{|\J_1|}$/$\mbox{Avg}\underline{~}\mbox{time}$: the average values of $|\J_1|$/computation times;
\item Iteration numbers: the average number of iterations. For AS+ALM, the column shows the average number of AS rounds, followed by the average number of outer augmented Lagrangian iterations, with the average number of inner semismooth Newton-CG iterations in parentheses.
\end{itemize}

\subsubsection{Numerical results on synthetic data}

Table \ref{tab_AS_solution_path_random} displays the numerical results from AS+ALM and Warm+ALM using a given sequence $\{C_i\}^N_{i=1}$ on random data. 
The parameter $C$ ranges from 0.1 to 100, divided into 50 equally spaced grid points on the \(\log_{10}\) scale.
A notable observation from Table \ref{tab_AS_solution_path_random} is that, on average, the AS strategy requires only one iteration per grid point to achieve the desired solution. It suggests that the initial index sets $\I^0(C_i)$  almost fully cover indices of the support matrices at the solutions for each $C_i$.

Table \ref{tab_AS_solution_path_random} also reveals that despite the average number of inner ALM-SNCG iterations in AS+ALM surpasses that of Warm+ALM across all instances, the AS strategy can still accelerate the Warm+ALM by an average factor of 2.69 (under $\eta_{kkt}\leq 10^{-4}$) and 3.07 (under $\eta_{kkt}\leq 10^{-6}$) in terms of time consumption. This is due to the smaller average sample size of AS's reduced subproblems compared to that of the original SMM models.
Furthermore, Figure \ref{fig:AS_vs_Warm_path_Random_Percentage_1e_4_6} depicts that the percentages of support matrices in the reduced subproblems ($\mbox{mean}(|\mbox{SM}|/n_{\scriptscriptstyle\I})$) consistently exceed those in the original problems ($\mbox{mean}(|\mbox{SM}|)/n$). Here, $|\mbox{SM}|$ and $n_{\scriptscriptstyle\I}$ represent the number of support matrices and samples in the reduced subproblem (\ref{model:AS_subpro}) during each AS iteration, respectively.

	\begin{table}[http]
		\caption{Results for generating solution paths using Warm+ALM (Warm) and AS+ALM (AS) with a sequence $\{C_i\}^N_{i=1}$ and $\eta_{kkt}\leq\varepsilon$ on random data}
		\centering
		\label{tab_AS_solution_path_random}
		\resizebox{\linewidth}{!}{
			\begin{tabular}{|c|c|c|c|c|c|c|c|c|c|c|c|c|c|}
				\hline			
				Data & \multicolumn{1}{ c| }{\multirow{2}{*}{$\varepsilon$}} &
				\multicolumn{1}{ c| }{\multirow{2}{*}{$\tau$}} &
				\multicolumn{3}{ c| }{\multirow{1}{*}{Information of the AS}} &
				\multicolumn{2}{ c| }{\multirow{1}{*}{Avg$\underline{~}|\J_1|$}} &
				\multicolumn{2}{ c| }{\multirow{1}{*}{Worst relkkt}} &
				\multicolumn{2}{ c| }{\multirow{1}{*}{Iteration numbers}} &
				\multicolumn{2}{ c| }{\multirow{1}{*}{Avg$\underline{~}$time (seconds)}} \\ \cline{4-14}
				$(n, p, q)$ & & & Avg$\underline{~}$nSM & Avg$\underline{~}$sam & Max$\underline{~}$sam & Warm & AS & Warm & AS & Warm & AS & Warm & AS \\ \hline
				\multicolumn{1}{ |c| }{\multirow{4}{*}{(1e4, 1e2, 1e2)}} & \multicolumn{1}{ c| }{\multirow{2}{*}{1e-4}}  & 10 & 1374 & 1532 & 4152 & 28 & 25 & 1e-4 & 5e-5 & (2, 20) & 1(3, 25) & 1.50 & 0.49\\ \cline{3-14}
				& & 100 & 2542 & 2825 & 8739 & 27 & 25 & 1e-4 & 1e-4 & (3, 20) & 1(4, 24) & 1.50 & 0.65 \\ \cline{2-14}
				& \multicolumn{1}{ c| }{\multirow{2}{*}{1e-6}} & 10 & 1375 & 1532 & 4153 & 21 & 21 & 1e-6 & 6e-7 & (11, 46) & 1(13, 50) & 3.93 & 1.20\\  \cline{3-14}
				& & 100 & 2542 & 2825 & 8736 & 20 & 19 & 1e-6 & 1e-6 & (12, 47) & 1(14, 50) & 3.80 & 1.55\\ \hline
				
				\multicolumn{1}{ |c| }{\multirow{4}{*}{(1e4, 1e3, 5e2)}} & \multicolumn{1}{ c| }{\multirow{2}{*}{1e-4}}  & 10 & 449 & 507 & 1360 & 24 & 22 & 1e-4 & 3e-5 & (5, 27) & 1(10, 39) & 98.37 & 40.39\\ \cline{3-14}
				& & 100 & 962 & 1086 & 3548 & 25 & 23 & 1e-4 & 1e-4 & (7, 31) & 1(11, 40) & 107.76 & 37.84\\ \cline{2-14}
				& \multicolumn{1}{ c| }{\multirow{2}{*}{1e-6}} & 10 & 449 & 507 & 1360 & 21 & 21 & 1e-6 & 4e-7 & (15, 57) & 1(21, 66) & 211.74 & 69.51\\ \cline{3-14}
				& & 100 & 962 & 1086 & 3550 & 21 & 21 & 1e-6 & 1e-6 & (20, 71) & 1(24, 73) & 244.64 & 64.56\\ \hline
				
				\multicolumn{1}{ |c| }{\multirow{4}{*}{(1e5, 5e1, 1e2)}} & \multicolumn{1}{ c| }{\multirow{2}{*}{1e-4}}  & 10 & 17750 & 19285 & 44278 & 105 & 76 & 1e-4 & 4e-5 & (2, 38) & 1(3, 39) & 12.50 & 2.71 \\ \cline{3-14}
				& & 100 & 22697 & 24878 & 67188 & 120 & 84 & 1e-4 & 4e-5 & (2, 10) & 1(3, 14) & 3.47 & 1.48\\ \cline{2-14}
				& \multicolumn{1}{ c| }{\multirow{2}{*}{1e-6}} & 10 & 17753 & 19287 & 44268 & 32 & 30 & 1e-6 & 6e-7 & (14, 70) & 1(17, 85) & 23.29 & 6.35\\ \cline{3-14}
				& & 100 & 22700 & 24879 & 67180 & 27 & 24 & 1e-6 & 8e-7 & (15, 51) & 1(19, 60) & 16.93 & 5.21\\ \hline
				
				\multicolumn{1}{ |c| }{\multirow{4}{*}{(1e6, 5e1, 1e2)}}  & \multicolumn{1}{ c| }{\multirow{2}{*}{1e-4}} & 10 & 276208 & 303668 & 485124 & 1363 & 984 & 1e-4 & 5e-5 & (2, 24) & 1(3, 28) & 76.63 & 34.20\\ \cline{3-14}
				& & 100 & 291663 & 321628 & 565727 & 1254 & 860 & 1e-4 & 5e-5 & (2, 15) & 1(3, 18) & 57.21 & 34.15\\ \cline{2-14}
				& \multicolumn{1}{ c| }{\multirow{2}{*}{1e-6}} & 10 & 276231 & 303681 & 485006 & 413 & 395 & 1e-6 & 6e-7 & (14, 67) & 1(17, 76) & 209.41 & 78.54
				\\ \cline{3-14}
				& & 100 & 291677 & 321643 & 565653 & 234 & 214 & 1e-6 & 7e-7 & (15, 60) & 1(18, 67) & 197.75 & 82.87\\ \hline
			\end{tabular}
		}
	\end{table}

	
	\begin{figure}[http]
		\centering
		\includegraphics[width=1.0\textwidth]{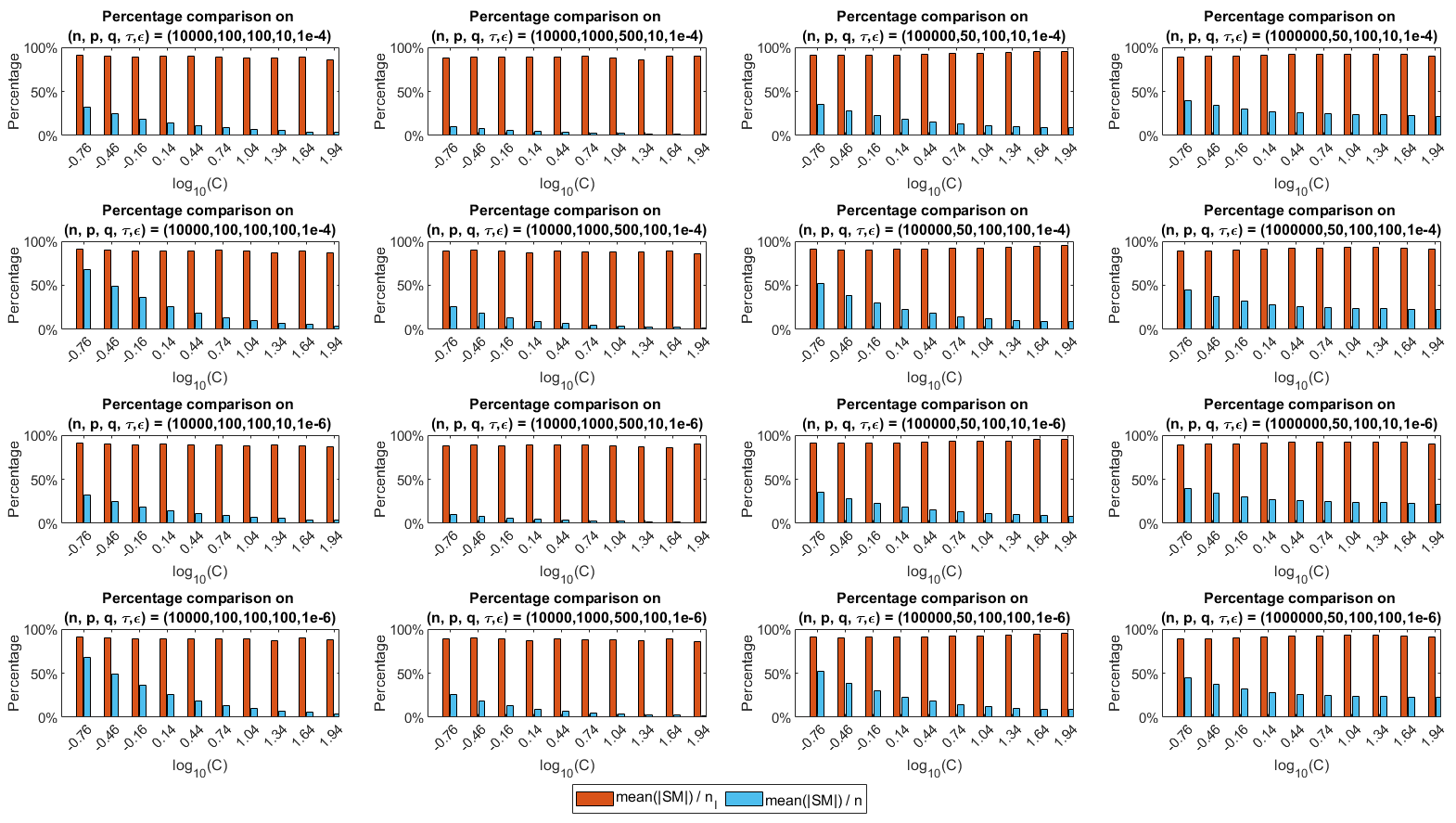}
		\caption{Comparison of support matrices percentage between reduced subproblems ($\mbox{mean}(|\mbox{SM}|/n_{\scriptscriptstyle\I}$)) and original problems ($\mbox{mean}(|\mbox{SM}|)/n$) on random data under $\eta_{kkt}\leq \varepsilon$}
		\label{fig:AS_vs_Warm_path_Random_Percentage_1e_4_6}
	\end{figure}

	\subsubsection{Numerical results on real data}
	Lastly, we compare AS+ALM and Warm+ALM in generating the entire solution path for SMM models (\ref{Model:SMM_eq2}) on large-scale CIFAR-10 and MNIST training datasets. 
	We vary $C$ from $10^{-3}$ to 1 for the CIFAR-10 and from $10^{-1}$ to $10^2$ for the MNIST, using 50 equally spaced grid points on a $\log_{10}$ scale.
	Similarly to the previous subsection, Table \ref{tab_AS_solution_path MNIST} and Figure \ref{fig:AS_vs_Warm_path_real_Percentage_1e_4_6} demonstrate that the AS strategy reduces the overall sample size of SMM models, thereby enhancing computational efficiency compared to the warm-starting strategy. Particularly, on the MNIST dataset, AS+ALM achieves average speedup factors of 1.50 and 1.92 over Warm+ALM for $\eta_{kkt}\leq 10^{-4}$ and $\eta_{kkt}\leq 10^{-6}$, respectively.

	\begin{table}
		\caption{Results for generating solution paths using Warm+ALM (Warm) and AS+ALM (AS) with a sequence $\{C_i\}^N_{i=1}$ and $\eta_{kkt}\leq\varepsilon$ on real data}
		\centering
		\label{tab_AS_solution_path MNIST}
		\resizebox{\linewidth}{!}{
			\begin{tabular}{|c|c|c|c|c|c|c|c|c|c|c|c|c|c|}
				\hline
				{Prob} & \multicolumn{1}{ c| }{\multirow{2}{*}{$\varepsilon$}} &
				\multicolumn{1}{ c| }{\multirow{2}{*}{$\tau$}} &
				\multicolumn{3}{ c| }{\multirow{1}{*}{information of the AS}} &
				\multicolumn{2}{ c| }{\multirow{1}{*}{Avg$\underline{~}|\J_1|$}} &
				\multicolumn{2}{ c| }{\multirow{1}{*}{Worst relkkt}} &
				\multicolumn{2}{ c| }{\multirow{1}{*}{Iteration numbers}} &
				\multicolumn{2}{ c| }{\multirow{1}{*}{Avg$\underline{~}$time (seconds)}} \\ \cline{4-14}
				$(n, p, q)$ & & & Avg$\underline{~}$nSM & Avg$\underline{~}$sam & Max$\underline{~}$sam & Warm & AS & Warm & AS & Warm & AS & Warm & AS \\ \hline
				\multicolumn{1}{ |c| }{\multirow{4}{*}{$\begin{array}{c} \mbox{CIFAR-10}\\ \mbox{(1e4, 32, 32)} \end{array}$}}  & \multicolumn{1}{ c| }{\multirow{2}{*}{1e-4}} & 1 & 4302 & 5709 & 7045 & 242 & 242 & 1e-4 & 1e-4 & (11, 52) & 1(12, 52) & 0.78 & 0.59\\ \cline{3-14}
				&  & 10 & 5073 & 6580 & 9661 & 86 & 86 & 1e-4 & 1e-4 & (15, 110) & 1(15, 109) & 1.41 & 1.05	\\ \cline{2-14}
				& \multicolumn{1}{ c| }{\multirow{2}{*}{1e-6}} & 1 & 4300 & 5709 & 7045 & 240 & 240 & 1e-6 & 1e-6 & (22, 72) & 1(22, 71) & 1.20 & 0.87\\ \cline{3-14}
				&  & 10 & 5073 & 6580 & 9661 & 85 & 85 & 1e-6 & 1e-6 & (26, 138) & 1(26, 138) & 1.73 & 1.30\\ \hline
				
				\multicolumn{1}{ |c| }{\multirow{4}{*}{$\begin{array}{c} \mbox{MNIST}\\ \mbox{(6e4, 28, 28)} \end{array}$}} & \multicolumn{1}{ c| }{\multirow{2}{*}{1e-4}} & 1 & 1063 & 1924 & 2403 & 443 & 442 & 1e-4 & 1e-4 & (5, 73) & 1(5, 75) & 4.16 & 2.77\\ \cline{3-14}
				&  & 10 & 1188 & 2130 & 3316 & 352 & 352 & 1e-4 & 1e-4 & (9, 96) & 1(9, 98) & 5.14 & 3.41\\ \cline{2-14}
				& \multicolumn{1}{ c| }{\multirow{2}{*}{1e-6}} & 1 & 1064 & 1924 & 2404 & 441 & 440 & 1e-6 & 1e-6 & (12, 221) & 1(13, 181) & 10.37 & 5.10 \\ \cline{3-14}
				&  & 10 & 1188 & 2130 & 3318 & 350 & 350 & 1e-6 & 1e-6 & (16, 157) & 1(17, 155) & 7.36 & 4.07\\ \hline
		\end{tabular}}
	\end{table}


	\begin{figure}[http]
		\centering
		\includegraphics[width=0.95\textwidth]{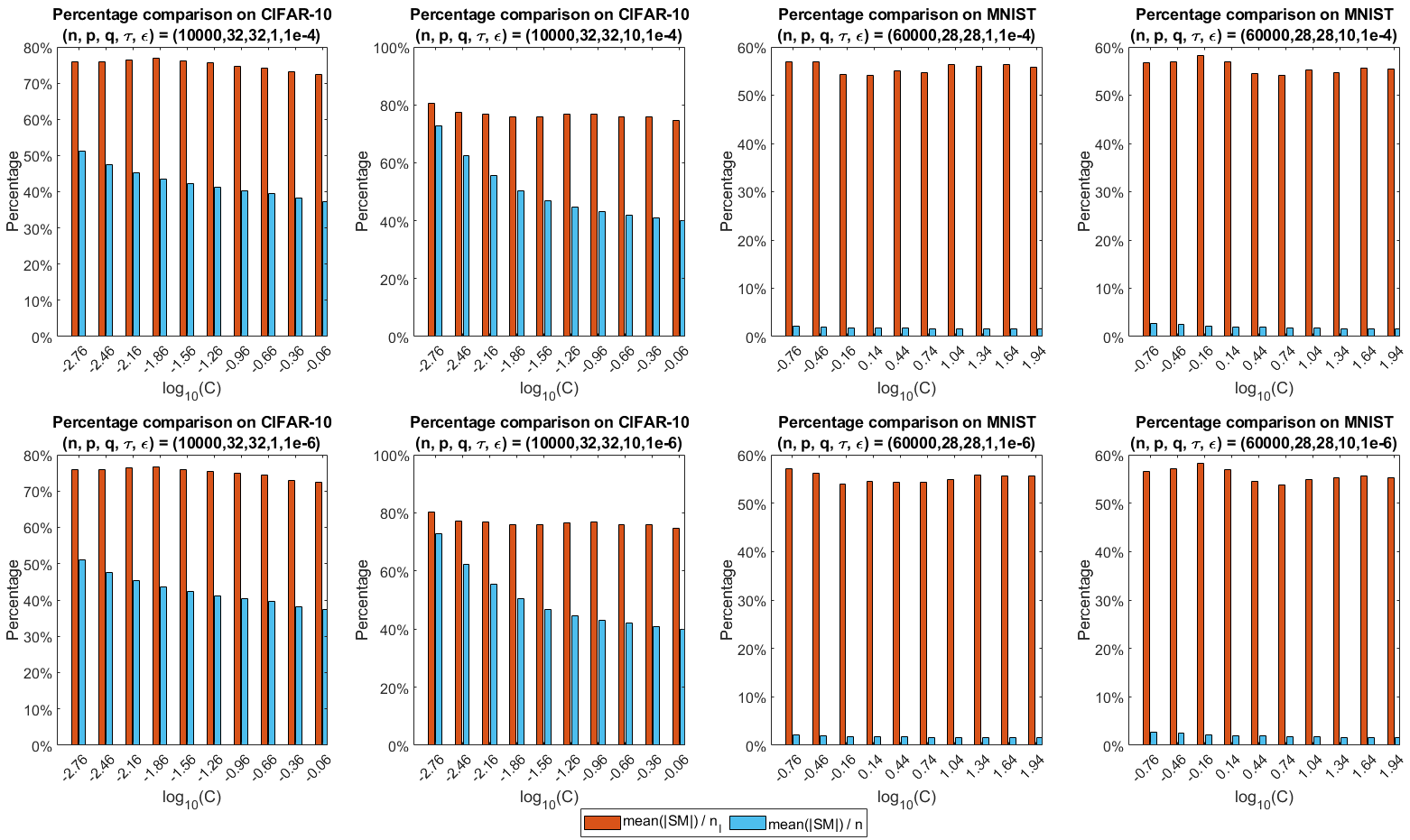}
		\caption{Comparison of support matrices percentage between reduced subproblems ($\mbox{mean}(|\mbox{SM}|/n_{\scriptscriptstyle\I})$) and original problems ($\mbox{mean}(|\mbox{SM}|)/n$) on real data under $\eta_{kkt}\leq \varepsilon$}
		\label{fig:AS_vs_Warm_path_real_Percentage_1e_4_6}
	\end{figure}

	\section{Conclusion}
	
	In this paper, we have proposed a semismooth Newton-CG based augmented Lagrangian method for solving the large-scale support matrix machine (SMM) model. Our algorithm effectively leverages the sparse and low-rank structures inherent in the second-order information associated with the SMM model. Furthermore, we have developed an adaptive sieving (AS) strategy aimed at iteratively guessing and adjusting active samples, facilitating the rapid generation of solution paths across a fine grid of parameters $C$ for the SMM models. Numerical results have convincingly demonstrated the superior efficiency and robustness of our algorithm and strategy in solving large-scale SMM models. Nevertheless, the current AS strategy focuses on reducing the sample size of subproblems. The further research direction is to explore methods for simultaneously reducing both sample size and the dimensionality of feature matrices. This may be achieved through the development of effective combinations of the active subspace selection \cite{hsieh2014nuclear,feng2022subspace} and the AS strategy.

%

\appendix
\newpage
\noindent{\bf\LARGE Appendices}
\section{Proof of Proposition \ref{Prop:sufficient_condition_quad_growth_cond}}\label{Appendix:Prof_Prop_sufficient_condition}
\begin{proof}
To proceed, we characterize the set $\Omega_D$.
It is not difficult to show that the KKT system (\ref{eq:KKT_system_SMM}) is equivalent to the following system:
\begin{equation}\label{eq:KKT_system_blamLAm}
	\left\{
	\begin{array}{l}
		0=y^{\top}\lambda,\\
		0\in by-\A(\A^*\lambda+\Lambda)-e_n+\partial\delta_S(-\lambda), \qquad \qquad (b,\lambda,\Lambda)\in\mathbb{R}\times\mathcal{Y}. \tag{A.1}\\
		0\in\A^*\lambda+\Lambda+\partial\delta_{\mathbb{B}^{\tau}_2}(\Lambda),\\
	\end{array}
	\right.
\end{equation}
Indeed, if $(\overline{W},\overline{b},\overline{v},\overline{U},\overline{\lambda},\overline{\Lambda})\in\mathcal{X}\times\mathcal{Y}$ satisfies  (\ref{eq:KKT_system_SMM}), then $(\overline{b},\overline{\lambda},\overline{\Lambda})$ solves (\ref{eq:KKT_system_blamLAm}).
Conversely, if $(\overline{b},\overline{\lambda},\overline{\Lambda})\in\mathbb{R}\times\mathcal{Y}$ solves (\ref{eq:KKT_system_blamLAm}), then
$(\overline{W},\overline{b},\overline{v},\overline{U},\overline{\lambda},\overline{\Lambda})$ with $\overline{W}=-\A^*\overline{\lambda}-\overline{\Lambda}$, $\overline{U}=\overline{W}$, and $\overline{v}=e_n-\A\overline{W}-\overline{b}y$
satisfies the KKT system (\ref{eq:KKT_system_SMM}).
For any pair $\overline{z}:=(\overline{\lambda},\overline{\Lambda})\in\Omega_D$,  $\overline{\Upsilon}=\A^*\overline{\lambda}+\overline{\Lambda}$ is an invariant (see, e.g., \cite{mangasarian1988simple}). We then define $\overline{\eta}:=\A(\overline{\Upsilon})+e_n$.
Based on the arguments preceding Subsection \ref{subsec:sample_sparsity} and the equivalence between (\ref{eq:KKT_system_SMM}) and (\ref{eq:KKT_system_blamLAm}), we claim that for any $\overline{b}\in\M_{D}(\overline{z})$, the set $\Omega_D$ can be expressed as
\begin{equation}\label{eq:Omega_D2}
	\Omega_{D}=\V_{D}\cap\G^1_{D}\cap\G^2_{D}(\overline{b}), \tag{A.2}
\end{equation}
where
\begin{align*}
	\M_{D}(\overline{z}):=\{b\in\mathbb{R}\ |\ (b, \overline{\lambda}, \overline{\Lambda})\ \mbox{satisfies}\ (\ref{eq:KKT_system_blamLAm})\},\ \G^1_{D}:=\left\{(\lambda,\Lambda)\in\mathcal{Y}\ |\ y^{\top}\lambda=0\right\},\\
	\G^2_{D}(\overline{b}):=\left(-\partial\delta^*_S(\overline{\eta}-\overline{b}y)\right)\times\partial \tau\|-\overline{\Upsilon}\|_*,\
	\V_{D}:=\left\{(\lambda,\Lambda)\in\mathcal{Y}\ |\ \A^*\lambda+\Lambda=\overline{\Upsilon}\right\}.
\end{align*}

Building on the above preparation, we proceed to establish the conclusion of Proposition \ref{Prop:sufficient_condition_quad_growth_cond}.
	We first show that for any $((-\overline{\lambda},\overline{\Lambda}),(\overline{\xi},\overline{\Xi}))\in\mbox{gph}\,\partial\delta_{S\times\mathbb{B}^{\tau}_2}$, $\partial\delta_{S\times\mathbb{B}^{\tau}_2}$ is metrically subregular at $(-\overline{\lambda},\overline{\Lambda})$ for $(\overline{\xi},\overline{\Xi})$, i.e., there exist a constant $\kappa>0$ and a neighborhood $\U\subseteq\Y$ of $(-\overline{\lambda},\overline{\Lambda})$ such that
	$$
	\mbox{dist}\left( (\lambda,\Lambda), (\partial\delta_{S\times\mathbb{B}^\tau_2})^{-1}(\overline{\xi},\overline{\Xi}) \right)\leq\kappa\mbox{dist}\left( (\overline{\xi},\overline{\Xi}), \partial\delta_{S\times\mathbb{B}^{\tau}_2}(\lambda, \Lambda) \right), \quad \forall\, (\lambda,\Lambda)\in\U.
	$$
	Here, $\mbox{dist}(x, D):=\min_{d\in D}\|d-x\|$ denotes the distance from $x$ to a closed convex set $D$ in a finite dimensional real Euclidean space $\Z$, and $\mbox{gph}\,F:=\{(x,y)\in\Z\times\Z\ |\ y\in F(x)\}$ represents the graph of a multi-valued mapping $F$ from $\Z$ to $\Z$.
	In fact,
	it follows from the piecewise linearity of $\delta^*_S(\cdot)$ on $\mathbb{R}^n$ and \cite[Proposition 12.30]{rockafellar2009variational} that $\partial\delta^*_S(\cdot)$ is piecewise polyhedral. Then, for any $(\overline{\xi},-\overline{\lambda})\in\mbox{gph}\,\partial\delta^*_S$, we obtain from \cite[Proposition 1]{robinson1981some} that $\partial\delta^*_S(\cdot)$ is locally upper Lipschitz continuous at $\overline{\xi}$, which further implies the metric subregularity of $\partial\delta_S$ at $-\overline{\lambda}$ for $\overline{\xi}$.
	Let $\mathbb{B}^{\tau}_1\subset\mathbb{R}^p$ be the $\ell_1$-norm ball centered at 0 with radius $\tau$.
	Similarly, for any $(v,\hat{v})\in\mbox{gph}\,\partial\delta_{\mathbb{B}^{\tau}_1}$, $\partial\delta_{\mathbb{B}^{\tau}_1}$ is metrically subregular at $v$ for $\hat{v}$ due to the piecewise linearity of $\tau\|\cdot\|_{\infty}$ on $\mathbb{R}^p$.
	For any given $(\overline{\Lambda},\overline{\Xi})\in\mbox{gph}\,\partial\delta_{\mathbb{B}^{\tau}_2}$, it follows from  \cite[Proposition 3.8]{cui2017quadratic} that $\partial\delta_{\mathbb{B}^{\tau}_2}$ is also metrically subregular at $\overline{\Lambda}$ for $\overline{\Xi}$. Thus, there exist positive constants $\kappa_1$ and $\kappa_2$ and a neighborhood $\U$ of $(-\overline{\lambda},\overline{\Lambda})$ such that for any $(\lambda,\Lambda)\in\U$,
	$$
	\begin{aligned}
		\mbox{dist}\left( (\lambda,\Lambda), (\partial\delta_{S\times\mathbb{B}^{\tau}_2}){}^{-1}(\overline{\xi},\overline{\Xi})\right)
		&\leq \mbox{dist}\left( \lambda, (\partial\delta_S)^{-1}(\overline{\xi}) \right)+\mbox{dist}\left(\Lambda, (\partial\delta_{\mathbb{B}^{\tau}_2})^{-1}(\overline{\Xi})\right)\\
		&\leq\kappa_1\mbox{dist}\left(\overline{\xi},\partial\delta_S(\lambda)\right)+\kappa_2\mbox{dist}\left(\overline{\Xi}, \partial\delta_{\mathbb{B}^\tau_2}(\Lambda)\right)\\
		&\leq\max\{\kappa_1,\kappa_2\}\left( \mbox{dist}(\overline{\xi}, \partial\delta_S(\lambda))+\dist(\overline{\Xi}, \partial\delta_{\mathbb{B}^{\tau}_2}(\Lambda)) \right)\\
		&\leq\sqrt{2}\max\{\kappa_1,\kappa_2\}\mbox{dist}\left( (\overline{\xi},\overline{\Xi}),\partial\delta_{S\times\mathbb{B}^{\tau}_2}(\lambda,\Lambda) \right),
	\end{aligned}
	$$
    where the second inequality follows from the metric subregularity of $\partial\delta_S$ at $-\overline{\lambda}$ for $\overline{\xi}$ and the metric subregularity of
	$\partial\delta_{\mathbb{B}^{\tau}_2}$ at $\overline{\Lambda}$ for $\overline{\Xi}$.

	Next,  for any $\overline{b}\in\M_{D}(\overline{z})$, we prove that the collection of sets $\{\V_{D},\G^1_{D},\G^2_{D}(\overline{b})\}$ is boundedly linearly regular, i.e., for every bounded set $B\subset\mathcal{Y}$, there exists a constant $\kappa'>0$ such that for any $(\lambda, \Lambda)\in B$,
	$$
	\mbox{dist}\left((\lambda, \Lambda), \V_{D}\cap\G^1_{D}\cap\G^2_{D}(\overline{b})\right)\leq\kappa'\max\left\{\mbox{dist}((\lambda,\Lambda), \V_{D}), \mbox{dist}((\lambda,\Lambda), \G^1_{D}), \mbox{dist}((\lambda,\Lambda), \G^2_{D}(\overline{b}))\right\}.
	$$
	In fact, based on the arguments preceding Subsection \ref{subsec:sample_sparsity} and (\ref{eq:Omega_D2}), the intersection $\V_{D}\cap\G^1_{D}\cap\G^2_{D}(\overline{b})$ is nonempty.
	Define $\G^{21}_{D}(\overline{b}):=\left(-\partial\delta^*_{S}(\overline{\eta}-\overline{b}y)\right)\times\mathbb{R}^{p\times q}$ and $\G^{22}_{D}:=\mathbb{R}^n\times\partial \tau\|-\overline{\Upsilon}\|_*$. It follows from the equality (\ref{eq:Omega_D2}) that
	\begin{equation}\label{eq:Omega_D2_4}
		\Omega_{D}=\V_{D}\cap\G^1_{D}\cap\G^{21}_{D}(\overline{b})\cap\G^{22}_{D}.\tag{A.3}
	\end{equation}
	It can be checked that $\{\V_{D},\G^1_{D},\G^2_{D}(\overline{b})\}$ is boundedly linearly regular if and only if $\{\V_{D},\G^1_{D}$, $\G^{21}_{D}(\overline{b}),\G^{22}_{D}\}$ is boundedly linearly regular.
	Furthermore, from \cite[Corollary 19.2.1 and Theorem 23.10]{rockafellar1970convex},  $\partial\delta^*_S(\overline{\eta}-\overline{b}y)$ is a polyhedral convex set, implying that  $\G^{21}_D(\overline{b})$ is also polyhedral convex.
	Since $\V_{D}$ and $\G^1_{D}$ are polyhedral convex, based on \cite[Corollary 3]{bauschke1999strong} and (\ref{eq:Omega_D2_4}), we only need to show that there exists $(\widetilde{\lambda},\widetilde{\Lambda})\in\Omega_{D}$ such that $(\widetilde{\Lambda},\widetilde{\Lambda})\in\mbox{ri}\,(\G^{22}_{D})$, i.e., $\widetilde{\Lambda}\in\mbox{ri}\left( \partial\tau\|-\overline{\Upsilon}\|_* \right)$.
	Let $\mbox{rank}(\widetilde{\Lambda})=\overline{r}$.
	It follows from \cite[Example 2]{watson1992characterization} that $\|\widetilde{\Lambda}\|_2\leq\tau$.
	Suppose the SVD of $\widetilde{\Lambda}$ in (\ref{eq:svd_Lambda_simple}) has the following form:
	\begin{equation}\label{eq:SVD_Lambda}
		\widetilde{\Lambda}=\overline{U}\left[ \overline{\Sigma}(\widetilde{\Lambda})\ \, 0 \right]\overline{V}^{\top}=[\overline{U}_{1}\ \,\overline{U}_{(0,1)}\ \, \overline{U}_0]\left[
		\begin{array}{ccc}
			\tau\I_{\overline{s}} & 0 & 0\\[2mm]
			0 & \tau\overline{\Sigma}_{(0,1)}(\widetilde{\Lambda})  & 0\\
			0 & 0 & 0
		\end{array}
		\right]\left[
		\begin{array}{l}
			\overline{V}^{\top}_1\\[1mm]
			\overline{V}^{\top}_{(0,1)}\\[1mm]
			\overline{V}^{\top}_{0}
		\end{array}
		\right], \tag{A.4}
	\end{equation}
	where $\overline{\Sigma}(\widetilde{\Lambda})=\mbox{Diag}\left( \nu_1(\widetilde{\Lambda}),\ldots,\nu_p(\widetilde{\Lambda}) \right)\in\mathcal{S}^p_{+}$ and $\overline{\Sigma}_{(0,1)}(\widetilde{\Lambda})=\mbox{Diag}(\nu_{\,\overline{s}+1}(\widetilde{\Lambda})$, $\ldots,\nu_{\,\overline{r}}(\widetilde{\Lambda}))\in\mathcal{S}^{\,\overline{r}-\overline{s}}_{++}$ with singular values
	$1=\nu_1(\widetilde{\Lambda})=\ldots=\nu_{\,\overline{s}}(\widetilde{\Lambda})>\nu_{\,\overline{s}+1}(\widetilde{\Lambda})\geq\ldots\geq\nu_{\,\overline{r}}(\widetilde{\Lambda})>\nu_{\,\overline{r}+1}(\widetilde{\Lambda})=\ldots=\nu_p(\widetilde{\Lambda})=0$, and $\overline{U}:=\left[ \overline{U}_{1}\ \overline{U}_{(0,1)}\ \overline{U}_0 \right]\in\mathcal{O}^p$ and $\overline{V}:=\left[ \overline{V}_1\ \overline{V}_{(0,1)}\ \overline{V}_0 \right]\in\mathcal{O}^q$ whose columns form a compatible set of orthonormal left and right singular vectors of $\widetilde{\Lambda}$ with $\overline{U}_1\in\mathbb{R}^{p\times\overline{s}}$,
	$\overline{U}_{(0,1)}\in\mathbb{R}^{p\times(\overline{r}-\overline{s})}$,
	$\overline{U}_0\in\mathbb{R}^{p\times(p-\overline{r})}$, $\overline{V}_1\in\mathbb{R}^{q\times\overline{s}}$,
	$\overline{V}_{(0,1)}\in\mathbb{R}^{q\times(\overline{r}-\overline{s})}$, and $\overline{V}_0\in\mathbb{R}^{q\times(q-\overline{r})}$.
	Based on \cite[Example 2]{watson1992characterization} and (\ref{eq:SVD_Lambda}), one has that $\widetilde{\Lambda}\in\mbox{ri}(\partial\tau\|-\overline{\Upsilon}\|_*)$ if and only if $\mbox{rank}(-\overline{\Upsilon})=\overline{s}$.
	One the other hand, from the definition of the matrix $\overline{I}$ in (\ref{eq:sufficient_condition_quatratic_growth}), we obtain that
	\begin{equation}\label{eq:tauI_tildeLambda}
		\tau\overline{I}-\widetilde{\Lambda}=\left[ \overline{U}_1\ \, \overline{U}_{(0,1)}\ \, \overline{U}_0\right]\left[
		\begin{array}{cccc}
			0 & 0 & 0 & 0\\[2mm]
			0 & \tau(I_{\overline{r}-\overline{s}}-\overline{\Sigma}_{(0,1)}(\widetilde{\Lambda})) & 0 & 0\\[2mm]
			0 & 0 & I_{p-\overline{r}} & 0
		\end{array}
		\right]\left[
		\begin{array}{l}
			\overline{V}^{\top}_1\\[1mm]
			\overline{V}^{\top}_{(0,1)}\\[1mm]
			\overline{V}^{\top}_0
		\end{array}
		\right],\tag{A.5}
	\end{equation}
	which implies that $\mbox{rank}(\tau\overline{I}-\widetilde{\Lambda})=p-\overline{s}$.
	It means that $\mbox{rank}(-\overline{\Upsilon})=\overline{s}$ if and only if $p=\mbox{rank}(-\overline{\Upsilon})+\mbox{rank}(\tau\overline{I}-\widetilde{\Lambda})$. Therefore, if there exists $(\widetilde{\lambda},\widetilde{\Lambda})\in\Omega_D$ such that $p=\mbox{rank}(-\overline{\Upsilon})+\mbox{rank}(\tau\overline{I}-\widetilde{\Lambda})$, then $\{\V_{D},\G^1_{D},\G^2_{D}(\overline{b})\}$ is boundedly linearly regular.
	
	Based on all the above analysis, the desired conclusion can be established utilizing \cite[Theorem 3.1]{cui2016asymptotic} and \cite[Theorem 3.3]{artacho2008characterization}. \qed
\end{proof}

\section{Proof of Theorem \ref{Thm:constraint_nondegenerate}}\label{Appendix:Prof_constraint_nondegenrate}
\begin{proof}
	Building upon a technique similar to \cite[Proposition 3.1]{li2018qsdpnal}, we obtain the following equivalent relation:
	\begin{equation}\label{eq:equivalence_yTMy}
		\V\ \mbox{is self-adjoint and positive definite in }\ \mathbb{R}^{p\times q}\times\mathbb{R} \quad \Longleftrightarrow\quad y^{\top}My>0, \tag{B.1}
	\end{equation}
	where $\V$ is any element in $\widehat{\partial}^2\varphi_k(\widehat{W},\widehat{b})$ defined in (\ref{eq:Hession_phi}) with $M\in\partial\Pi_S(\omega_k(\widehat{W}, \widehat{b}))$. Here, for any $\omega\in\mathbb{R}^n$, the Clarke generalized Jacobian $\partial\Pi_S(\omega)$ can be expressed as
	\begin{equation}\label{eq:Jacobian_M}
		\partial\Pi_S(\omega)=\left\{\mbox{Diag}(v)\ \left|\
		\begin{array}{l}
		 v_j=1,\ \mbox{if}\ \omega_j\in(0,C);\\
		 v_j\in[0,1],\ \mbox{if}\ \omega_j=0\ \mbox{or}\ \omega_j=C;\\
		 v_j=0,\ \mbox{if}\ \omega_j<0\ \mbox{or}\ \omega_j>C
		 \end{array}\right.
		 \right\}.	\tag{B.2}
	\end{equation}
	
	We consider the explicit expression of $\mbox{lin}(T_S(-\widehat{\lambda}))$. Denote the following two index sets
	$$
	\J_2(-\widehat{\lambda}):=\{j\in[n]\ |\ \widehat{\lambda}_j=0\} \quad \mbox{and}\quad \J_3(-\widehat{\lambda}):=\{j\in[n]\ |\ -\widehat{\lambda}_j=C\}.
	$$
	Then we can deduce from \cite[Theorem 6.9]{rockafellar2009variational} that
	$$
	T_S(-\widehat{\lambda})=\{d\in\mathbb{R}^n\ |\ d_{\tiny\J_2(-\widehat{\lambda})}\geq0,\ d_{\J_3(-\widehat{\lambda})}\leq0\}.
	$$
	It follows that
	$$
	\mbox{lin}(T_S(-\widehat{\lambda}))=\{d\in\mathbb{R}^n\ |\ d_{\J_2(-\widehat{\lambda})}=0,\ d_{\J_3(-\widehat{\lambda})}=0\}.
	$$
	
	``$(i)\Leftrightarrow(ii)$". The equality (\ref{eq:constraint_nondegenerate}) at $\widehat{\lambda}$ holds if and only if
	$
	\mathbb{R}=y^{\top}_{\J_1(-\widehat{\lambda})}\mathbb{R}^{|\J_1(-\widehat{\lambda})|},
	$
	i.e., $\J_1(-\widehat{\lambda})\neq\emptyset$.
	
	``$(ii)\Rightarrow(iii)$". If $\J_1(-\widehat{\lambda})\neq\emptyset$, we obtain
	from $\widehat{\lambda}=-\Pi_S(\omega_k(\widehat{W},\widehat{b}))$ that there exists an index $j^0\in[n]$ such that $0<(\omega_k(\widehat{W},\widehat{b}))_{j^0}<C$. It follows from (\ref{eq:Jacobian_M}) that $y^{\top}My>0$. By the equivalent relation (\ref{eq:equivalence_yTMy}), we prove condition $(iii)$.
	
	``$(iii)\Rightarrow(ii)$". If the condition $(iii)$ holds, then by the equivalence in (\ref{eq:equivalence_yTMy}), $y^{\top}My>0$ for any $M\in\partial\Pi_S(\omega_k(\widehat{W},\widehat{b}))$. It means from (\ref{eq:Jacobian_M}) that there exists $j^0\in[n]$ such that $0<(\omega_k(\widehat{W},\widehat{b}))_{j^0}<C$. From $\widehat{\lambda}=-\Pi_S(\omega_k(\widehat{W},\widehat{b}))$, it follows that $0<(-\widehat{\lambda})_{j^0}<C$, implying condition $(ii)$. \qed
\end{proof}

\section{Proof of Theorem \ref{Thm:AS_convergence}}\label{Appendix:Theorem_AS}
\begin{proof}
	For each given $C_i$, the cardinality of $\overline{\I}^{\,k}(C_i)$ will decrease as that of $\I^k(C_i)$ increases. According to the finiteness of samples size $n$, the index set $\J^k(C_i)$ as a subset of $\overline{\I}^k(C_i)$ will be empty after a finite number of iterations, ensuring finite termination.
	
	Next, we show that the output solution $(\overline{W}(C_i), \overline{b}(C_i), \overline{v}(C_i), \overline{U}(C_i), \overline{\lambda}(C_i), \overline{\Lambda}(C_i))$ of Algorithm \ref{alg_AS} is a KKT tuple of the problem (\ref{eq:AS_whole_model}). Indeed, by the finite termination of Algorithm \ref{alg_AS}, we know that there exists an integer $\bar{k}\in[n]$ such that $(\overline{W}(C_i), \overline{b}(C_i), {v}^{\,\bar{k}}_{\scriptscriptstyle\I}(C_i), \overline{U}(C_i), {\lambda}^{\,\bar{k}}_{\scriptscriptstyle\I}(C_i)$, $\overline{\Lambda}(C_i))$ is a KKT solution of the reduced subproblem (\ref{model:AS_subpro}) with the error $(\delta_W, \delta_b, \delta_{v_{\scriptscriptstyle\I}}, \delta_{U}, \delta_{\lambda_{\scriptscriptstyle\I}}$, $\delta_{\Lambda})$ satisfying the condition (\ref{eq:error_AS}) and $\J^{\,\bar{k}}(C_i)=\emptyset$. And $\overline{v}(C_i)$ and $\overline{\lambda}(C_i)$ are obtained by expanding ${v}^{\,\bar{k}}_{\scriptscriptstyle\I}(C_i)$ and ${\lambda}^{\,\bar{k}}_{\scriptscriptstyle\I}(C_i)$ to the $n$-dimensional vectors with the rest entries being $1-y_j(\langle \overline{W}(C_i),X_j \rangle+\overline{b}(C_i))$ and $0$ for any $j\in\overline{\I}^{\,\bar{k}}(C_i)$, respectively.
	
    For the sake of convenience in further analysis, we will exclude $C_i$ from the above KKT solutions and index sets. Additionally, we will eliminate $\bar{k}$ from $\I^{\,\bar{k}}$, $\overline{\I}^{\,\bar{k}}$, ${\lambda}^{\,\bar{k}}_{\scriptscriptstyle\I}$, and ${v}^{\,\bar{k}}_{\scriptscriptstyle\I}$.
	By the KKT conditions of the reduced subproblem (\ref{model:AS_subpro}) and the condition (\ref{eq:error_AS}), one has that
	\begin{equation}\label{eq:KKT_AS_reduced}
		\begin{array}{c}
			\delta_W=\overline{W}+\A^*_{\scriptscriptstyle\I}{\lambda}_{\scriptscriptstyle\I}+\overline{\Lambda},\ \delta_b=y_{\scriptscriptstyle\I}^\top{\lambda}_{\scriptscriptstyle\I},\ \delta_{\Lambda}=\overline{W}-\overline{U},\ \delta_{v_{\scriptscriptstyle\I}}-{\lambda}_{\scriptscriptstyle\I}\in\partial\delta^*_{S^{\,i}_{\scriptscriptstyle\I}}({v}_{\scriptscriptstyle\I}),\ \delta_U+\overline{\Lambda}\in\partial\tau\|\overline{U}\|_*,\\[2mm]
			\delta_{\lambda_{\scriptscriptstyle\I}}=\A_{\scriptscriptstyle\I}\overline{W}+\overline{b}y_{\scriptscriptstyle\I}+{v}_{\scriptscriptstyle\I}-e_{|\I|},\
			\max(\|\delta_W\|, |\delta_b|, \|\delta_{U}\|, \|\delta_{\lambda_{\scriptscriptstyle\I}}\|, \|\delta_{v_{\I}}\|, \|\delta_{\Lambda}\|)\leq\varepsilon,
		\end{array}\tag{C.1}
	\end{equation}
	where $\A_{\scriptscriptstyle\I}$ and $\A^*_{\scriptscriptstyle\I}$ are defined in (\ref{eq:def_AI}).
	By extending $\delta_{\lambda_{\scriptscriptstyle\I}}$ and $\delta_{v_{\scriptscriptstyle\I}}$ to the $n$-dimensional error vectors $\delta_{\lambda}$ and $\delta_{v}$ with the rest elements being $0$, we only need to show that
	$$
	\delta_{W}=\overline{W}+\A^*\overline{\lambda}+\overline{\Lambda},\quad \delta_b=y^{\top}\overline{\lambda},\quad \delta_{\lambda}=\A\overline{W}+\overline{b}y+\overline{v}-e_n, \quad \delta_v-\overline{\lambda}\in\partial\delta^*_{S^{\,i}}(\overline{v}).
	$$
	Indeed, due to the fact that $(\overline{\lambda})_{\scriptscriptstyle\overline{\I}}=0$, we obtain
	$$
	\A^*\overline{\lambda}=\A^*_{\scriptscriptstyle\I}{\lambda}_{\scriptscriptstyle\I}+\A^*_{\scriptscriptstyle\overline{\I}}(\overline{\lambda})_{\scriptscriptstyle\overline{\I}}=\A^*_{\scriptscriptstyle\I}{\lambda}_{\scriptscriptstyle\I}\quad \mbox{and}\quad y^{\top}\overline{\lambda}=y_{\scriptscriptstyle\I}^{\top}{\lambda}_{\scriptscriptstyle\I}+y_{\scriptscriptstyle\overline{\I}}^{\top}(\overline{\lambda})_{\scriptscriptstyle\overline{\I}}=y_{\scriptscriptstyle\I}^{\top}\overline{\lambda}_{\scriptscriptstyle\I},
	$$
	which follows from (\ref{eq:KKT_AS_reduced}) that $\delta_W=\overline{W}+\A^*\overline{\lambda}+\overline{\Lambda}$ and $\delta_b=\overline{\lambda}^{\top}y$.
	From expanding manners of $\overline{v}$ and $\delta_{\lambda}$, we know that ${v}_{\scriptscriptstyle\overline{\I}}=e_{\scriptscriptstyle\overline{\I}}-\A_{\scriptscriptstyle\overline{\I}}\overline{W}-\overline{b}y_{\scriptscriptstyle\overline{\I}}$, which further implies from (\ref{eq:KKT_AS_reduced}) that $\delta_{\lambda}=\A\overline{W}+\overline{b}y+\overline{v}-e_{n}$.
	At last, from $\J^{\bar{k}}(C_i)=\emptyset$, one has $\overline{v}_j<0$ for all $j\in\overline{\I}$. It follows that $(\delta_v)_{\scriptscriptstyle\overline{\I}}-(\overline{\lambda})_{\scriptscriptstyle\overline{\I}}=0=\nabla\delta^*_{S^{\,i}_{\scriptscriptstyle\overline{\I}}}({v}_{\scriptscriptstyle\overline{\I}})$. Combing with (\ref{eq:KKT_AS_reduced}), we obtain that $\delta_v-\overline{\lambda}\in\partial\delta^*_{S^{i}}(\overline{v})$. \qed
\end{proof}

\section{An isPADMM for solving SMM model }\label{Appendex:isPADMM}
In this part, we outline the framework of the inexact semi-proximal ADMM (isPADMM) to solve the SMM model (\ref{Model:SMM}). Specifically, according to the definition of the linear operator $\A$ in (\ref{def:A_AT}), the SMM model (\ref{Model:SMM}) can be equivalently reformulated as follows:
\begin{equation}\label{Model:SMM_one_constraint}
\begin{array}{cl}
	\mathop{\mbox{minimize}}\limits_{(W,b,U)\in\mathbb{R}^{p\times q}\times\mathbb{R}\times\mathbb{R}^{p\times q}} & \displaystyle\frac{1}{2}\|W\|^2_{\F}+\tau\|U\|_*+\delta_S^*(e_n-\A W-by)\\
	\mbox{subject to} & W-U=0. \tag{D.1}
\end{array}
\end{equation}
Its augmented Lagrangian function is defined for any $(W, b, U, \Lambda)\in\mathbb{R}^{p\times q}\times\mathbb{R}\times\mathbb{R}^{p\times q}\times\mathbb{R}^{p\times q}$ as follows:
$$
L_{\gamma}(W,b,U;\Lambda)=\frac{1}{2}\|W\|^2_{\F}+\tau\|U\|_*+\delta_S^*(e_n-\A W-by)+\langle \Lambda,W-U \rangle+\frac{\gamma}{2}\|W-U\|^2_{\F},
$$
where $\gamma$ is a positive penalty parameter.
The framework of isPADMM for solving (\ref{Model:SMM_one_constraint}) is outlined in Algorithm \ref{alg_isPADMM_SMM}.

\begin{algorithm}[htbp]
	\small
	\caption{An isPADMM for solving SMM model (\ref{Model:SMM_one_constraint})}
	\vskip 0.05in
	\noindent
	{\bf Initialization:} Let $\gamma>0$ and $\zeta\in(0,(1+\sqrt{5})/2)$ be given parameters, $\{\overline{\varepsilon}_k\}_{k\geq0}$ be a nonnegative summable sequence. Choose $\delta>0$ and an initial point $(U^0,\Lambda^0)\in\mathbb{R}^{p\times q}\times\mathbb{R}^{p\times q}$. Set $k=0$ and perform the following steps in each iteration:	
	\vskip 0.1in
	\begin{algorithmic}[0]
		\State {\bf Step 1}. Compute
		\begin{align}\label{eq:subproblem_W_isPADMM}	
			(W^{\,k+1},{b}^{\,k+1})
			=&\mathop{\mbox{argmin}}_{W,b}\left\{ L_{\gamma}\big(W,b,U^k;\Lambda^k\big)+\frac{\delta}{2}(b-b^k)^2-\langle d^{\,k}_W,W \rangle-bd^{\,k}_b \right\}\nonumber\\
			=&\mathop{\mbox{argmin}}_{W,b}\left\{ \frac{1+\gamma}{2}\|W+\frac{\Lambda^k-\gamma U^k}{1+\gamma}\|^2_{\F}+\delta^*_{S}(e_n-\A W-by)\right.\nonumber\\
			&\qquad\qquad\qquad\left.+\frac{\delta}{2}(b-b^k)^2-\langle d^{\,k}_W,W \rangle-bd^{\,k}_b \right\},\tag{D.2}
		\end{align}
		where $(d^{\,k}_W, d^{\,k}_b)\in\mathbb{R}^{p\times q}\times\mathbb{R}$ is an error term such that $
		(1/\gamma)\|\delta^{\,k}_W\|^2_{\F}+(1/\delta)|d^{\,k}_b|^2\leq\overline{\varepsilon}_k$.
		\vskip 0.05in
		\State {\bf Step 2}. Compute
		\begin{align*}
			U^{k+1}&=\mathop{\mbox{argmin}}\limits_{U\in\mathbb{R}^{p\times q}}\left\{L_{\gamma}(W^{k+1},b^{k+1},U;\Lambda^k)\right\}=\frac{1}{\gamma}\mbox{Prox}_{\tau\|\cdot\|_*}(\gamma W^{k+1}+\Lambda^k).
		\end{align*}
		\vskip 0.05in
		\State {\bf Step 3}. Update
		$
		\Lambda^{k+1}=\Lambda^k+\zeta\,\gamma\left(W^{k+1}-U^{k+1} \right).
		$
	\end{algorithmic}
	\label{alg_isPADMM_SMM}
\end{algorithm}

Notably, subproblem (\ref{eq:subproblem_W_isPADMM}) closely resembles the original SMM model (\ref{Model:SMM}), except
$\tau$ is set to zero and an additional proximal term involving
$b$ is included.
This similarity allows the direct application of the semismooth Newton-based augmented Lagrangian method to these subproblems, as discussed in Sections \ref{Sec:ALM-SNCG} and \ref{Sec:linear_system}.
The derivations, not detailed here, follow a rationale similar to that of the aforementioned algorithms.
Furthermore, by leveraging the results of \cite[Theorem 5.1]{chen2017efficient}, we can deduce the global convergence of Algorithm \ref{alg_isPADMM_SMM} in a direct manner.
To ensure computational feasibility, the isPADMM iterations are capped at
30,000, with $\delta$ set to  $10^{-6}$.

\section{A sGS-isPADMM for solving SMM model}\label{Appendex:sGS-isPADMM}
In this part, we present the framework of the symmetric Gauss-Seidel based inexact semi-proximal ADMM (sGS-isPADMM) to effectively solve the SMM model (\ref{Model:SMM_eq2}). Based on the definition of the augmented Lagrangian function in (\ref{eq:augmented Lagrangin funciton}),
the steps of the sGS-isPADMM for solving (\ref{Model:SMM_eq2}) are listed as follows.
\begin{algorithm}[htbp]
	\caption{A sGS-isPADMM for solving SMM model (\ref{Model:SMM_eq2})}
	\vskip 0.05in
	\noindent
	{\bf Initialization:} Let $\gamma>0$ and $\zeta\in(0,(1+\sqrt{5})/2)$ be given parameters, $\{\overline{\varepsilon}_k\}_{k\geq0}$ be a nonnegative summable sequence. Select a self-adjoint positive semidefinite linear operator $\mathcal{S}_1:\mathbb{R}^{p\times q}\rightarrow\mathbb{R}^{p\times q}$.
	Choose an initial point $(W^0,b^0,v^0,U^0,\lambda^0,\Lambda^0)\in\mathbb{R}^{p\times q}\times\mathbb{R}\times\mathbb{R}^{n}\times\mathbb{R}^{p\times q}\times\mathbb{R}^n\times\mathbb{R}^{p\times q}$. Set $k=0$ and perform the following steps in each iteration:	
	\vskip 0.1in
	\begin{algorithmic}[1]
		\State {\bf Step 1}. Compute
		{\small\begin{align*}\label{eq:subproblem_W_sGS}
			\overline{b}^{\,k+1}&=\mathop{\mbox{argmin}}_{b\in\mathbb{R}}\left\{ L_{\gamma}\big(W^k,b,v^k,U^k;\lambda^k,\Lambda^k\big) \right\}
			=-\frac{1}{n}y^{\top}\left(\A W^k+v^k-e_n+\frac{\lambda^k}{\gamma}\right),\nonumber\\
			W^{k+1}&=\mathop{\mbox{argmin}}_{W\in\mathbb{R}^{p\times q}}\left\{ L_{\gamma}(W,\overline{b}^{k+1},v^k,U^k;\lambda^k,\Lambda^k)+\frac{1}{2}\|W-W^k\|^2_{\mathcal{S}_1}-\langle \delta^k_W,W \rangle \right\},\nonumber\\
			b^{k+1}&=\mathop{\mbox{argmin}}_{b\in\mathbb{R}}\left\{ L_{\gamma}(W^{k+1},b,v^k,U^k;\lambda^k,\Lambda^k) \right\}=-\frac{1}{n}y^{\top}\left(\A W^{k+1}+v^k-e_n+\frac{\lambda^k}{\gamma}\right),\nonumber
		\end{align*}}
		where $\delta^k_W\in\mathbb{R}^{p\times q}$ is an error matrix such that $\|\delta^k_W\|\leq\overline{\varepsilon}_k$.
		\vskip 0.05in
		\State {\bf Step 2}. Compute
		{\small\begin{align*}
			v^{k+1}&=\frac{1}{\gamma}\mbox{Prox}_{\delta^*_{S}}\left(-\lambda^k-\gamma(\A W^{k+1}+b^{k+1}y-e_n)\right),\\
			U^{k+1}&= (1/\gamma)\mbox{Prox}_{\tau\|\cdot\|_*}\left(\Lambda^k+\gamma W^{k+1}\right).	
		\end{align*}}
		\vskip 0.05in
		\State {\bf Step 3}. Update
		\begin{align*}
			\lambda^{k+1}&=\lambda^k+\zeta\,\gamma(\A W^{k+1}+b^{k+1}y+v^{k+1}-e_n),\\
			\Lambda^{k+1}&=\Lambda^k+\zeta\,\gamma( W^{k+1}-U^{k+1}).
		\end{align*}
	\end{algorithmic}
	\label{alg_isPADMM_SMM}
\end{algorithm}

Building upon the results in \cite[Proposition 4.2 and Theorem 5.1]{chen2017efficient}, we can directly obtain the global convergence results for Algorithm \ref{alg_isPADMM_SMM}.
Lastly, the maximum number of iterations for the sGS-isPADMM is set to 30,000.


%
\section*{Statements and Declarations}
{\bf Funding} The work of Defeng Sun was supported in part by RGC Senior Research Fellow Scheme No. SRFS2223-5S02 and GRF Project No. 15307822.

\medskip

\noindent{\bf Availability of data and materials} The references of all datasets are provided in this published article.

\medskip

\noindent{\bf Conflict of interest} The authors declare that they have no conflict of interest.

\nocite{*}
\bibliographystyle{spmpsci}      
\bibliography{Reference_SMM}   

\begin{thebibliography}{10}
\providecommand{\url}[1]{{#1}}
\providecommand{\urlprefix}{URL }
\expandafter\ifx\csname urlstyle\endcsname\relax
  \providecommand{\doi}[1]{DOI~\discretionary{}{}{}#1}\else
  \providecommand{\doi}{DOI~\discretionary{}{}{}\begingroup
  \urlstyle{rm}\Url}\fi

\bibitem{artacho2008characterization}
Artacho, F.A., Geoffroy, M.H.: Characterization of metric regularity of
  subdifferentials.
\newblock J. Convex Anal. \textbf{15}(2), 365--380 (2008)

\bibitem{bauschke1999strong}
Bauschke, H.H., Borwein, J.M., Li, W.: Strong conical hull intersection
  property, bounded linear regularity, {J}ameson’s property ({G}), and error
  bounds in convex optimization.
\newblock Math. Program. \textbf{86}(1), 135--160

\bibitem{chen2017efficient}
Chen, L., Sun, D.F., Toh, K.C.: An efficient inexact symmetric {Gauss-Seidel}
  based majorized {ADMM} for high-dimensional convex composite conic
  programming.
\newblock Math. Program. \textbf{161}(1), 237--270 (2017)

\bibitem{chen2020novel}
Chen, Y., Hang, W., Liang, S., Liu, X., Li, G., Wang, Q., Qin, J., Choi, K.S.:
  A novel transfer support matrix machine for motor imagery-based brain
  computer interface.
\newblock Front. Neurosci. \textbf{14}, 606949 (2020)

\bibitem{clarke1990optimization}
Clarke, F.H.: Optimization and {N}onsmooth {A}nalysis.
\newblock John Wiley and Sons, New York (1983)

\bibitem{cortes1995support}
Cortes, C., Vapnik, V.: Support-vector networks.
\newblock Mach. Learn. \textbf{20}, 273--297 (1995)

\bibitem{cui2017quadratic}
Cui, Y., Ding, C., Zhao, X.: Quadratic growth conditions for convex matrix
  optimization problems associated with spectral functions.
\newblock SIAM J. Optim. \textbf{27}(4), 2332--2355 (2017)

\bibitem{cui2016asymptotic}
Cui, Y., Sun, D.F., Toh, K.C.: On the asymptotic superlinear convergence of the
  augmented {L}agrangian method for semidefinite programming with multiple
  solutions.
\newblock arXiv preprint
  \href{https://arxiv.org/abs/1610.00875}{\color{blue}arXiv:1610.00875}  (2016)

\bibitem{cui2019r}
Cui, Y., Sun, D.F., Toh, K.C.: On the {R}-superlinear convergence of the {KKT}
  residuals generated by the augmented {L}agrangian method for convex composite
  conic programming.
\newblock Math. Program. \textbf{178}(1), 381--415 (2019)

\bibitem{duan2017quantum}
Duan, B., Yuan, J., Liu, Y., Li, D.: Quantum algorithm for support matrix
  machines.
\newblock Phys. Rev. A \textbf{96}(3), 032301 (2017)

\bibitem{FP2007}
Facchinei, F., Pang, J.S.: Finite-{D}imensional {V}ariational {I}nequalities
  and {C}omplementarity {P}roblems.
\newblock Springer Science \& Business Media, New York (2007)

\bibitem{feng2022support}
Feng, R., Xu, Y.: Support matrix machine with pinball loss for classification.
\newblock Neural Comput. Appl. \textbf{34}, 18643--18661 (2022)

\bibitem{feng2022subspace}
Feng, R., Zhong, P., Xu, Y.: A subspace elimination strategy for accelerating
  support matrix machine.
\newblock Pac. J. Optim. \textbf{18}(1), 155--176 (2022)

\bibitem{geng2023fault}
Geng, M., Xu, Z., Mei, M.: Fault diagnosis method for railway turnout with
  pinball loss-based multiclass support matrix machine.
\newblock Appl. Sci. \textbf{13}(22), 12375 (2023)

\bibitem{goldstein2014fast}
Goldstein, T., O'Donoghue, B., Setzer, S., Baraniuk, R.: Fast alternating
  direction optimization methods.
\newblock SIAM J. Imaging Sci. \textbf{7}(3), 1588--1623 (2014)

\bibitem{hang2020deep}
Hang, W., Feng, W., Liang, S., Wang, Q., Liu, X., Choi, K.S.: Deep stacked
  support matrix machine based representation learning for motor imagery eeg
  classification.
\newblock Comput. Meth. Programs Biomed. \textbf{193}, 105466 (2020)

\bibitem{hang2023deep}
Hang, W., Li, Z., Yin, M., Liang, S., Shen, H., Wang, Q., Qin, J., Choi, K.S.:
  Deep stacked least square support matrix machine with adaptive multi-layer
  transfer for {EEG} classification.
\newblock Biomed. Signal Process. Control \textbf{82}, 104579 (2023)

\bibitem{harrow2009quantum}
Harrow, A.W., Hassidim, A., Lloyd, S.: Quantum algorithm for linear systems of
  equations.
\newblock Phys. Rev. Lett. \textbf{103}(15), 150502 (2009)

\bibitem{hastie2009elements}
Hastie, T., Tibshirani, R., Friedman, J.H., Friedman, J.H.: The {E}lements of
  {S}tatistical {L}earning: {D}ata {M}ining, {I}nference, and {P}rediction,
  vol.~2.
\newblock Springer (2009)

\bibitem{hsieh2014nuclear}
Hsieh, C.J., Olsen, P.: Nuclear norm minimization via active subspace
  selection.
\newblock In: International Conference on Machine Learning, pp. 575--583. PMLR
  (2014)

\bibitem{kaifeng2011algorithms}
Jiang, K.: Algorithms for {L}arge {S}cale {N}uclear {N}orm {M}inimization and
  {C}onvex {Q}uadratic {S}emidefinite {P}rogramming {P}roblems.
\newblock Ph.D. thesis (2011)

\bibitem{jiang2013solving}
Jiang, K., Sun, D.F., Toh, K.C.: Solving nuclear norm regularized and
  semidefinite matrix least squares problems with linear equality constraints.
\newblock In: Discrete Geometry and Optimization, pp. 133--162. Springer (2013)

\bibitem{kumari2024support}
Kumari, A., Akhtar, M., Shah, R., Tanveer, M.: Support matrix machine: A
  review.
\newblock Neural Networks (2024).
  \href{https://doi.org/10.1016/j.neunet.2024.106767}
  {https://doi.org/10.1016/j.neunet.2024.106767}

\bibitem{el2012safe}
Laurent, E.G., Vivian, V., Tarek, R.: Safe feature elimination in sparse
  supervised learning.
\newblock Pac. J. Optim. \textbf{8}(4), 667--698 (2012)

\bibitem{li2024support}
Li, H., Xu, Y.: Support matrix machine with truncated pinball loss for
  classification.
\newblock Appl. Soft. Comput. \textbf{154}, 111311 (2024)

\bibitem{li2023mars}
Li, Q., Jiang, B., Sun, D.F.: {MARS}: A second-order reduction algorithm for
  high-dimensional sparse precision matrices estimation.
\newblock J. Mach. Learn. Res. \textbf{24}, 1--44 (2023)

\bibitem{li2022fusion}
Li, X., Cheng, J., Shao, H., Liu, K., Cai, B.: A fusion cwsmm-based framework
  for rotating machinery fault diagnosis under strong interference and
  imbalanced case.
\newblock IEEE Trans. Industr. Inform. \textbf{18}(8), 5180--5189 (2022)

\bibitem{li2021fusion}
Li, X., Cheng, J., Shao, H., Liu, K., Cai, B.: A fusion {CWSMM}-based framework
  for rotating machinery fault diagnosis under strong interference and
  imbalanced case.
\newblock IEEE Trans. Ind. Inform. \textbf{18}(8), 5180--5189 (2022)

\bibitem{li2024dynamics}
Li, X., Li, S., Wei, D., Si, L., Yu, K., Yan, K.: Dynamics simulation-driven
  fault diagnosis of rolling bearings using security transfer support matrix
  machine.
\newblock Reliab. Eng. Syst. Saf. \textbf{243}, 109882 (2024)

\bibitem{li2023intelligent}
Li, X., Li, Y., Yan, K., Shao, H., Lin, J.J.: Intelligent fault diagnosis of
  bevel gearboxes using semi-supervised probability support matrix machine and
  infrared imaging.
\newblock Reliab. Eng. Syst. Saf. \textbf{230}, 108921 (2023)

\bibitem{li2024intelligent}
Li, X., Li, Y., Yan, K., Si, L., Shao, H.: An intelligent fault detection
  method of industrial gearboxes with robustness one-class support matrix
  machine toward multisource nonideal data.
\newblock IEEE ASME Trans. Mechatron. \textbf{29}(1), 388--399 (2024)

\bibitem{li2022highly}
Li, X., Shao, H., Lu, S., Xiang, J., Cai, B.: Highly efficient fault diagnosis
  of rotating machinery under time-varying speeds using {LSISMM} and small
  infrared thermal images.
\newblock IEEE Trans. Syst. Man Cybern. Syst. \textbf{52}(12), 7328--7340
  (2022)

\bibitem{li2018qsdpnal}
Li, X., Sun, D.F., Toh, K.C.: {QSDPNAL}: {A} two-phase augmented {L}agrangian
  method for convex quadratic semidefinite programming.
\newblock Math. Program. Comput. \textbf{10}, 703--743 (2018)

\bibitem{li2020non}
Li, X., Yang, Y., Pan, H., Cheng, J., Cheng, J.: Non-parallel least squares
  support matrix machine for rolling bearing fault diagnosis.
\newblock Mech. Mach. Theory \textbf{145}, 103676 (2020)

\bibitem{li2022auto}
Li, Y., Wang, D., Liu, F.: The auto-correlation function aided sparse support
  matrix machine for {EEG}-based fatigue detection.
\newblock IEEE Trans. Circuits Syst. II-Express Briefs \textbf{70}(2), 836--840
  (2023)

\bibitem{liang2024adaptive}
Liang, S., Hang, W., Lei, B., Wang, J., Qin, J., Choi, K.S., Zhang, Y.:
  Adaptive multimodel knowledge transfer matrix machine for {EEG}
  classification.
\newblock IEEE Trans. Neural Netw. Learn. Syst. \textbf{35}(6), 7726--7739
  (2024)

\bibitem{liang2022deep}
Liang, S., Hang, W., Yin, M., Shen, H., Wang, Q., Qin, J., Choi, K.S., Zhang,
  Y.: Deep {EEG} feature learning via stacking common spatial pattern and
  support matrix machine.
\newblock Biomed. Signal Process. Control \textbf{74}, 103531 (2022)

\bibitem{lin2020adaptive}
Lin, M., Yuan, Y., Sun, D.F., Toh, K.C.: Adaptive sieving with {PPDNA}:
  {G}enerating solution paths of exclusive lasso models.
\newblock arXiv preprint
  \href{https://arxiv.org/abs/2009.08719}{\color{blue}arXiv:2009.08719}  (2020)

\bibitem{liu2012implementable}
Liu, Y.J., Sun, D., Toh, K.C.: An implementable proximal point algorithmic
  framework for nuclear norm minimization.
\newblock Math. Program. \textbf{133}, 399--436 (2012)

\bibitem{luo2015support}
Luo, L., Xie, Y., Zhang, Z., Li, W.J.: Support matrix machines.
\newblock In: International Conference on Machine Learning, pp. 938--947. PMLR
  (2015)

\bibitem{mangasarian1988simple}
Mangasarian, O.L.: A simple characterization of solution sets of convex
  programs.
\newblock Oper. Res. Lett. \textbf{7}, 21--26 (1988)

\bibitem{meng2005semismoothness}
Meng, F., Sun, D.F., Zhao, G.: Semismoothness of solutions to generalized
  equations and the {M}oreau-{Y}osida regularization.
\newblock Math. Program. \textbf{104}, 561--581 (2005)

\bibitem{ogawa2013safe}
Ogawa, K., Suzuki, Y., Takeuchi, I.: Safe screening of non-support vectors in
  pathwise {SVM} computation.
\newblock In: International Conference on Machine Learning, pp. 1382--1390.
  PMLR (2013)

\bibitem{pan2022pinball}
Pan, H., Sheng, L., Xu, H., Tong, J., Zheng, J., Liu, Q.: Pinball transfer
  support matrix machine for roller bearing fault diagnosis under limited
  annotation data.
\newblock Appl. Soft. Comput. \textbf{125}, 109209 (2022)

\bibitem{pan2023deep}
Pan, H., Sheng, L., Xu, H., Zheng, J., Tong, J., Niu, L.: Deep stacked pinball
  transfer matrix machine with its application in roller bearing fault
  diagnosis.
\newblock Eng. Appl. Artif. Intell. \textbf{121}, 105991 (2023)

\bibitem{pan2024semi}
Pan, H., Xu, H., Zheng, J., Shao, H., Tong, J.: A semi-supervised matrixized
  graph embedding machine for roller bearing fault diagnosis under few-labeled
  samples.
\newblock IEEE Trans. Industr. Inform. \textbf{20}(1), 854--863 (2024)

\bibitem{pan2022multi}
Pan, H., Xu, H., Zheng, J., Su, J., Tong, J.: Multi-class fuzzy support matrix
  machine for classification in roller bearing fault diagnosis.
\newblock Adv. Eng. Inform. \textbf{51}, 101445 (2022)

\bibitem{pan2019fault}
Pan, H., Yang, Y., Zheng, J., Li, X., Cheng, J.: A fault diagnosis approach for
  roller bearing based on symplectic geometry matrix machine.
\newblock Mech. Mach. Theory \textbf{140}, 31--43 (2019)

\bibitem{pan2019novel}
Pan, X., Xu, Y.: A novel and safe two-stage screening method for support vector
  machine.
\newblock IEEE Trans. Neural Netw. Learn. Syst. \textbf{30}(8), 2263--2274
  (2019)

\bibitem{qi2006quadratically}
Qi, H., Sun, D.F.: A quadratically convergent {N}ewton method for computing the
  nearest correlation matrix.
\newblock SIAM J. Matrix Anal. Appl. \textbf{28}(2), 360--385 (2006)

\bibitem{qi1993nonsmooth}
Qi, L., Sun, J.: A nonsmooth version of {N}ewton's method.
\newblock Math. Program. \textbf{58}, 353--367 (1993)

\bibitem{qian2019robust}
Qian, C., Tran-Dinh, Q., Fu, S., Zou, C., Liu, Y.: Robust multicategory support
  matrix machines.
\newblock Math. Program. \textbf{176}, 429--463 (2019)

\bibitem{razzak2019multiclass}
Razzak, I., Blumenstein, M., Xu, G.: Multiclass support matrix machines by
  maximizing the inter-class margin for single trial {EEG} classification.
\newblock IEEE Trans. Neural Syst. Rehabilitation Eng. \textbf{27}(6),
  1117--1127 (2019)

\bibitem{robinson1981some}
Robinson, S.M.: Some continuity properties of polyhedral multifunctions.
\newblock in Mathematical Programming at Oberwolfach, Math. Program. Stud. pp.
  206--214 (1981)

\bibitem{rockafellar1970convex}
Rockafellar, R.T.: Convex {A}nalysis.
\newblock Princeton University Press, Princeton (1970)

\bibitem{rockafellar1976augmented}
Rockafellar, R.T.: Augmented {L}agrangians and applications of the proximal
  point algorithm in convex programming.
\newblock Math. Oper. Res. \textbf{1}(2), 97--116 (1976)

\bibitem{Rockafellar1976}
Rockafellar, R.T.: Augmented {Lagrangians} and applications of the proximal
  point algorithm in convex programming.
\newblock Math. Oper. Res. \textbf{1}(2), 97--116 (1976)

\bibitem{rockafellar2009variational}
Rockafellar, R.T., Wets, R.J.B.: Variational {A}nalysis, vol. 317.
\newblock Springer Science \& Business Media (2009)

\bibitem{sanei2013eeg}
Sanei, S., Chambers, J.A.: EEG {S}ignal {P}rocessing.
\newblock John Wiley \& Sons (2013)

\bibitem{shapiro2003sensitivity}
Shapiro, A.: Sensitivity analysis of generalized equations.
\newblock J. Math. Sci. \textbf{115}(4), 2554--2565 (2003)

\bibitem{wang2014scaling}
Wang, J., Wonka, P., Ye, J.: Scaling {SVM} and least absolute deviations via
  exact data reduction.
\newblock In: International Conference on Machine Learning, pp. 523--531. PMLR
  (2014)

\bibitem{watson1992characterization}
Watson, G.A.: Characterization of the subdifferential of some matrix norms.
\newblock Linear Alg. Appl. \textbf{170}, 33--45 (1992)

\bibitem{wu2023convex}
Wu, C., Cui, Y., Li, D., Sun, D.F.: Convex and nonconvex risk-based linear
  regression at scale.
\newblock INFORMS J. Comput. \textbf{35}(4), 797--816 (2023)

\bibitem{xu2024intelligent}
Xu, H., Pan, H., Zheng, J., Tong, J., Zhang, F., Chu, F.: Intelligent fault
  identification in sample imbalance scenarios using robust low-rank matrix
  classifier with fuzzy weighting factor.
\newblock Appl. Soft. Comput. \textbf{152}, 111229 (2024)

\bibitem{yang2004two}
Yang, J., Zhang, D., Frangi, A.F., Yang, J.y.: Two-dimensional {PCA}: a new
  approach to appearance-based face representation and recognition.
\newblock IEEE Trans. Pattern Anal. Mach. Intell. \textbf{26}(1), 131--137
  (2004)

\bibitem{yuan2022dimension}
Yuan, Y., Chang, T.H., Sun, D.F., Toh, K.C.: A dimension reduction technique
  for large-scale structured sparse optimization problems with application to
  convex clustering.
\newblock SIAM J. Optim. \textbf{32}(3), 2294--2318 (2022)

\bibitem{yuan2023adaptive}
Yuan, Y., Lin, M., Sun, D.F., Toh, K.C.: Adaptive sieving: A dimension
  reduction technique for sparse optimization problems.
\newblock arXiv preprint
  \href{https://arxiv.org/abs/2306.17369}{\color{blue}arXiv:2306.17369}  (2023)

\bibitem{zhang2022proximal}
Zhang, W., Liu, Y.: Proximal support matrix machine.
\newblock J. appl. math. phys \textbf{10}(7), 2268--2291 (2022)

\bibitem{ZST2010}
Zhao, X.Y., Sun, D.F., Toh, K.C.: A {Newton-CG} augmented {Lagrangian} method
  for semidefinite programming.
\newblock SIAM J. Optim. \textbf{20}(4), 1737--1765 (2010)

\bibitem{zheng2018robust}
Zheng, Q., Zhu, F., Heng, P.A.: Robust support matrix machine for single trial
  {EEG} classification.
\newblock IEEE Trans. Neural Syst. Rehabilitation Eng. \textbf{26}(3), 551--562
  (2018)

\bibitem{zheng2018sparse}
Zheng, Q., Zhu, F., Qin, J., Chen, B., Heng, P.A.: Sparse support matrix
  machine.
\newblock Pattern Recognit. \textbf{76}, 715--726 (2018)

\bibitem{zheng2018multiclass}
Zheng, Q., Zhu, F., Qin, J., Heng, P.A.: Multiclass support matrix machine for
  single trial {EEG} classification.
\newblock Neurocomputing \textbf{275}, 869--880 (2018)

\bibitem{zhou2017unified}
Zhou, Z., So, A.M.C.: A unified approach to error bounds for structured convex
  optimization problems.
\newblock Math. Program. \textbf{165}, 689--728 (2017)

\bibitem{zimmert2015safe}
Zimmert, J., de~Witt, C.S., Kerg, G., Kloft, M.: Safe screening for support
  vector machines.
\newblock In: NIPS Workshop on Optimization in Machine Learning (OPT) (2015)

\end{thebibliography}


\end{document}